\documentclass[reqno,a4paper,11pt]{amsart}
\usepackage{graphicx}
\usepackage{amsmath}
\usepackage{amsthm}
\usepackage{amsfonts}
\usepackage{amssymb}
\usepackage{enumerate}
\usepackage{latexsym}
\usepackage{multirow}
\usepackage{graphpap}
\usepackage{mathrsfs}
\usepackage{eucal}

\newtheorem{theorem}{Theorem}[section]
\newtheorem{corollary}[theorem]{Corollary}
\newtheorem{lemma}[theorem]{Lemma}
\newtheorem{proposition}[theorem]{Proposition}

\newtheorem{conA}{Condition~A\!\!}
\newtheorem{conB}{Condition~B\!\!}

\theoremstyle{definition}

\newtheorem{example}[theorem]{Example}

\newtheorem{question}[theorem]{Question}

\newtheorem{remark}[theorem]{Remark}

\theoremstyle{remark}
\numberwithin{equation}{section}

\newcommand{\id}{\approx}
\newcommand{\lbr}{\big[\hspace{-0.044in}\big[}      
\newcommand{\lbrs}{[\hspace{-0.024in}[}             
\newcommand{\rbr}{\big]\hspace{-0.044in}\big]}      
\newcommand{\rbrs}{]\hspace{-0.024in}]}             
\newcommand{\lpbr}{\langle\hspace{-0.037in}\langle} 
\newcommand{\rpbr}{\rangle\hspace{-0.037in}\rangle} 
\newcommand{\lobr}{(}                               
\newcommand{\robr}{)}                               
\newcommand{\til}{\widetilde}
\newcommand{\ov}{\overline}
\newcommand{\wh}{\widehat}
\newcommand{\ba}{\mathsf{bar}}
\newcommand{\occ}{\operatorname{\mathsf{occ}}}
\newcommand{\ini}{\operatorname{\mathsf{ini}}}
\newcommand{\fin}{\operatorname{\mathsf{fin}}}
\newcommand{\cont}{\operatorname{\mathsf{con}}}
\newcommand{\excl}{\operatorname{\mathsf{Excl}}}

\newcommand{\op}{\mathsf{op}}
\newcommand{\malce}{\,{\mathbin{\hbox{$\bigcirc$\rlap{\kern-8.5pt\raise0,50pt\hbox{${\tt m}$}}}}}\,}

\newcommand{\At}{A_2}
\newcommand{\Az}{A_0}
\newcommand{\Bt}{B_2}
\newcommand{\Bz}{B_0}
\newcommand{\el}{\ell_3}
\newcommand{\elop}{\ell_3^\op}
\newcommand{\LZ}{L_2}
\newcommand{\N}{N}
\newcommand{\RZ}{R_2}
\newcommand{\SL}{Sl_2}
\newcommand{\W}{W}
\newcommand{\Z}{\mathbb{Z}}
\newcommand{\OO}{O}
\newcommand{\Oe}{e}
\newcommand{\Ox}{x}
\newcommand{\eee}{\mathsf}
\newcommand{\ea}{\eee{a}}
\newcommand{\eb}{\eee{b}}
\newcommand{\ee}{\eee{e}}
\newcommand{\eE}{\eee{E}}
\newcommand{\ef}{\eee{f}}
\newcommand{\eg}{\eee{g}}
\newcommand{\eX}{\eee{X}}
\newcommand{\eY}{\eee{Y}}
\newcommand{\ei}{1}
\newcommand{\ez}{0}
\newcommand{\ZB}{\mathbb{Z}_2^\ba}                      
\newcommand{\ZBop}{\lobr\mathbb{Z}_2^\ba\robr^\op}      
\newcommand{\NB}{\N_2^\ba}                               
\newcommand{\NBop}{\lobr \N_2^\ba\robr^\op}              
\newcommand{\NBI}{\lobr \N_2^\ba\robr^I}                 
\newcommand{\NBIop}{\lobr\lobr \N_2^\ba\robr^I\robr^\op} 
\newcommand{\LZB}{\LZ^\ba}                              
\newcommand{\LZBop}{\lobr\LZ^\ba\robr^\op}              
\newcommand{\elB}{\el^\ba}                              
\newcommand{\elBop}{\lobr\el^\ba\robr^\op}              
\newcommand{\A}{\mathscr{A}} 
\newcommand{\be}{\mathbf{e}}

\newcommand{\bu}{\mathbf{u}}
\newcommand{\bv}{\mathbf{v}}
\newcommand{\bw}{\mathbf{w}}
\newcommand{\bA}{\mathbf{A}}
\newcommand{\bB}{\mathbf{B}}
\newcommand{\bCom}{\mathbf{Com}}
\newcommand{\bG}{\mathbf{G}}
\newcommand{\bLNB}{\mathbf{LNB}}
\newcommand{\bLZ}{\mathbf{LZ}}
\newcommand{\bPV}{\mathbf{PV}}
\newcommand{\bRZ}{\mathbf{RZ}}
\newcommand{\bSl}{\mathbf{Sl}}
\newcommand{\bU}{\mathbf{U}}
\newcommand{\bV}{\mathbf{V}}
\newcommand{\bW}{\mathbf{W}}
\newcommand{\latF}{\mathscr{F}}
\newcommand{\latL}{\mathscr{L}}
\newcommand{\latX}{\mathscr{X}}
\newcommand{\irr}{irre\-duci\-ble}
\newcommand{\jirr}{join irre\-duci\-bility}
\newcommand{\mog}{minimal order generator}
\newcommand{\pid}{pseu\-do\-identity}
\newcommand{\pids}{pseu\-do\-iden\-ti\-ties}
\newcommand{\pvar}{pseu\-do\-va\-ri\-ety}
\newcommand{\pvars}{pseu\-do\-va\-rie\-ties}
\newcommand{\ji}{\text{$\mathsf{ji}$}}
\newcommand{\fji}{\text{$\mathsf{fji}$}}
\newcommand{\sji}{\text{$\mathsf{sji}$}}
\newcommand{\sfji}{\text{$\mathsf{sfji}$}}
\newcommand{\mi}{\text{$\mathsf{mi}$}}
\newcommand{\fmi}{\text{$\mathsf{fmi}$}}
\newcommand{\sdi}{\text{$\mathsf{sdi}$}}
\newcommand{\smi}{\text{$\mathsf{smi}$}}
\newcommand{\sfmi}{\text{$\mathsf{sfmi}$}}
\newcommand{\xp}{$\times$-prime}
\newcommand{\KN}{\text{$\mathsf{KN}$}}

\newcommand{\up}{\textup}

\allowdisplaybreaks

\makeatletter
\def\@tocline#1#2#3#4#5#6#7{\relax
  \ifnum #1>\c@tocdepth 
  \else
    \par \addpenalty\@secpenalty\addvspace{#2}%
    \begingroup \hyphenpenalty\@M
    \@ifempty{#4}{%
      \@tempdima\csname r@tocindent\number#1\endcsname\relax
    }{%
      \@tempdima#4\relax
    }%
    \parindent\z@ \leftskip#3\relax \advance\leftskip\@tempdima\relax
    \rightskip\@pnumwidth plus4em \parfillskip-\@pnumwidth
    #5\leavevmode\hskip-\@tempdima
      \ifcase #1
       \or\or \hskip 1em \or \hskip 2em \else \hskip 3em \fi%
      #6\nobreak\relax
    \dotfill\hbox to\@pnumwidth{\@tocpagenum{#7}}\par
    \nobreak
    \endgroup
  \fi}
\makeatother

\begin{document}

\title{Join irreducible semigroups}

\author[E. W. H. Lee]{Edmond W. H. Lee}

\address{Department of Mathematics, Nova Southeastern University, 3301 College Avenue, Fort Lauderdale, FL~33314, USA}
\email{edmond.lee@nova.edu}

\author[J. Rhodes]{John Rhodes}

\address{Department of Mathematics, University of California, Berkeley, 970 Evans Hall \#3840, Berkeley, CA~94720, USA}
\email{rhodes@math.berkeley.edu}
\email{blvdbastille@gmail.com}
 
\author[B. Steinberg]{Benjamin Steinberg}

\address{Department of Mathematics, City College of New York, NAC 8/133, Convent Avenue at 138th Street, New York, NY 10031, USA}
\email{bsteinberg@ccny.cuny.edu}

\thanks{The second author was supported by Simons Foundation Collaboration Grants for Mathematicians \#313548.
The third author was supported by Simons Foundation \#245268, United States--Israel Binational Science Foundation \#2012080, and NSA MSP \#H98230-16-1-0047.}


\begin{abstract}
We begin a systematic study of finite semigroups that generate join irreducible members of the lattice of pseudovarieties of finite semigroups, which are important for the spectral theory of this lattice.
Finite semigroups~$S$ that generate join irreducible pseudovarieties are characterized as follows: whenever~$S$ divides a direct product $A \times B$ of finite semigroups, then~$S$ divides either~$A^n$ or~$B^n$ for some $n \geq 1$.
We present a new operator ${ \mathbf{V} \mapsto \mathbf{V}^\mathsf{bar} }$ that preserves the property of join irreducibility, as does the dual operator, and show that iteration of these operators on any nontrivial join irreducible pseudovariety leads to an infinite hierarchy of join irreducible pseudovarieties.
We also describe all join irreducible pseudovarieties generated by a semigroup of order up to five.
It turns out that there are~$30$ such pseudovarieties, and there is a relatively easy way to remember them.
In addition, we survey most results known about join irreducible pseudovarieties to date and generalize a number of results in Sec.~7.3 of [\textit{The $q$-theory of Finite Semigroups}, Springer Monographs in Mathematics (Springer, Berlin, 2009)].
\end{abstract}

\dedicatory{Dedicated to the 80th birthday of Norman Reilly on 30 Jan 2020 \\ and the 65th birthday of Mikhail Volkov on 27 May 2020}

\maketitle

\setcounter{tocdepth}{2} \tableofcontents

\section{Introduction}

In the 1970s, Eilenberg~\cite{Eil76} highlighted the importance of~$\bPV$, the algebraic lattice of all {\pvars} of finite semigroups, via his research with Sch\"{u}tzenberger, by providing a correspondence between~$\bPV$ and varieties of regular languages.
Specifically, they proved that the lattice~$\bPV$ is isomorphic to the algebraic lattice of varieties of regular languages; see the monograph by the second and third authors~\cite[Introduction]{RS09} and the references therein.

The $\mathfrak{q}$-theory of finite semigroups focuses on~$\bPV$, but in a different manner, and can be viewed in analogy with the classical real analysis theory of continuous and differentiable functions from $[0,1]$ into $[0,1]$.
The analogy is given by replacing $[0,1]$ with~$\bPV$, continuous functions with $\mathbf{Cnt}(\bPV)$, and differentiable functions with $\mathbf{GMC}(\bPV)$; see \cite[Chapter~2]{RS09}.

From a number of points of view, $\bPV$ is an important algebraic lattice with many interesting properties, and several theories have been developed for its investigation.
For instance, the theorem of Reiterman~\cite{Rei82} characterized {\pvars} as exactly the classes defined by {\pids}.
This led to the syntactic approach---employed by Almeida in his work and monograph~\cite{Alm94}---that has became a fundamental tool in finite semigroup theory.
Some of these results and techniques will be employed in this paper.
Another important approach is the abstract spectral theory of~$\bPV$ going back to Stone with lattice theoretic foundations going back to Birkhoff; see \cite[Chapter~7]{RS09}.

Since~$\bPV$ is a lattice, it is natural to investigate its elements that satisfy important lattice properties.
For any element~$\ell$ in a lattice~$\latL$,
\begin{enumerate}[\ (1)]
\item $\ell$ is \textit{compact} if, for any $\latX \subseteq \latL$, \[ \ell \leq \bigvee \latX \quad \Longrightarrow \quad \ell \leq \bigvee \latF \text{ for some finite } \latF \subseteq \latX; \]
\item $\ell$ is \textit{join \irr} (\ji) if, for any $\latX \subseteq \latL$, \[ \ell \leq \bigvee \latX \quad \Longrightarrow \quad \ell \leq x \text{ for some } x \in \latX; \]
\item $\ell$ is \textit{finite join \irr} (\fji) if, for any finite $\latF \subseteq \latL$, \[ \ell \leq \bigvee \latF \quad \Longrightarrow \quad \ell \leq x \text{ for some } x \in \latF; \]
\item $\ell$ is \textit{meet \irr} (\mi) if, for any set $\latX \subseteq \latL$, \[ \ell \geq \bigwedge \latX \quad \Longrightarrow \quad \ell \geq x \text{ for some } x \in \latX; \]
\item $\ell$ is \textit{finite meet \irr} (\fmi) if, for any finite $\latF \subseteq \latL$, \[ \ell \geq \bigwedge \latF \quad \Longrightarrow \quad \ell \geq x \text{ for some } x \in \latF; \]
\item $\ell$ is \textit{strictly join \irr} (\sji) if, for any set $\latX \subseteq \latL$, \[ \ell = \bigvee \latX \quad \Longrightarrow \quad \ell \in \latX; \]
\item $\ell$ is \textit{strictly finite join \irr} (\sfji) if, for any finite $\latF \subseteq \latL$, \[ \ell = \bigvee \latF \quad \Longrightarrow \quad \ell \in \latF; \]
\item $\ell$ is \textit{strictly meet \irr} (\smi) if, for any $\latX \subseteq \latL$, \[ \ell = \bigwedge \latX \quad \Longrightarrow \quad \ell \in \latX; \]
\item $\ell$ is \textit{strictly finite meet \irr} (\sfmi) if, for any finite $\latF \subseteq \latL$, \[ \ell = \bigwedge \latF \quad \Longrightarrow \quad \ell \in \latF. \]
\end{enumerate}
An \textit{algebraic lattice} is a complete lattice that is join generated by its compact elements.
The compact elements of~$\bPV$ are the finitely generated {\pvars}.
The {\pvar} generated by a finite semigroup~$S$ is denoted by $\lpbr S \rpbr$.
It is clear that for any $\bV \in \bPV$, \[ \bV = \bigvee \{ \lpbr S\rpbr \mid S \in \bV \}. \]

The abstract spectral theory of a lattice is closely connected to the computation of its maximal distributive image, which is determined by the lattice's {\fji} and {\fmi} elements; see \cite[Chapter~7]{RS09} and the references therein.
The {\fji} and {\fmi} elements of~$\bPV$ are thus very important.
The {\ji} {\pvars} are just the compact {\fji} {\pvars}, as is easy to see, so we are interested in finite semigroups that generate {\pvars} that are {\fji} or equivalently {\ji}.

By abuse of terminology, we say that a finite semigroup~$S$ is \textit{join \irr} (\ji) if the {\pvar} $\lpbr S\rpbr$ is~{\ji}; finite semigroups that satisfy the properties in (3)--(9) are similarly defined.
A finite semigroup~$S$ is~{\ji} if and only if for all finite semigroups~$T_1$ and~$T_2$,
\[
S \prec T_1 \times T_2 \quad \Longrightarrow \quad S \prec T_1^n \text{ or } S \prec T_2^n \text{ for some } n \geq 1,
\]
where $A \prec B$ means that~$A$ is a homomorphic image of a subsemigroup of~$B$, and $A^n = A \times A \times \cdots \times A$ is the direct product of~$n$ copies of~$A$.
For finite semigroups, there are several properties stronger than being~{\ji}: a finite semigroup~$S$ is \textit{\xp} \cite[Section~9.3]{Alm94} if for all finite semigroups~$T_1$ and~$T_2$,
\[
S \prec T_1 \times T_2 \quad \Longrightarrow \quad S \prec T_1 \text{ or } S \prec T_2;
\]
a semigroup~$S$ is \textit{Kov\'{a}cs--Newman} (\KN) if whenever $f : T \twoheadrightarrow S$ is a surjective homomorphism where~$T$ is a subsemigroup of $T_1 \times T_2$ for some finite semigroups~$T_1$ and~$T_2$, subdirectly embedded, then~$f$ factors through one of the projections.
Semigroups that are {\KN} have been completely classified~\cite[Section~7.4]{RS09}.

The proper inclusions
\[
\{\text{{\KN} semigroups}\} \subsetneqq \{\text{{\xp} semigroups}\} \subsetneqq \{\text{{\ji} semigroups}\}
\]
are known to hold.
For example, while any simple non-abelian group  is~{\KN}, any cyclic group~$\Z_p$ of prime order~$p$ is~{\xp} but not~{\KN}.
The well-known Brandt semigroup~$\Bt$ of order five is~{\ji} but not~{\xp} \cite[Example~7.4.3]{RS09}.

Since the lattice~$\bPV$ is algebraic, it follows from a well-known theorem of Birkhoff that its~{\smi} elements constitute the unique minimal set of meet generators~\cite[Section~7.1]{RS09}.
It easily follows from Reiterman's theorem~\cite[Section~3.2]{RS09} that each~{\smi} {\pvar} is defined by a single {\pid} but not conversely.
Now the reverse of the lattice~$\bPV$ is not algebraic but is locally dually algebraic, so the {\sji} elements
of $\bPV$ constitute the unique minimal set of join generators for~$\bPV$~\cite[Section~7.2]{RS09}.
The {\sji} {\pvars} are precisely those having a unique proper maximal sub\-{\pvar}.

Every~{\ji} {\pvar} is~{\sji}, but the converse does not hold, as demonstrated by several known examples~\cite[Proposition~7.3.22]{RS09} and additional examples in Propositions~\ref{P: Bz not ji} and~\ref{P: W not ji}.
Hence~{\ji} {\pvars} do not join generate the lattice~$\bPV$.
This prompts the following tantalizing question.

\begin{question} \label{Q: ji join generate}
What do the~{\ji} elements in~$\bPV$ join generate?
\end{question}

It is well known and not difficult to prove that the function
\[
S \mapsto \begin{cases} 1 & \text{if $\lpbr S\rpbr$ is \sji} \\ 0 & \text{otherwise} \end{cases}
\]
on the class of finite semigroups is computable; see, for example, Proposition~\ref{P: maximal} and its proof.
On the other hand, it is unknown if the function
\[
S \mapsto \begin{cases} 1 & \text{if $\lpbr S\rpbr$ is \ji} \\ 0 & \text{otherwise} \end{cases}
\]
on the class of finite semigroups is decidable.

\begin{question} \label{Q: ji decidable}
Is~{\ji} decidable, that is, is the above function computable?
\end{question}

If~{\ji} is not decidable, then a systematic study of {\ji} semigroups seems doomed in general.
But even if~{\ji} is decidable, then it is probably hopeless, in practice, to find all~{\ji} semigroups.
In any case, an important step is to find methods to produce new {\ji} semigroups and methods to identify and eliminate finite semigroups that are not~{\ji}.
This paper develops several new methods.
For semigroups of small order, in particular, the (Birkhoff) equational theory is crucial and is often used.

A pleasant feature of a finite semigroup~$S$ being~{\ji} is the ``five for one phenomenon" related to the \textit{exclusion class} $\excl(S)$ of~$S$, the class of all finite semigroups~$T$ for which $S \notin \lpbr T \rpbr$.
Indeed, a finite semigroup~$S$ is {\ji} if and only if $\excl(S)$ is a {\pvar} \cite[Theorem~7.1.2]{RS09}.
In this case, $\excl(S)$ is~{\mi} and so is defined by a single {\pid}, and since $\excl(S)$ is also~{\smi}, it has $\excl(S) \vee \lpbr S\rpbr$ as a unique cover.
Further, $\lpbr S\rpbr \cap \excl(S)$ is the unique maximal sub\-{\pvar} of~$\lpbr S\rpbr$, and so $\excl(S)$ determines~$\lpbr S\rpbr$; see \cite[Section~7.1]{RS09}.
For example, the Brandt semigroup~$\Bt$ is~{\ji}, the exclusion class $\excl(\Bt)$ coincides with the {\pvar}
\[ \mathbf{DS} = \lbr ((xy)^\omega(yx)^\omega(xy)^\omega)^\omega \id (xy)^\omega \rbr \]
of finite semigroups whose $\mathscr{J}$-classes are sub\-semigroups \cite[Example~7.3.4]{RS09}, and $\lpbr \Bt\rpbr \cap \mathbf{DS} = \lpbr \Bz\rpbr$ is the unique maximal sub\-{\pvar} of~$\lpbr \Bt\rpbr$, where~$\Bz$ is a sub\-semigroup of~$\Bt$ of order four~\cite{Lee04}; see Subsection~\ref{sub: comp 0-simple}.
More examples of maximal sub\-{\pvars} can be found in Section~\ref{sec: ji}.

As mentioned earlier, a goal of this paper is to find new {\ji} semigroups.
One approach---and a very important problem in its own right---is to find new operators on $\bPV$ that preserve the property of being~{\ji}.
The following are some known examples.

\begin{example} \label{E: Sdual}
For any semigroup~$S$, the \textit{opposite} semigroup~$S^\op$ of~$S$ is obtained by reversing the multiplication on~$S$.
Then the \textit{dual operator} \[\bV \mapsto \bV^\op = \{ S^\op \mid S \in \bV \}\] on $\bPV$ preserves the property of being {\ji}.
\end{example}

\begin{example}[See Lemma~\ref{L: local}] \label{E: SI}
For any semigroup~$S$, let~$S^I$ denote the monoid obtained by adjoining an external identity element~$I$ to~$S$, and define
\[
S^\bullet =
\begin{cases}
S & \text{if $S$ is a monoid}, \\
S^I & \text{otherwise}.
\end{cases}
\]
Then the operator $\bV \mapsto \bV^\bullet = \{ S^\bullet \mid S \in \bV \}$ on $\bPV$ preserves the property of being {\ji}.
\end{example}

Example~\ref{E: Sdual} is not surprising; in fact, in many investigations, such as the finite basis problem for small semigroups \cite{LLZ12,Tra83}, it is common to identify~$S^\op$ with~$S$.
The situation for the operator $\bV \mapsto \bV^\bullet$, however, can be different because it is possible that no new {\ji} {\pvar} is produced.
Indeed, if a {\pvar}~$\bV$ is generated by some monoid, then $\bV^\bullet = \bV$ cannot be a new example of {\ji} {\pvar}.
But if $\bV = \lpbr S \rpbr$ is a {\ji} {\pvar} that is not generated by any monoid, then $\bV^\bullet = \lpbr S^I\rpbr$ is a~{\ji} {\pvar} properly containing~$\bV$.
Note that the operator $\bV \mapsto \bV^I = \{ S^I \mid S \in \bV \}$ does not preserve the property of being {\ji}.
For example, the cyclic group $\Z_p$ of any prime order~$p$ generates a {\ji} {\pvar}, but the {\pvar} $\lpbr \Z_p \rpbr^I = \lpbr \Z_p^I \rpbr$ is not {\ji} because $\lpbr \Z_p^I \rpbr = \lpbr \Z_p \rpbr \vee \lpbr \SL \rpbr$, where $\SL$ is the semilattice of order two.

On the other hand, it is possible for~$\bV^I$ to be~{\ji} even though~$\bV$ is not~{\ji}.
For example, if $S = \SL \times \RZ$, where~$\RZ$ is the right zero semigroup of order two, then the {\pvar} $\lpbr S\rpbr = \lpbr\SL\rpbr \vee \lpbr\RZ\rpbr$ is not~{\ji} but $\lpbr S \rpbr^I = \lpbr \SL^I \times \RZ^I\rpbr = \lpbr\RZ^I\rpbr$ is~{\ji} \cite[Example~7.3.1]{RS09}.

\begin{remark}
It is clear that the operator $\bV \mapsto \bV^\op$ also preserves the property of being {\sji}, but the operator $\bV \mapsto \bV^\bullet$ does not preserve this property.
For instance, the {\pvar} $\lpbr \Bz \rpbr$ is {\sji} while $\lpbr \Bz \rpbr^\bullet = \lpbr \Bz^I \rpbr$ is not {\sji}; see Proposition~\ref{P: Bz not ji}.
\end{remark}

Given a finite semigroup~$S$, consider the right regular representation $(S^\bullet,S)$ of~$S$ acting on~$S^\bullet$ by right multiplication.
Then~$S^\ba$ is defined by adding all constant maps on~$S^\bullet$ to~$S$, where multiplication is composition with the variable written on the left.
Note that if $(S,S)$ is a faithful transformation semigroup, then we shall see later that the semigroup obtained from~$S$ by adjoining the constant mappings on~$S$ generates the same {\pvar} as~$S^\ba$ and hence we sometimes (abusively) denote this latter semigroup by~$S^\ba$ as well.
Some small examples of~$S^\ba$ can be found in Section~\ref{sec: req sgps}.

It turns out that the operator $\bV \mapsto \bV^\ba = \lpbr S^\ba \mid S \in \bV \rpbr$ on $\bPV$ preserves the property of being~{\ji}.
This result, the details of which are given in Subsection~\ref{sub: augmentation}, is important: for any finite nontrivial {\ji} semigroup~$S$, the {\pvars}
\[ \lpbr S^\ba \rpbr, \, \lpbr \lobr S^\ba\robr^\flat \rpbr, \, \lpbr \lobr\lobr S^\ba\robr^\flat\robr^\ba \rpbr, \, \lpbr \lobr\lobr\lobr S^\ba\robr^\flat\robr^\ba\robr^\flat \rpbr, \, \ldots, \]
where $X^\flat=\lobr\lobr X^\op\robr^\ba\robr^\op$, constitute an infinite increasing chain of~{\ji} {\pvars} (Corollary~\ref{C: hierarchy}) whose complete union is an~{\fji} {\pvar} that is not compact \cite[Chapter~7]{RS09}.

Unsurprisingly, irregularities do show up when the operator $\bV \mapsto \bV^\ba$ is applied.
For instance, it is sometimes possible for $\lpbr S^\ba\rpbr = \lpbr S\rpbr$, so that no new~{\ji} {\pvar} is obtained.
Further, it is possible for $\lpbr S^\ba\rpbr$ to be~{\ji} even though~$\lpbr S\rpbr$ is not~{\ji}.
An important class of examples will be given in Subsection~\ref{sub: Ok}.

A main result of this paper is the complete classification of all {\ji} {\pvars} generated by a semigroup of order up to five.
We want to give the reader an easy way to remember their generators.
First, we have the three operators $S \mapsto S^\op$, $S \mapsto S^\bullet$, and $S \mapsto S^\ba$, and their \textit{iterations} such as $\lobr\lobr\lobr S^\op \robr^\ba \robr^\op$ and $\lobr\lobr\lobr\lobr\lobr S^\ba \robr^\op \robr^\bullet \robr^\ba \robr^\op \robr^\ba$.
If we have a list of~{\ji} semigroups, applying the three operators and their iterations give~{\ji} semigroups that may or may not generate new~{\ji} {\pvars}.

A~{\ji} {\pvar}~$\bV$ is \textit{primitive} if $\bV \neq \lpbr S^\bullet\rpbr$ and $\bV \neq \lpbr S^\ba\rpbr$ for any finite semigroup~$S$ that generates a~{\ji} proper sub\-{\pvar} of~$\bV$.
Now we are only interested in knowing the primitive~{\ji} {\pvars} up to isomorphism and anti-isomorphism of members since the others can be obtained by applying the operators.
Therefore when describing~{\ji} {\pvars} generated by a semigroup of order up to five, it suffices to list, up to isomorphism and anti-isomorphism, only generators of those that are primitive; see Table~\ref{Tab: primitive ji}.
Presentations and multiplication tables of these semigroups can be found in Section~\ref{sec: req sgps}.
The only new example of a semigroup of order five that generates a primitive {\ji} {\pvar} is~$\elB$; all the other semigroups were previously known to be~{\ji}.
Note that~$\elB$ is {\ji} but~$\el$ is not; see Subsection~\ref{sub: Ok}, where this example is extended to an infinite family of examples.

\begin{table}[ht!]
\[\def\arraystretch{1.5}
\begin{tabular}[m]{ccclc} \hline
& \multirow{2}{*}{$n$} & & Semigroups of order~$n$ that generate & \\[-0.08in]
&                      & & primitive {\ji} {\pvars}              & \\\hline
& 2                    & & $\Z_2$, $\N_2$, $\LZ$                 & \\[-0.04in]
& 3                    & & $\Z_3$, $\N_3$                        & \\[-0.04in]
& 4                    & & $\Z_4$, $\N_4$, $\Az$                 & \\[-0.04in]
& 5                    & & $\Z_5$, $\N_5$, $\At$, $\Bt$, $\elB$  & \\[0.02in] \hline
\end{tabular}
\]
\caption{Some generators of primitive {\ji} {\pvars}} \label{Tab: primitive ji}
\end{table}

The statement of the above result regarding semigroups of order up to five is straightforward, but its proof is not so; it requires knowledge of sub\-{\pvars} of {\pvars} generated by small semigroups \cite{Lee04,Lee08,LL11,LL15,LV07,LZ15,Tis07,Tra94,ZL09} and of bases of {\pids} for many {\pvars} of the form $\bigvee_{i=1}^k \lpbr S_i \rpbr$, and advanced algebraic theory of finite semigroups~\cite{RS09}.

The following are all other~{\ji} semigroups known to us, except for some well-known results on completely simple semigroups.

\subsection{Groups} \label{sub: groups}

It is an easy observation that a finite group generates a {\ji} {\pvar} of semigroups if and only if it generates a {\ji} {\pvar} of groups; see \cite[Chapter~7]{RS09}.
A {\pvar}~$\bV$ of groups is called \emph{saturated} if whenever $\varphi: G\to H$ is a homomorphism of finite groups with $H\in \bV$, there exists a subgroup $K\leq G$ such that $K\in \bV$ and $K\varphi=H$.
It is observed in~\cite[Example~7.6.5]{RZ10} that any {\pvar} of groups closed under extension is saturated.
In particular, for any prime~$p$, the {\pvar} of $p$-groups is saturated.
It is almost immediate from the definition that if~$\bV$ is a saturated {\pvar} of groups, then a group $G\in \bV$ generates a {\ji} {\pvar} in the lattice of all semigroup {\pvars} if and only if it generates a {\ji} member of the lattice of sub{\pvars} of~$\bV$.
In particular, a $p$-group~$G$ is {\ji} if and only if whenever~$G$ divides a direct product $A\times B$ of $p$-groups, then~$G$ divides either~$A^n$ or~$B^n$ for some $n \geq 1$.

\subsubsection*{Abelian groups}

The following statements on any directly indecomposable finite abelian group~$A$ are equivalent: $A$ is~{\ji}, $A$ is {\xp}, and $A \cong \Z_{p^n}$ for some prime~$p$ and $n \geq 1$.
This result follows from the Fundamental Theorem of Finite Abelian Groups and that $\Z_{p^n}$ \textit{lifts} in the sense that whenever~$\Z_{p^n}$ is a homomorphic image of some semigroup~$S$, then~$\Z_{p^{n+r}}$ embeds into~$S$ for some $r \geq 0$.

\subsubsection*{Monolithic groups}

A finite group~$G$ is \textit{monolithic} if it contains a unique minimal nontrivial normal subgroup~$N$; in this case, $N$ is called the \textit{monolith} of~$G$, and it is well known that $N \cong H^n$ for some simple group~$H$ and $n \geq 1$.

A finite group is monolithic if and only if it is sub\-directly indecomposable;
recall that a semigroup~$S$ is a \textit{subdirect product} of~$S_1$ and~$S_2$, written $S \ll S_1 \times S_2$, if~$S$ is a subsemigroup of $S_1 \times S_2$ mapping onto both~$S_1$ and~$S_2$ via the projections~$\pi_i$.
A semigroup~$S$ is \textit{subdirectly indecomposable} (\sdi) if $S \ll S_1 \times S_2$ implies that at least one of the projections $\pi_i : S \twoheadrightarrow S_i$ is an isomorphism.
Therefore when locating {\ji} groups from among finite groups, it suffices to concentrate on those that are monolithic.

\subsubsection*{Groups with non-abelian monolith}

Kov\'{a}cs and Newman proved that any monolithic group with non-abelian monolith is {\KN} \cite[Section~7.4]{RS09} and so also {\xp} and~{\ji}.
Therefore, all simple non-abelian groups  are~{\ji}.

\subsubsection*{Groups with abelian monolith}

An abelian monolith~$N$ of a finite group~$G$ \textit{splits} if there exists a subgroup~$K$ of~$G$ so that $N \cap K = \{1\}$ and $NK = G$.
A finite subdirectly indecomposable group with an abelian monolith that splits is {\ji}; this result is due to G.\,M. Bergman and its proof is given in Subsection~\ref{sub: Bergman}.
Therefore, the symmetric group $\mathsf{Sym}_3$ over three symbols is~{\ji}.

\subsubsection*{Groups of small order}

The {\ji} {\pvars} generated by a group of order seven or less are $\lpbr \Z_2 \rpbr$, $\lpbr \Z_3 \rpbr$, $\lpbr \Z_4 \rpbr$, $\lpbr \Z_5 \rpbr$, $\lpbr \mathsf{Sym}_3 \rpbr$, and $\lpbr \Z_7 \rpbr$.
Regarding groups of order eight that generate other {\ji} {\pvars}, besides~$\Z_8$, there are two nontrivial cases: the dihedral group~$D_4$ of the square and the quaternion group~$Q_8$.
Let $G \in \{ D_4, Q_8 \}$.
Then forming $G \times G$ and dividing out the two centers identified, $(G \times G)/\{(1,1),(a,a) \}$ gives isomorphic groups, denoted by $G\circ G$.
Since $G \leq G \circ G$, it follows that $\lpbr D_4\rpbr = \lpbr Q_8\rpbr$.
Therefore, the groups~$D_4$ and~$Q_8$ are not {\xp} and so also not {\KN}.
However, the {\pvar} $\lpbr D_4\rpbr = \lpbr Q_8\rpbr$ is {\ji}; see Subsection~\ref{sub: D4 Q8}.
This result is due independently to Kearnes~\cite{Kea18} and the anonymous reviewer.

Other {\ji} {\pvars} generated by a group of order up to~11 are $\lpbr \Z_9 \rpbr$, $\lpbr D_5 \rpbr$, and $\lpbr \Z_{11} \rpbr$.

\subsection{$\mathscr{J}$-trivial semigroups} \label{sub: J-trivial}

Presently, the only {\ji} {\pvars} of $\mathscr{J}$-trivial semigroups known in the literature are generated by the following:
\begin{enumerate}[1.]
\item[$\bullet$] $\N_n = \langle \ea \mid \ea^n=\ez \rangle$, $n \geq 1$;
\item[$\bullet$] $H_n = \langle \ee,\ef \mid \ee^2=\ee,\,\ef^2=\ef,\,(\ee\ef)^n = 0 \rangle$, $n \geq 1$;
\item[$\bullet$] $K_n = \langle \ee,\ef \mid \ee^2=\ee,\,\ef^2=\ef,\,(\ee\ef)^n\ee = 0 \rangle$, $n \geq 1$;
\item[$\bullet$] $\N_n^I$, $H_n^I$, $K_n^I$, $n \geq 1$.
\end{enumerate}
The {\pid} defining the {\pvar} $\excl(\N_n)$ is given in Subsection~\ref{sub: Nn ji}, while the {\pvars} $\excl(H_n)$ and $\excl(K_n)$ are defined by the {\pids} \[ (x^\omega y^\omega)^{n+\omega} \id (x^\omega y^\omega)^n \quad \text{and} \quad (x^\omega y^\omega)^{n+\omega}x^\omega \id (x^\omega y^\omega)^nx^\omega, \] respectively \cite[Propositions~2.3 and~3.3]{Lee17}.

\subsection{Commutative semigroups}

The {\pvar} $\bCom$ of finite commutative semigroups can be decomposed as
\[
\bCom = (\bCom \cap \bG) \vee (\bCom \cap \bA),
\]
where~$\bG$ is the {\pvar} of finite groups and~$\bA$ is the {\pvar} of finite aperiodic semigroups \cite[Figure~9.1]{Alm94}.
Therefore any {\ji} {\pvar} of commutative semigroups is contained in either $\bCom \cap \bG$ or $\bCom \cap \bA$.
As noted in Subsection~\ref{sub: groups}, the {\ji} sub{\pvars} of $\bCom \cap \bG$ are each generated by a cyclic group~$\Z_{p^n}$ of prime power order.
As for $\bCom \cap \bA$, each of its finite semigroups satisfies the identity $x^{n+1} \id x^n$ for all sufficiently large $n \geq 1$ and so belongs to $\lpbr \N_n^I\rpbr$; see Proposition~\ref{P: NnI}(i).
A complete description of {\ji} sub{\pvars} of $\bCom$ is thus dependent on the answer to the following question.

\begin{question}
For each $n \geq 1$\up, what are the {\ji} sub{\pvars} of $\lpbr \N_n^I\rpbr$?
\end{question}

Presently, the only known examples of {\ji} sub{\pvars} of $\lpbr \N_n^I\rpbr$ are $\lpbr \N_k\rpbr$ and $\lpbr \N_k^I\rpbr$, where $1 \leq k \leq n$.

\subsection{Bands}

The {\pvar}~$\bB$ of finite bands is {\fji} (Corollary~\ref{C: B fji}).
Each proper sub\-{\pvar} of~$\bB$ is compact and a complete description of the lattice of sub\-{\pvars} of~$\bB$ is well known; see, for example, Almeida~\cite[Section~5.5]{Alm94}.
The atoms of this lattice are $\bSl = \lpbr\SL\rpbr$, $\bLZ = \lpbr\LZ\rpbr$, and $\bRZ = \lpbr\RZ\rpbr$; see Subsection~\ref{sub: bands}.

Let $\bLNB=\bSl\vee \bLZ$.
For any {\pvar}~$\bV$, define the Mal'cev products $\til\alpha\bV = \bRZ\malce \bV$ and $\til\beta\bV=\bLZ\malce \bV$; see Subsection~\ref{sub: large}.
Then by~\cite{Pas90}, the proper, nontrivial {\sji} {\pvars} of bands are as follows:
\begin{enumerate}[1.]
\item[$\bullet$] $\bLZ$, $\bRZ$, and $\bSl$;
\item[$\bullet$] $(\til\alpha\til\beta)^n\bSl$ and $\til \beta(\til\alpha\til\beta)^n\bSl$, $n \geq 1$, and their duals;
\item[$\bullet$] $(\til\beta\til\alpha)^{n+1}\bLNB$ and $\til \alpha(\til\beta\til\alpha)^n\bLNB$, $n \geq 0$, and their duals.
\end{enumerate}
However, we observe that $\til\alpha \bLNB=\til\alpha \bLZ$.
Since $S\mapsto S^{\ba}$ preserves {\jirr}, it follows that the {\pvar} generated by a finite band is {\ji} if and only if it is {\sji}; see Theorem~\ref{t: band sji}.
As observed after Question~\ref{Q: ji join generate}, it is decidable if a finite semigroup generates a {\sji} {\pvar}.
Therefore, Question~\ref{Q: ji decidable} is affirmatively answered for bands.

\subsection{Kov\'{a}cs--Newman semigroups}

All {\KN} semigroups are known~\cite[Section~7]{RS09}.
These are semigroups with kernel (minimal two-sided ideal) a Rees matrix semigroup over a monolithic group with non-abelian monolith that acts faithfully on the right and left of the kernel.

\subsection{The subdirectly indecomposable viewpoint}

Since every finite semigroup divides (in fact, is a subdirect product of) its {\sdi} homomorphic images, we can restrict our search for new {\ji} semigroups to {\sdi} semigroups, just as in the case of groups, when we can restrict our search to monolithic finite groups.

In more detail, to find the {\ji} {\pvars}, we clearly need only to find finite semigroups~$S$ such that~$\lpbr S\rpbr$ is {\ji} and there exist no semigroups~$T$ with $|T| < |S|$ and $\lpbr T\rpbr = \lpbr S\rpbr$.
Such a semigroup~$S$ is called a \textit{\mog} for the compact {\pvar}~$\lpbr S\rpbr$.

Now the {\mog}s of {\ji} {\pvars}, in fact of {\sji} {\pvars}, must be {\sdi}.
To see this, suppose that~$S$ is any finite semigroup that is not {\sdi}.
Then $S \ll S_1 \times S_2 \times \cdots \times S_k$ for some homomorphic images~$S_j$ of~$S$ such that $|S_j| < |S|$.
But since $\lpbr S\rpbr = \lpbr S_1\rpbr \vee \lpbr S_2\rpbr \vee \cdots \vee \lpbr S_k\rpbr$ and~$\lpbr S\rpbr$ is {\sji}, it follows that $\lpbr S\rpbr = \lpbr S_j\rpbr$ for some~$j$, whence~$S$ is not a {\mog}

If a finite semigroup~$S$ is {\xp} (e.g. {\KN}), then~$S$ is a {\mog} and any {\mog} for~$\lpbr S\rpbr$ is isomorphic to~$S$.
The proof is clear.
However, {\mog}s for the same {\ji} {\pvar} need not be isomorphic; for example, $\lpbr Q_8\rpbr = \lpbr D_4\rpbr$ is {\ji} and $Q_8 \ncong D_4$.

It should be pointed out that a finite semigroup~$S$ being {\sdi} does not imply that the {\pvar}~$\lpbr S\rpbr$ is {\ji} or even {\sji}.
For example, the Rees matrix semigroup \[ S = \mathscr{M}^0 \big(\Z_2,\{1,2\},\{1,2\}; \big[\begin{smallmatrix} 1&0 \\ 0&1 \end{smallmatrix}\big]\big)\] is {\sdi}, but $\lpbr S\rpbr = \lpbr B_2\rpbr \vee \lpbr\Z_2\rpbr$ is not {\sji}; see \cite[Section~4.7]{RS09}.

\subsection{Organization}
The article is organized as follows.
In Section~\ref{sec: bar}, the operator $\bV \mapsto \bV^\ba$ is introduced in detail and some related results are established.
In Section~\ref{sec: req sgps}, some important small semigroups that are required for this paper are defined.
In Section~\ref{sec: general}, some general results regarding {\ji} {\pvars} are established.
In Section~\ref{sec: ji}, some explicit {\pvars} are shown to be {\ji}, and conditions sufficient for a finite semigroup to generate one of them are established.
In Section~\ref{sec: non-ji}, some conditions sufficient for a finite semigroup to generate a non-{\ji} {\pvar} are established.
Results in Sections~\ref{sec: general}--\ref{sec: non-ji} are then employed in Section~\ref{sec: proof} to prove that among all {\pvars} generated by a semigroup of order up to five, only 30~are {\ji}.

\section{Augmented semigroups} \label{sec: bar}

All semigroups and transformation semigroups, with the exception of free semigroups and free profinite semigroups, are assumed finite.
Notation in the monograph~\cite{RS09} will often be followed closely.

Let $(X,S)$ be a transformation semigroup where~$S$ is a semigroup that acts faithfully on the right of a set~$X$.
Then $\ov{(X,S)}=(X,S\cup \ov X)$ where~$\ov X$ is the set of constant maps on~$X$.
The constant map to a fixed element $x \in X$ is denoted by~$\ov x$.
If $(X,S)$ and $(Y,T)$ are transformation semigroups, then \[(X,S)\times (Y,T)=(X\times Y,S\times T)\] with the action $(x,y)(s,t)=(xs,yt)$.

Refer to Eilenberg~\cite{Eil76} for the definition of division~$\prec$ of transformation semigroups.

\begin{lemma}[Eilenberg~{\cite[Exercise~I.4.1, Propositions~I.5.4, and page~20]{Eil76}}] \label{L: Eilenberg}
Let $(X,S)$ and $(Y,T)$ be any transformation semigroups\up.
Then
\begin{enumerate}[\rm(i)]
\item $(X,S) \prec (Y,T)$ implies that $\ov{(X,S)}\prec \ov{(Y,T)}$\up;
\item $(X,S) \prec (Y,T)$ implies that $S\prec T$\up;
\item $\ov{(X,S)\times (Y,T)}\prec \ov{(X,S)}\times \ov{(Y,T)}$\up.
\end{enumerate}
\end{lemma}

Lemma~\ref{L: Eilenberg}(ii) holds because the mappings involved are total.

\begin{lemma}[D.~Allen; see Eilenberg~{\cite[Proposition~I.9.8]{Eil76}}] \label{l:Allen}
If $(X,S)$ is any transformation semigroup\up, then $(S^\bullet,S)\prec (X,S)^{|X|}$\up.
\end{lemma}

Following~\cite[Chapter 4]{RS09}, write $\ov{(S^\bullet,S)}=(S^\bullet, S^\ba)$ and call $S^\ba$ the \textit{augmentation} of~$S$.
Note that if $(X,S)\prec (S^\bullet,S)$, then $(X,S)\prec (S^\bullet,S)\prec (X,S)^{|X|}$ by Lemma~\ref{l:Allen} and hence
\[\ov{(X,S)}\prec \ov{(S^\bullet,S)}\prec \ov{(X,S)}^{|X|}.\]
Thus if $S'=S\cup \ov X$, then $S'\prec S^\ba\prec (S')^{|X|}$, yielding the following result.

\begin{corollary}\label{c:samepv}
If $(X,S)$ is a transformation semigroup such that $(X,S)\prec (S^\bullet,S)$\up, then $\lpbr S\cup \ov X \rpbr = \lpbr S^\ba \rpbr$\up.
In particular\up, if~$S$ is any semigroup and~$J$ is any right ideal of~$S$ on which it acts faithfully\up, then $\lpbr S^\ba \rpbr = \lpbr S\cup \ov J \rpbr$\up.
\end{corollary}

The following are some elementary properties enjoyed by augmentation.

\begin{proposition}\label{p:basic.aug}
Let~$S$ and~$T$ be any finite semigroups\up.
Then
\begin{enumerate}[\rm(i)]
\item $S\prec T$ implies that $S^\ba \prec T^\ba$\up;
\item $\lobr S\times T\robr^\ba\prec S^\ba\times T^\ba$\up.
\end{enumerate}
\end{proposition}
\begin{proof}
(i) Suppose that $S\prec T$, so that $(S^\bullet,S)\prec (T^\bullet,T)$ by Eilenberg~\cite[Proposition~I.5.8]{Eil76}.
Then by Lemma~\ref{L: Eilenberg}(i),
\[ (S^\bullet, S^\ba)=\ov{(S^\bullet, S)} \prec \ov {(T^\bullet,T)}=(T^\bullet, T^\ba). \]
Therefore, $S^\ba\prec T^\ba$ by Lemma~\ref{L: Eilenberg}(ii).

(ii) First note that $(\lobr S\times T\robr^{\bullet}, S\times T)\prec (S^\bullet\times T^{\bullet},S\times T)$.
Then
\begin{alignat*}{2}
(\lobr S & \times T\robr^{\bullet}, \lobr S\times T \robr^\ba) \quad && \\
& = \ov{(\lobr S\times T\robr^{\bullet}, S\times T)} \prec \ov{(S^\bullet\times T^{\bullet},S\times T)} \quad && \text{by Lemma~\ref{L: Eilenberg}(i)} \\
& = \ov{(S^\bullet,S) \times (T^{\bullet},T)} \prec \ov{(S^\bullet,S)} \times \ov{(T^{\bullet},T)} \quad && \text{by Lemma~\ref{L: Eilenberg}(iii)} \\
& = (S^\bullet,S^\ba) \times (T^\bullet,T^\ba) = (S^\bullet \times T^\bullet,S^\ba \times T^\ba). &&
\end{alignat*}
Therefore, $\lobr S\times T \robr^\ba \prec S^\ba \times T^\ba$ by Lemma~\ref{L: Eilenberg}(ii).
\end{proof}

In the following, augmentation is viewed as a continuous operator on the lattice $\bPV$ of {\pvars}.
An operator is \textit{continuous} if it preserves order and directed joins~\cite{RS09}.
For any {\pvar} $\bV$, define \[ \bV^\ba = \lpbr S^\ba \mid S \in \bV\rpbr.\]
Recall that $\bRZ = \lpbr\RZ\rpbr$ is the {\pvar} of right zero semigroups.

\begin{proposition}\label{p:aug.pv}
The operator on $\bPV$ defined by $\bV \mapsto \bV^\ba$ is continuous\up, non-decreasing\up, and idempotent\up.
Further\up,
\begin{enumerate}[\rm(i)]
\item $\lpbr S\rpbr^\ba = \lpbr S^\ba \rpbr$ for any finite semigroup~$S$\up;
\item $\bRZ \subseteq \bV^\ba$ for any nontrivial {\pvar}~$\bV$\up.
\end{enumerate}
Consequently\up, if $\lpbr S \rpbr = \lpbr T \rpbr$\up, then $\lpbr S^\ba \rpbr = \lpbr T^\ba \rpbr$\up.
\end{proposition}

\begin{proof}
Clearly augmentation is order preserving.
Let $\{\bV_\delta \mid \delta \in D \}$ be any directed set of {\pvars}, so that the complete join $\bV = \bigvee_{\delta \in D} \bV_\delta$ is a union.
The inclusion $\bV_\delta^\ba \subseteq \bV^\ba$ clearly holds for all $\delta \in D$, so that $\bigvee_{\delta\in D} \bV_\delta^\ba \subseteq \bV^\ba$.
Conversely, if $S\in \bV^\ba$, say $S\prec T_1^\ba \times T_2^\ba \times \cdots \times T_k^\ba$ for some $T_1,T_2,\ldots, T_k \in \bV$, then due to directedness, there exists $\delta \in D$ with $T_1, T_2,\ldots, T_k \in \bV_\delta$, whence $S\in \bV_\delta^\ba$.
Therefore, augmentation is continuous.

Since $S\prec S^\ba$, it is obvious that augmentation is non-decreasing and the inclusion $\bV^\ba \subseteq \lobr\bV^\ba\robr^\ba$ holds.
To establish the reverse inclusion, it suffices to prove that $\lobr S^\ba\robr^\ba \in \bV^\ba$ for all $S \in \bV$.
But~$S^\ba$ acts faithfully on the right of its minimal ideal $\ov {S^\bullet}$ and it contains all the constant mappings.
Thus $(\ov{S^\bullet}, S^\ba) = \ov{(\ov{S^\bullet}, S^\ba)}$ and $(\ov{S^\bullet}, S^\ba) \prec (\lobr S^\ba\robr^{\bullet}, S^\ba)$.
It follows from Corollary~\ref{c:samepv} that $S^\ba = S^\ba \cup \ov{S^\bullet}$ generates the same {\pvar} as $\lobr S^\ba\robr^\ba$.
This shows that $\lobr S^\ba\robr^\ba \in \bV^\ba$, so that augmentation is idempotent.

It remains to establish parts~(i) and~(ii).

(i) The inclusion $\lpbr S^\ba \rpbr \subseteq \lpbr S \rpbr^\ba$ holds trivially.
To establish the reverse inclusion, suppose that $T \in \lpbr S\rpbr^\ba$, so that $T \prec U^\ba$ for some $U \in \lpbr S \rpbr$.
Then $U \prec S^n$ for some $n\geq 0$ and so $T \prec U^\ba \prec \lobr S^n \robr^\ba\prec \lobr S^\ba\robr^n$ by Proposition~\ref{p:basic.aug}(ii).
Therefore, $T \in \lpbr S^\ba \rpbr$.
Consequently, $\lpbr S^\ba \rpbr = \lpbr S \rpbr^\ba$.

(ii) If~$S$ is a nontrivial semigroup in~$\bV$, then the right zero semigroup~$\RZ$ is a subsemigroup of $S^\ba$, whence $\bRZ \subseteq \bV$.
\end{proof}

\begin{corollary}\label{c:freeze}
Let~$S$ be any finite semigroup whose minimal ideal~$J$ consists of right zeroes\up.
Suppose that~$S$ acts faithfully on the right of~$J$\up.
Then $\lpbr S \rpbr^\ba = \lpbr S \rpbr$\up.
\end{corollary}

\begin{proof}
By Proposition~\ref{p:aug.pv}, it suffices to prove that $\lpbr S^\ba \rpbr = \lpbr S \rpbr$.
But since $\ov{(J,S)}=(J,S)$, it follows that $S = S \cup \ov J$.
The desired conclusion then follows from Corollary~\ref{c:samepv}.
\end{proof}

\section{Some important semigroups} \label{sec: req sgps}

In this section, semigroups that are required throughout the paper are introduced.
Semigroups are given by their presentations, and whenever feasible, multiplication tables.
In presentations, the symbols~$\ee$ and~$\ef$ are exclusively reserved for idempotent elements.

\subsection{Cyclic groups}

The cyclic group of order $n \geq 1$ is \[ \Z_n = \langle \eg \mid \eg^n = \ei \rangle = \{ \ei,\eg,\eg^2,\ldots,\eg^{n-1}\}. \]
The augmentation of $\Z_2 = \{ \ei,\eg \}$ is the semigroup $\ZB = \{ \ei,\eg,\ov\ei,\ov\eg\}$ given in Table~\ref{Tab: Z2bar N2bar}.
The semigroup~$\ZB$ is isomorphic to the semigroup of transformations of the set $\{1,2\}$.
Information on identities satisfied by the semigroups~$\Z_n$ and~$\ZB$ is given in Subsections~\ref{sub: Zn ji} and~\ref{sub: ZB ji}, respectively.

\subsection{Nilpotent semigroups} \label{sub: nilpotent}

The monogenic nilpotent semigroup of order $n \geq 1$ is
\[
\N_n = \langle \ea \mid \ea^n = \ez \rangle = \{ \ez,\ea,\ea^2,\ldots,\ea^{n-1}\}.
\]
The augmentation of $\N_2  = \{ \ez,\ea \}$ is the semigroup $\N_2^\ba = \{ \ov\ez, \ea, \ov\ea, \ov I \}$ given in Table~\ref{Tab: Z2bar N2bar}.
Information on identities satisfied by the semigroups~$\N_n$, $\N_n^I$, $\NB$, and~$\NBI$ is given in Subsections~\ref{sub: Nn ji}, \ref{sub: NnI ji}, \ref{sub: NB ji}, and~\ref{sub: NBI ji}, respectively.

\begin{table}[ht!]
\[\def\arraystretch{1.3}
\begin{array} [c]{c|cccc}
   \ZB & \,    \ei &    \eg & \ov\ei & \ov\eg \\ \hline
   \ei & \,    \ei &    \eg & \ov\ei & \ov\eg \\
   \eg & \,    \eg &    \ei & \ov\ei & \ov\eg \\
\ov\ei & \, \ov\ei & \ov\eg & \ov\ei & \ov\eg \\
\ov\eg & \, \ov\eg & \ov\ei & \ov\ei & \ov\eg
\end{array}
\qquad
\begin{array} [c]{c|cccc}
\N_2^\ba & \, \ov\ez &    \ea & \ov\ea & \ov I \\ \hline
  \ov\ez & \, \ov\ez & \ov\ez & \ov\ea & \ov I \\
     \ea & \, \ov\ez & \ov\ez & \ov\ea & \ov I \\
  \ov\ea & \, \ov\ez & \ov\ez & \ov\ea & \ov I \\
   \ov I & \, \ov\ez & \ov\ea & \ov\ea & \ov I
\end{array}
\]
\caption{Multiplication tables of~$\ZB$ and~$\N_2^\ba$} \label{Tab: Z2bar N2bar}
\end{table}

\subsection{Bands} \label{sub: bands}

The smallest nontrivial bands are the semilattice $\SL = \{ \ez,\ei \}$ and the left zero and right zero semigroups of order two:
\begin{align*}
\LZ & = \langle \ee,\ef \mid \ee^2=\ee\ef=\ee, \, \ef^2=\ef\ee=\ef \rangle = \{ \ee,\ef \}, \\
\RZ & = \langle \ee,\ef \mid \ee^2=\ef\ee=\ee, \, \ef^2=\ee\ef=\ef \rangle = \{ \ee,\ef \};
\end{align*}
see Table~\ref{Tab: Sl2 L2 R2}.
Note that $\SL \cong \N_1^I$ and $\LZ^\op \cong \RZ$.
It is well known that~$\SL$ generates the {\pvar} $\bSl$ of semilattices, $\LZ$ generates the {\pvar} $\bLZ$ of left zero semigroups, and~$\RZ$ generates the {\pvar} $\bRZ$ of right zero semigroups.

\begin{table}[ht!]
\[\def\arraystretch{1.3}
\begin{array} [c]{c|cccc}
\SL \, & \, \ez & \ei \\ \hline
\ez \, & \, \ez & \ez \\
\ei \, & \, \ez & \ei
\end{array}
\qquad
\begin{array} [c]{c|cccc}
\LZ \, & \, \ee & \ef \\ \hline
\ee \, & \, \ee & \ee \\
\ef \, & \, \ef & \ef
\end{array}
\qquad
\begin{array} [c]{c|cccc}
\RZ \, & \, \ee & \ef \\ \hline
\ee \, & \, \ee & \ef \\
\ef \, & \, \ee & \ef
\end{array}
\]
\caption{Multiplication tables of $\SL$, $\LZ$, and~$\RZ$} \label{Tab: Sl2 L2 R2}
\end{table}

The augmentation of $\LZ$ is the semigroup $\LZB = \{ \ee,\ef,\ov\ee,\ov\ef, \ov I\}$ given in Table~\ref{Tab: L2bar}.
Information on identities satisfied by the semigroups~$\LZ$, $\LZ^I$, and~$\LZB$ is given in Subsections~\ref{sub: LZ ji}, \ref{sub: LZI ji}, and~\ref{sub: LZB ji}, respectively.

\begin{table}[ht!]
\[\def\arraystretch{1.3}
\begin{array} [c]{c|ccccc}
  \LZB & \,    \ee &    \ef & \ov\ee & \ov\ef & \ov I \\ \hline
   \ee & \,    \ee &    \ee & \ov\ee & \ov\ef & \ov I \\
   \ef & \,    \ef &    \ef & \ov\ee & \ov\ef & \ov I \\
\ov\ee & \, \ov\ee & \ov\ee & \ov\ee & \ov\ef & \ov I \\
\ov\ef & \, \ov\ef & \ov\ef & \ov\ee & \ov\ef & \ov I \\
 \ov I & \, \ov\ee & \ov\ef & \ov\ee & \ov\ef & \ov I
\end{array}
\]
\caption{Multiplication table of $\LZB$} \label{Tab: L2bar}
\end{table}

\subsection{Completely 0-simple semigroups} \label{sub: comp 0-simple}

The smallest completely 0-simple semi\-groups with zero divisors are the idempotent-generated semigroup
\[
\At = \langle \ea,\ee \mid \ea^2=\ez, \, \ea\ee\ea=\ea, \, \ee^2 = \ee\ea\ee = \ee \rangle = \{ \ez,\ea,\ee,\ea\ee,\ee\ea \}
\]
and the Brandt semigroup
\[
\Bt = \langle \ea,\eb \mid \ea^2=\eb^2=\ez, \, \ea\eb\ea=\ea, \, \eb\ea\eb = \eb \rangle = \{ \ez,\ea,\eb,\ea\eb,\eb\ea\};
\]
see Table~\ref{Tab: A2 B2}.
The Rees matrix representations of these semigroups are \begin{align*} \At & = \mathscr{M}^0 \big(\{1\},\{1,2\},\{1,2\}; \big[\begin{smallmatrix} 1&1 \\ 0&1 \end{smallmatrix}\big]\big) \\ \text{and} \quad \Bt & = \mathscr{M}^0 \big(\{1\},\{1,2\},\{1,2\}; \big[\begin{smallmatrix} 1&0 \\ 0&1 \end{smallmatrix}\big]\big).\end{align*}
The semigroups~$\At$ and~$\Bt$ contain subsemigroups isomorphic to
\begin{align*}
\Az & = \langle \ee,\ef \mid \ee^2=\ee, \, \ef^2=\ef, \, \ee\ef = \ez \rangle = \{ \ez,\ee,\ef,\ef\ee \} \\
\text{and} \quad \Bz & = \langle \ea,\ee,\ef \mid \ee^2=\ee, \, \ef^2=\ef, \, \ee\ef=\ef\ee=\ez, \, \ee\ea = \ea\ef = \ea \rangle = \{ \ez,\ea,\ee,\ef \},
\end{align*}
respectively; see Table~\ref{Tab: A0 B0}.
The semigroup
\[
\el = \langle \ea,\ee \mid \ea\ee=\ez, \, \ee\ea=\ea, \, \ee^2=\ee \rangle = \{ \ez,\ea,\ee\}
\]
and its augmentation $\elB = \{ \ov\ez, \ea,\ee,\ov\ea,\ov\ee \}$ are given in Table~\ref{Tab: l3 l3bar}.
Note that \[\Az \cong \At \setminus \{ \ee \}, \quad \Bz \cong \Bt \setminus \{ \eb \}, \quad \text{and} \quad \el \cong \Az \setminus \{ \ee \} \cong \Bz \setminus \{ \ef \}. \]
Information on identities satisfied by the semigroups~$\Az$, $\Az^I$, $\At$, $\Bt$, and~$\elB$ is given in Subsections~\ref{sub: Az ji}, \ref{sub: AzI ji}, \ref{sub: At ji}, \ref{sub: Bt ji}, and~\ref{sub: elB ji}, respectively.

\begin{table}[ht!]
\[\def\arraystretch{1.2}
\begin{array} [c]{c|ccccc}
   \At & \, \ez &    \ea & \ea\ee & \ee\ea & \ee \\ \hline
   \ez & \, \ez &    \ez &    \ez &    \ez & \ez \\
   \ea & \, \ez &    \ez &    \ez &    \ea & \ea\ee \\
\ea\ee & \, \ez &    \ea & \ea\ee &    \ea & \ea\ee \\
\ee\ea & \, \ez &    \ez &    \ez & \ee\ea & \ee \\
   \ee & \, \ez & \ee\ea &    \ee & \ee\ea & \ee
\end{array}
\qquad
\begin{array} [c]{c|ccccc}
   \Bt & \, \ez &    \ea & \ea\eb & \eb\ea & \eb \\ \hline
   \ez & \, \ez &    \ez &    \ez &    \ez & \ez \\
   \ea & \, \ez &    \ez &    \ez &    \ea & \ea\eb \\
\ea\eb & \, \ez &    \ea & \ea\eb &    \ez & \ez \\
\eb\ea & \, \ez &    \ez &    \ez & \eb\ea & \eb \\
   \eb & \, \ez & \eb\ea &    \eb &    \ez & \ez
\end{array}
\]
\caption{Multiplication tables of $\At$ and~$\Bt$} \label{Tab: A2 B2}
\end{table}

\begin{table}[ht!]
\[\def\arraystretch{1.2}
\begin{array} [c]{c|ccccc}
   \Az & \, \ez & \ef\ee & \ef & \ee \\ \hline
   \ez & \, \ez & \ez    & \ez & \ez \\
\ef\ee & \, \ez & \ez    & \ez & \ef\ee \\
   \ef & \, \ez & \ef\ee & \ef & \ef\ee \\
   \ee & \, \ez & \ez    & \ez & \ee
\end{array}
\qquad
\begin{array} [c]{c|ccccc}
\Bz & \, \ez & \ea & \ee & \ef \\ \hline
\ez & \, \ez & \ez & \ez & \ez \\
\ea & \, \ez & \ez & \ez & \ea \\
\ee & \, \ez & \ea & \ee & \ez \\
\ef & \, \ez & \ez & \ez & \ef
\end{array}
\]
\caption{Multiplication tables of $\Az$ and $\Bz$} \label{Tab: A0 B0}
\end{table}

\begin{table}[ht!]
\[\def\arraystretch{1.3}
\begin{array} [t]{c|ccccc}
\el & \, \ez & \ea & \ee \\ \hline
\ez & \, \ez & \ez & \ez \\
\ea & \, \ez & \ez & \ez \\
\ee & \, \ez & \ea & \ee
\end{array}
\qquad
\begin{array} [t]{c|ccccc}
  \elB & \, \ov\ez &    \ea &    \ee & \ov\ea & \ov\ee \\ \hline
\ov\ez & \, \ov\ez & \ov\ez & \ov\ez & \ov\ea & \ov\ee \\
   \ea & \, \ov\ez & \ov\ez & \ov\ez & \ov\ea & \ov\ee \\
   \ee & \, \ov\ez &    \ea &    \ee & \ov\ea & \ov\ee \\
\ov\ea & \, \ov\ez & \ov\ez & \ov\ez & \ov\ea & \ov\ee \\
\ov\ee & \, \ov\ez & \ov\ea & \ov\ee & \ov\ea & \ov\ee
\end{array}
\]
\caption{Multiplication tables of~$\el$ and $\elB$} \label{Tab: l3 l3bar}
\end{table}

It is shown in Subsection~\ref{sub: Ok} that the semigroup~$\el$ belongs to an infinite class of semigroups~$S$ with the property that $\lpbr S \rpbr$ is not {\ji} but $\lpbr S \rpbr^\ba$ is {\ji}.
The semigroup~$\Bz$ serves as a counterexample to the implications \[ S \text{ is } \sji \ \Longrightarrow \ S \text{ is } \ji \qquad \text{and} \qquad S \text{ is } \sji \ \Longrightarrow \ S^I \text{ is } \sji \] mentioned in the introduction.

\begin{proposition} \label{P: Bz not ji}
\begin{enumerate}[\rm(i)]
\item The {\pvar}~$\lpbr \Bz \rpbr$ is {\sji}\up.
\item The {\pvar}~$\lpbr \Bz \rpbr$ is not {\ji}\up.
\item The {\pvar}~$\lpbr \Bz^I \rpbr$ is not {\sji}\up.
\end{enumerate}
\end{proposition}

\begin{proof}
The {\pvar} $\lpbr \Bz \rpbr$ is {\sji} since it has a unique maximal proper sub{\pvar} \cite[Lemma~5(b)]{Lee07b}.
The {\pvar} $\lpbr \Bz^I \rpbr$ is not {\sji} since it has two maximal proper sub{\pvars} \cite[Lemma~6(b)]{Lee07b}.
In particular, the {\pvar} $\lpbr \Bz^I \rpbr$ is not {\ji}, whence the {\pvar}~$\lpbr \Bz \rpbr$ is also not {\ji}; see Lemma~\ref{L: local}.
\end{proof}

\section{Some general results on {\jirr}} \label{sec: general}

The {\pvar} defined by a class~$\Sigma$ of {\pids} is denoted by~$\lbrs\Sigma\rbrs$, while the {\pvar} generated by a class~$\mathscr{K}$ of finite semigroups is denoted by~$\lpbr\mathscr{K}\rpbr$.
A {\pvar} is \textit{compact} if it is generated by a single finite semigroup.

\begin{proposition} \label{P: maximal}
Every compact {\pvar} contains positively and only finitely many maximal sub{\pvars}\up.
\end{proposition}

\begin{proof}
Let $S = \{ s_1, s_2, \ldots, s_n\}$ be any finite semigroup and let~$\mathcal{V}$ denote the variety generated by~$S$.
Since the lattice of subvarieties of~$\mathcal{V}$ is isomorphic to the lattice of sub{\pvars} of~$\lpbr S\rpbr$, it suffices to show that~$\mathcal{V}$ contains positively and only finitely many maximal subvarieties.
Given any identity $\bu \id \bv$ such that $S \not\models \bu \id \bv$, there exists some substitution~$\varphi$ into~$S$ such that $\bu\varphi \neq \bv\varphi$.
Then~$\varphi$ induces a substitution~$\varphi'$ into the set $\mathscr{X}_n = \{ x_1, x_2, \ldots, x_n \}$ given by $x\varphi' = x_i$ if $x\varphi = s_i$.
Therefore $\bu\varphi' \id \bv\varphi'$ is an identity over $\mathscr{X}_n$ such that $\bu \id \bv \vdash \bu\varphi' \id \bv\varphi'$ and $S \not\models \bu\varphi' \id \bv\varphi'$.
It follows that every proper subvariety of~$\mathcal{V}$ satisfies some identity over $\mathscr{X}_n$.
Modulo the equational theory of the semigroup~$S$, there can only be finitely many identities over $\mathscr{X}_n$ that are violated by~$S$; these identities form a finite preordered set~$\mathcal{P}$ under equational deduction~$\vdash$.
Each greatest element of $(\mathcal{P}, \vdash)$ defines within~$\mathcal{V}$ a maximal subvariety.
\end{proof}

The \textit{exclusion class} $\excl(S)$ of a finite semigroup~$S$ is the class of all finite semigroups~$T$ for which $S \notin \lpbr T \rpbr$.
Recall that a finite semigroup~$S$ is {\ji} if and only if $\excl(S)$ is a {\pvar} \cite[Theorem~7.1.2]{RS09}.

In this section, some results on the property of being {\ji} are established.
There are seven subsections.
The main result of Subsection~\ref{sub: simple pids} demonstrates that many exclusion classes of {\ji} semigroups in this paper are not definable by a certain type of {\pids}.
In Subsection~\ref{sub: large}, the notion of a ``large'' {\pvar} is introduced.
It turns out that the exclusion class of a {\ji} semigroup that is right letter mapping, left letter mapping, or group mapping satisfies this largeness condition.
In Subsection~\ref{sub: augmentation}, it is shown that the operator $\bV \mapsto \bV^\ba$ on $\bPV$ preserves the property of being {\ji}.
More specifically, if $\bu \id \bv$ is a {\pid} that defines the exclusion class $\excl(S)$ of a {\ji} semigroup~$S$, then it is shown how a {\pid} that defines $\excl(S^\ba)$ can be obtained from $\bu \id \bv$.

In Subsection~\ref{sub: iterating}, it is shown that alternately performing the operators $\bV \mapsto \bV^\ba$ and $\bV \mapsto \bV^{\flat} = \lpbr \lobr\lobr S^\op\robr^\ba\robr^\op \mid S \in \bV \rpbr$ on a nontrivial {\pvar}~$\lpbr S\rpbr$ results in an infinite increasing chain of {\pvars}; if the semigroup~$S$ is {\ji} to begin with, then the {\pvars} are all {\ji}.
In Subsection~\ref{sub: Ok}, an infinite class $\{ \OO_k \mid k \geq 2 \}$ of finite semigroups is introduced and shown to satisfy the following property: for each $k \geq 2$, the {\pvar} $\lpbr \OO_k \rpbr$ is not {\ji}, while the {\pvar} $\lpbr \OO_k \rpbr^\ba$ is {\ji}.

In Subsection~\ref{sub: Bergman}, a sufficient condition, due to G.\,M. Bergman, is presented under which a finite {\sdi} group is {\ji}.
In Subsection~\ref{sub: D4 Q8}, the {\pvar} $\lpbr Q_8\rpbr = \lpbr D_4\rpbr$ is shown to be {\ji}; this result is due independently to Kearnes~\cite{Kea18} and the anonymous reviewer.

\subsection{Non-definability by simple {\pids}} \label{sub: simple pids}

For this subsection, the assumption that all semigroups are finite is temporarily abandoned.
The free profinite semigroup on a set~$\A$ is denoted by~$\wh{\A^+}$.
A {\pid} $\bu \id \bv$ is \textit{simple} if~$\bu$ and~$\bv$ belong to the smallest subsemigroup $F(\A)$ of $\wh{\A^+}$ containing~$\A$ that is closed under product and unary implicit operations; the latter condition means that ${ \ov{\{\bw\}^+} \subseteq F(\A) }$ for all $\bw \in F(\A)$.

The following theorem was essentially proved by Almeida and Volkov~\cite{AV03}, based on an earlier variant of Rhodes~\cite{Rho87}.

\begin{theorem}\label{t:simple.ids}
Suppose that~$\bV$ is any proper {\pvar} of semigroups containing all semigroups with abelian maximal subgroups\up.
Then~$\bV$ cannot be defined by simple {\pids}\up.
\end{theorem}

\begin{proof}
Let~$\A$ be a fixed countably infinite set and for any $m,n\geq 1$, let $\mathcal B_{m,n}$ be the variety of semigroups defined by the identity $x^m\id x^{m+n}$.
Then the free semigroup $B(1,m,n)$ on one-generator in $\mathcal B_{m,n}$ is finite and if $x^\eta\in \wh{\{x\}^+}$, then there exists an integer $n_\eta\leq m+n-1$ such that $x^\eta=x^{n_\eta}$ in $B(1,m,n)$.
Thus each implicit operation in $F(\A)$ has a natural interpretation on any semigroup in $\mathcal B_{m,n}$ which agrees with its usual interpretation in finite semigroups (namely interpret $w^\eta$ as $w^{n_\eta}$ for every element~$w$ of a semigroup $S\in \mathcal B_{m,n}$).

Suppose that~$\bV$ is defined by a set~$\Sigma$ of simple {\pids}.
Let $\mathcal W$ be the variety of universal algebras defined by~$\Sigma$ in the signature~$\tau$ consisting of multiplication and all unary implicit operations and let~$T$ be a finite semigroup.
Then there exist $m \geq 6$ and $n \geq 1$ such that~$T$ belongs to $\mathcal B_{m,n}$.
As discussed above, $\mathcal B_{m,n}$ can be viewed as a variety in the signature~$\tau$ such that the unary implicit operations have their usual interpretations in all finite semigroups in $\mathcal B_{m,n}$.

Now McCammond~\cite{McC91} has shown that for each integer $k\geq 1$, the semigroup $B(k,m,n)$ has cyclic maximal subgroups and that there is a system of cofinite ideals for $B(k,m,n)$ with empty intersection.
Therefore, $B(k,m,n)$ is an infinite subdirect product of finite semigroups with abelian maximal subgroups.
Since $\mathcal W$ contains all finite semigroups with abelian maximal subgroups, it follows that $B(k,m,n)\in \mathcal W$, whence $\mathcal B_{m,n} \subseteq \mathcal W$.
Therefore, $T$ belongs to $\mathcal W$ and so satisfies the {\pids}~$\Sigma$.
Consequently, $T\in \bV$ and hence~$\bV$ is the {\pvar} of all finite semigroups.
\end{proof}

In this paper, {\pids} involving idempotents from the minimal ideal of a free profinite semigroup are often used to define the exclusion {\pvars} of {\ji} semigroups.
Since many of these exclusion {\pvars} contain all semigroups with abelian maximal subgroups, Theorem~\ref{t:simple.ids} implies that, in general, simple {\pids} cannot be used in their definition.
It is presently unknown if one must use idempotents from the minimal ideal.

\subsection{Large exclusion {\pvars}} \label{sub: large}

If~$\bV$ and~$\bW$ are {\pvars} of semigroups, then their \emph{Mal'cev product} $\bV \malce \bW$ is the {\pvar} generated by all semigroups~$S$ with a homomorphism $\varphi\colon S\to T$ such that $T\in \bW$ and $e\varphi^{-1}\in \bV$ for all idempotents $e\in T$.
A remarkable property of the Mal'cev product is that
\begin{equation}\label{eq:malcev.meet}
\bigg(\bigcap_{\alpha\in A}\bV_{\alpha}\bigg)\malce \bW = \bigcap_{\alpha\in A}(\bV_\alpha\malce \bW);
\end{equation}
see~\cite{RS09} for details.

Let~$\mathbf{1}$ denote the {\pvar} of trivial semigroups.
For any {\ji} semigroup~$S$, we say that $\excl(S)$ is \textit{large} if \[\mathbf 1\malce \excl(S)=\excl(S).\]
If $\excl(S)$ is large and $\{ \bV_\alpha \mid \alpha \in A \}$ is a collection of {\pvars} such that $\bigcap_{\alpha\in A} \bV_\alpha=\mathbf{1}$, then it follows from~\eqref{eq:malcev.meet} and the fact that $\excl(S)$ is {\mi} that $\bV_\alpha\malce \excl(S)=\excl(S)$ for some $\alpha\in A$.
In particular, either \[ \bA\malce \excl(S) = \excl(S) \quad \text{or} \quad \bG\malce \excl(S)=\excl(S), \]where~$\bA$ is the {\pvar} of finite aperiodic semigroups and~$\bG$ is the {\pvar} of finite groups.
For more examples of {\pvars} with trivial intersection, see~\cite{RS09}.

If $S$ is a finite subdirectly indecomposable semigroup, then~$S$ has a unique $0$-minimal ideal~$I$ (where if~$S$ has no zero, then we consider the minimal ideal as $0$-minimal).
Moreover, one of the following cases holds:
\begin{enumerate}[1.]
\item[$\bullet$] $I^2=0$ (the null case);
\item[$\bullet$] $S$ acts faithfully on the right of the set of $\mathscr L$-classes of~$I$ (the left letter mapping case);
\item[$\bullet$] $S$ acts faithfully on the left of the set of $\mathscr L$-classes of~$I$ (the right letter mapping case);
\item[$\bullet$] $I$ contains a nontrivial maximal subgroup and~$S$ acts faithfully on both the left and right of~$I$ (the group mapping case).
\end{enumerate}
In the last three cases we say that~$S$ is of \textit{semisimple} type; see~\cite[Sec.~4.7]{RS09}.

\begin{theorem} \label{T: large}
Let~$S$ be any subdirectly indecomposable {\ji} semigroup of semi\-simple type \up(left letter mapping\up, right letter mapping\up, or group mapping\up)\up.
Then $\excl(S)$ is large\up.
\end{theorem}

\begin{proof}
Obviously, $\excl(S)\subseteq \mathbf{1} \malce \excl(S)$.
As $\excl(S)$ is the largest {\pvar} that fails to contain~$S$, it suffices to show that $S \notin \mathbf{1} \malce \excl(S)$.
But~\cite[Theorem~4.6.50]{RS09} immediately implies that in any of the three cases, $S\in \mathbf{1} \malce \bV$ if and only if $S\in \bV$ for any {\pvar}~$\bV$.
Thus $S\notin \mathbf{1}\malce \excl(S)$ and so $\mathbf{1} \malce \excl(S)=\excl(S)$.
\end{proof}

The proof of Theorem~\ref{T: large} is in fact valid if~$S$ is left letter mapping, right letter mapping, or group mapping even if it is not {\sdi}.

\subsection{Augmentation preserves {\jirr}} \label{sub: augmentation}

In this subsection, augmentation is shown to preserve {\jirr}.
Some special cases were previously considered in~\cite[Section~7.3]{RS09}.

\begin{theorem} \label{T: Sbar}
The operator $\bV \mapsto \bV^\ba$ preserves the property of being {\ji}\up.
In particular\up, if a {\pvar} $\lpbr S\rpbr$ is {\ji}\up, then the {\pvar} $\lpbr S^\ba\rpbr$ is also {\ji}\up.
Further\up, if $\excl(S)=\lbrs \bu\id\bv\rbrs$ where $\bu,\bv\in \wh{\A^+}$\up, then \[ \excl(S^\ba) = \lbr (\be z\bu)^{\omega}\id (\be z\bv)^{\omega}\rbr\] where $z\notin \A$ and~$\be$ is an idempotent in the minimal ideal of $\wh{(\A\cup \{z\})^+}$\up.
\end{theorem}

\begin{proof}
First note that since $S \notin \excl(S)$, there exists some homomorphism $\varphi \colon \wh{\A^+}\to S$ such that $\bu \varphi \neq \bv \varphi$.
Let~$1$ denote the identity element of~$S^\bullet$, and extend~$\varphi$ to a homomorphism $\wh{(\A\cup \{z\})^+} \to S^\ba$ by sending~$z$ to $\ov 1$.
Then $(\be z\bu)^{\omega}\varphi = \ov{\bu\varphi} \neq \ov{\bv\varphi} = (\be z\bv)^{\omega}\varphi$ and so $S^\ba \notin \lbrs (\be z\bu)^\omega \id (\be z\bv)^\omega \rbrs$.

To complete the proof, it suffices to assume that $T \notin \lbrs (\be z\bu)^\omega \id (\be z\bv)^\omega \rbrs$, and show that $S^\ba \in \lpbr T \rpbr$.
Replacing~$T$ by a subsemigroup if necessary, generality is not lost by assuming the existence of a surjective homomorphism $\psi \colon \wh{(\A\cup\{z\})^+} \to T$ such that $(\be z\bu)^\omega \psi\neq (\be z\bv)^\omega \psi$.
Now the semigroup~$T$ acts on the right of the set~$B$ of $\mathscr L$-classes of its minimal ideal~$J$; let $(B,\mathsf{RLM}(T))$ denote the resulting faithful transformation semigroup.
Note that $(B,\mathsf{RLM}(T)) = \ov{(B,\mathsf{RLM}(T))}$ because if $b\in B$, then any element of~$T$ in the $\mathscr L$-class of~$b$ acts on~$B$ as a constant map to~$b$ by the structure of completely simple semigroups.
It follows from Corollary~\ref{c:freeze} that $\lpbr\mathsf{RLM}(T)\rpbr^\ba= \lpbr\mathsf{RLM}(T)\rpbr$, since the constant mappings form the minimal ideal of $\mathsf{RLM}(T)$.

Since $(\be z)\psi$ is in the minimal ideal~$J$ of~$T$, the elements $((\be z)\psi)( \bu\psi)$ and $((\be z)\psi)( \bv\psi)$ are $\mathscr R$-equivalent.
However, they are not $\mathscr L$-equivalent because otherwise they would be $\mathscr H$-equivalent and hence have the same idempotent power, as~$J$ is completely simple.
Thus $\bu\psi$ and $\bv\psi$ have distinct images under the quotient map $T \to \mathsf{RLM}(T)$.
Consequently, there is a homomorphism $\varphi \colon \wh \A^+\to \mathsf{RLM}(T)$ such that $\bu\varphi \neq \bv\varphi$, that is, $\mathsf{RLM}(T) \notin \excl(S)$.
Therefore, $S \in \lpbr \mathsf{RLM}(T) \rpbr$, whence $S^\ba \in \lpbr\mathsf{RLM}(T)\rpbr^\ba = \lpbr\mathsf{RLM}(T) \rpbr \subseteq \lpbr T\rpbr$ as required.
\end{proof}

\begin{corollary}
Let $(X,S)$ be any transformation semigroup with $(X,S) \prec (S^\bullet,S)$\up.
Suppose that the {\pvar} $\lpbr S\rpbr$ is {\ji}\up.
Then the {\pvar} $\lpbr S\cup \ov X\rpbr$ is also {\ji}\up.
\end{corollary}

\begin{proof}
This follows from Corollary~\ref{c:samepv} and Theorem~\ref{T: Sbar}.
\end{proof}

Note that if $S$ is {\ji}, then $\excl(S^\ba)$ will be large by Theorem~\ref{T: large} (and the remark following it).

\subsection{Iterating augmentation and its dual to bands} \label{sub: iterating}

For any semigroup~$S$, define \[ S^\flat=\lobr\lobr S^\op\robr^\ba\robr^\op .\]
In other words, $S^\flat$ is obtained by considering the left action of $S$ on $S^{\bullet}$ and adjoining constant maps.
For any {\pvar}~$\bV$, define \[ \bV^{\flat} = \lpbr S^{\flat} \mid S \in \bV \rpbr. \]
By symmetry, $\bV \mapsto \bV^\flat$ is a continuous idempotent operator that preserves {\jirr}; see~\cite[Chapter~2]{RS09}.
Define the operators $\alpha,\beta \colon \bPV \to \bPV$ by $\alpha\bV=\bV^\ba$ and $\beta\bV=\bV^\flat$.
The aim of this subsection is to show that for any nontrivial finite semigroup~$S$, the hierarchy
\begin{equation}\label{eq:hier}
\bV_n=(\beta\alpha)^n\lpbr S\rpbr, \quad n\geq 0
\end{equation}
is strict, as is the dual hierarchy obtained by interchanging the roles of $\alpha$ and $\beta$.
An important observation is that $\beta\alpha\lpbr S \rpbr$ is a compact {\pvar} containing~$\SL$ that is generated by $\lobr S^\ba \robr^\flat$, which is left mapping with respect to its minimal ideal.
Thus it suffices to handle the case that $\bSl \subseteq \lpbr S\rpbr$ and~$S$ is left mapping with respect to its minimal ideal.

\begin{proposition}\label{p:rzbound}
For any finite semigroup~$S$\up, \[ S^\ba\in \bRZ \malce (\lpbr S\rpbr\vee \bSl) \quad \text{and} \quad S^\flat\in \bLZ \malce (\lpbr S\rpbr\vee \bSl). \]
\end{proposition}

\begin{proof}
Clearly, $S^\ba/\ov{S^{\bullet}}$ divides the semigroup $S^0$ obtained by adjoining an external zero element~$0$ to~$S$.
Since $\ov{S^\bullet}$ is a right zero semigroup and $\lpbr S^0 \rpbr \subseteq \lpbr S\rpbr\vee \bSl$, the inclusion $S^\ba \in \bRZ\malce (\lpbr S\rpbr\vee \bSl)$ holds.
The second inclusion is dual.
\end{proof}

Define the operators $\til \alpha,\til \beta\colon \bPV \to \bPV$ by $\til \alpha \bV = \bRZ \malce \bV$ and $\til \beta\bV = \bLZ \malce \bV$.
These operators are idempotent.
For any finite semigroup~$S$ that contains~$\SL$ as a subsemigroup, define the hierarchy
\begin{equation}\label{eq:hier2}
\bU_n=(\til \beta\til \alpha)^n\lpbr S\rpbr,\quad n\geq 0.
\end{equation}
Observe that $\bV_n\subseteq \bU_n$ for all $n\geq 0$ as a consequence of Proposition~\ref{p:rzbound}.

\begin{proposition}\label{p:band.rm.lm}
Suppose that~$S$ is any nontrivial band that is left mapping with respect to its minimal ideal and that $\bV$ is any {\pvar} such that $\bSl \subseteq \bV$\up.
Then $S^\ba\in \bRZ \malce \bV$ if and only if $S \in \bV$\up.
\end{proposition}

\begin{proof}
If $S\in \bV$, then $S^\ba\in \bRZ \malce \bV$ by Proposition~\ref{p:rzbound}.
Conversely, since $S^\ba$ is a band, $S^\ba \in \bRZ \malce \bV$ if and only if $S^\ba\in \mathbf{D} \malce \bV$, where~$\mathbf{D}$ is the {\pvar} of semigroups whose idempotents are right zeroes; this occurs if and only if the quotient of $S^\ba$ by the intersection $\mathsf{LM}$ of all its left mapping congruences belongs to~$\bV$ \cite[Theorem~4.6.50]{RS09}.
Note that since~$S$ is a left mapping band with respect to its minimal ideal, its minimal ideal consists of at least two left zeroes.
Therefore, the minimal ideal of $S^\ba$ contains no elements of~$S$.
Then $S^\ba/\mathsf{LM}\cong S^0$ because $S^\ba$ acts trivially on the left of its minimal ideal and acts as~$S$ does on the left of its other $\mathscr J$-classes.
Since $\SL \in \bV$ by assumption, it follows that $S^\ba/\mathsf{LM}\in \bV$ if and only if $S \in \bV$.
\end{proof}

\begin{corollary}\label{c:erase}
Suppose that~$S$ is any nontrivial band that is left mapping with respect to its minimal ideal and that $\bV$ is any {\pvar} such that $\bSl \subseteq \bV$\up.
Then $\lobr S^\ba\robr^\flat \in \til\beta \til \alpha\bV$ if and only if $S\in \bV$\up.
\end{corollary}

\begin{proof}
Since $S^\ba$ is a nontrivial band that is right mapping with respect to its minimal ideal, the dual of Proposition~\ref{p:band.rm.lm} implies that $\lobr S^\ba \robr^\flat\in \bLZ \malce (\bRZ \malce \bV)$ if and only if $S^\ba\in \bRZ \malce \bV$.
An application of Proposition~\ref{p:band.rm.lm} then yields that $\lobr S^\ba\robr^\flat \in \bLZ \malce (\bRZ \malce \bV)$ if and only if $S\in \bV$.
\end{proof}

The hierarchies~\eqref{eq:hier} and~\eqref{eq:hier2} for the case $S = \SL$ are now analyzed.
Recall that~$\bB$ denotes the {\pvar} of finite bands.

\begin{lemma}\label{l:go.up}
Consider the hierarchies~\eqref{eq:hier} and \eqref{eq:hier2} with $S = \SL$\up.
Then
\begin{enumerate}[\rm(i)]
\item $\bV_n\nsubseteq \bU_{n-1}$ for all $n\geq 1$\up;
\item the hierarchies~\eqref{eq:hier} and~\eqref{eq:hier2} are strict\up;
\item $\bigcup_{n\geq 0} \bU_n=\bigcup_{n\geq 0}\bV_n = \bB$\up.
\end{enumerate}
\end{lemma}

\begin{proof}
(i) This is established by induction on~$n$.
The exclusion $\bV_1\nsubseteq \bU_0$ holds since $S^\ba\in \bV_1$ while $S^\ba \notin \bSl = \bU_0$ due to $\RZ \subseteq S^\ba$.
Suppose that $\bV_n\nsubseteq \bU_{n-1}$ for some $n\geq 2$.
Note that $\bV_n$ is generated by a band of the form $T=R^\flat$ and so $T$ is left mapping with respect to its minimal ideal.
Since $T\notin \bU_{n-1}$, it follows from Corollary~\ref{c:erase} that $\lobr T^\ba\robr^\flat\notin \bU_n$.
Therefore, $\lobr T^\ba\robr^\flat\in \bV_{n+1}\setminus \bU_n$, whence $\bV_{n+1}\nsubseteq \bU_n$.

(ii)
Since $\bV_n\nsubseteq \bU_{n-1}$ by part~(i) and $\bV_{n-1}\subseteq \bU_{n-1}$, the hierarchy~\eqref{eq:hier} is strict.
Similarly, $\bV_n\subseteq \bU_n$ and $\bV_n\nsubseteq \bU_{n-1}$ imply that the hierarchy~\eqref{eq:hier2} is strict.

(iii) This result holds because the lattice of band {\pvars} is well known not to contain any strictly increasing infinite chain of sub\-{\pvars} whose union is not~$\bB$.
\end{proof}

\begin{theorem}\label{t:hier}
The hierarchy~\eqref{eq:hier} is strict for any nontrivial finite semigroup~$S$\up.
\end{theorem}

\begin{proof}
Since the hierarchy stabilizes as soon as two consecutive {\pvars} are identical, replacing~$S$ by $\lobr S^\ba\robr^\flat$ if necessary, $S$ can be assumed to contain~$\SL$ as a subsemigroup.
It then follows from Lemma~\ref{l:go.up} that $\bigcup_{n\geq 0} \bV_n$ contains the {\pvar}~$\bB$.
But since~$\bB$ is not contained in any compact {\pvar} \cite{Sap05}, the union $\bigcup_{n\geq 0} \bV_n$ is not compact.
Since each {\pvar}~$\bV_n$ is compact, the hierarchy is strict.
\end{proof}

\begin{corollary} \label{C: hierarchy}
If $\lpbr S\rpbr$ is {\ji}\up, then the {\pvars} \[ \lpbr S \rpbr, \, \lpbr S^\ba \rpbr, \, \lpbr \lobr S^\ba\robr^\flat \rpbr, \, \lpbr \lobr\lobr S^\ba\robr^\flat\robr^\ba \rpbr, \, \ldots \] are {\ji}\up; these {\pvars} are all distinct except possibly for $\lpbr S \rpbr = \lpbr S^\ba \rpbr$\up.
A dual result holds when $^\flat$ is first applied before~$^\ba$\up.
\end{corollary}

\begin{corollary} \label{C: B fji}
The {\pvar}~$\bB$ is {\fji}\up.
\end{corollary}

\begin{proof}
Since the {\pvar}~$\bSl$ is {\ji}, each step in the hierarchy \eqref{eq:hier} is {\ji} with $S=\SL$.
As the union of a chain of {\ji} {\pvars} is {\fji} \cite[Lemma~6.1.13]{RS09}, it follows from Lemma~\ref{l:go.up} that~$\bB$ is {\fji}.
\end{proof}

Using the known structure of the lattice of band {\pvars}~\cite{Pas90} (which coincides with the lattice of all band varieties), we can say more. Namely, we will show that any {\sji} band is {\ji}.
Recall that $\bLNB=\bSl\vee\bLZ$.

\begin{proposition}\label{p: sjisave}
The {\pvar} $\bRZ\malce \bLNB$ is generated by $\LZ^{\ba}$\up.
\end{proposition}

\begin{proof}
It follows from Pastijn~\cite[Figure~3]{Pas90} and the description of the lattice of {\pvars} of bands (see, for example, Almeida~\cite[Figure~5.1]{Alm94}) that the {\pvar}
$\lpbr \LZ,\RZ^I\rpbr = \lbr x^2 \id x, \, xyz \id xzyz \rbr$ is the unique maximal sub\-{\pvar} of $\bRZ\malce\bLNB = \lbr x^2 \id x, \, xyz \id xzxyz \rbr$.
It is then routinely checked that $\LZ^{\ba} \in \bRZ\malce\bLNB \setminus \lpbr \LZ,\RZ^I\rpbr$.
Consequently, $\lpbr \LZ^{\ba} \rpbr = \bRZ\malce\bLNB$.
\end{proof}

By Proposition~\ref{p: sjisave} and results from Pastijn~\cite{Pas90}, a description of proper {\sji} {\pvars} of bands can be given as follows.
Let $S=\LZ^{\ba}$ and $T=\RZ^\flat$.
Then the proper nontrivial {\sji} {\pvars} of bands are $\bLZ$, $\bRZ$, and those {\pvars} that can be obtained by applying an alternating word $\bw(\til \alpha,\til \beta)$ over $\{ \til \alpha, \til \beta \}$ to the {\pvars} generated by~$S$, $T$, or~$\SL$ (where the last letter of~$\bw$ should be $\til \beta$ when starting from $\lpbr S\rpbr$ and should be $\til \alpha$ when starting from $\lpbr T\rpbr$).
Further, there are no {\sji} {\pvars} strictly in between any successive iterations of these operators.
Since $\alpha\bV\leq \til \alpha\bV$ and $\beta\bV\leq \til\beta\bV$ for any {\pvar} $\bV$ containing $\bSl$, and each successive iteration of~$\alpha$ and~$\beta$ starting from the {\pvar} generated by one of~$S$, $T$, or~$\SL$ (where the rightmost operator applied must be~$\beta$ for~$S$ and~$\alpha$ for~$T$) results in a new {\ji} {\pvar}, it follows that if $\bw(x,y)$ is any alternating word over $\{ x, y \}$, then $\bw(\alpha,\beta)\bV=\bw(\til \alpha,\til\beta)\bV$ whenever~$\bV$ is one of the {\pvars} generated by~$S$, $T$, or~$\SL$.
Consequently, each {\sji} proper {\pvar} of bands is, in fact, {\ji} by Corollary~\ref{C: hierarchy}.
The following result is thus established.

\begin{theorem}\label{t: band sji}
Any {\sji} band  is {\ji}\up, that is\up, a proper {\pvar} of bands is {\sji} if and only if it is {\ji}\up.
\end{theorem}

In particular, since {\sji} is a decidable property, {\ji} is also decidable for finite bands.
The answer to Question~\ref{Q: ji decidable} is thus affirmative for bands.

\subsection{From non-{\ji} {\pvars} to {\ji} {\pvars}} \label{sub: Ok}

For each $k \geq 2$, define the semigroup \[ \OO_k= \langle \Ox,\Oe \mid \Ox^k = \Ox^{k-1}\Oe = \ez, \, \Oe\Ox=\Ox, \, \Oe^2=\Oe \rangle.\]
The main goal of this subsection is to show that the {\pvar} $\lpbr \OO_k \rpbr$ is not {\ji} whereas the {\pvar} $\lpbr \OO_k \rpbr^\ba$ is {\ji}.
It is also shown that the {\pvars} $\lpbr \OO_2 \rpbr^\ba, \lpbr \OO_3 \rpbr^\ba, \lpbr \OO_4 \rpbr^\ba, \ldots$ are all distinct.

It is easily seen that the semigroups $\OO_2$ and $\el$ are isomorphic by referring to their presentations.
Since the semigroup $\elB$ is of order five (Subsection~\ref{sub: comp 0-simple}), the~{\ji} {\pvar} $\lpbr \elB \rpbr = \lpbr \OO_2 \rpbr^\ba$ is required later in the paper (Theorem~\ref{T: el3}).

\begin{lemma}\label{L: Ok order}
For each $k\geq 2$\up, the semigroup $\OO_k$ consists precisely of the following $2k-1$ distinct elements\up:
\begin{equation}\label{eq:rhs}
\ez, \, \Ox, \, \Ox^2, \ldots, \, \Ox^{k-1}, \, \Oe, \, \Ox \Oe, \, \Ox^2 \Oe, \ldots, \, \Ox^{k-2} \Oe.
\end{equation}
\end{lemma}

\begin{proof}
It is routinely checked that~\eqref{eq:rhs} are all the elements of~$\OO_k$.
Therefore, it remains to verify that the elements in~\eqref{eq:rhs} are distinct.
Recall that the right zero semigroup of order two is $\RZ = \{ \ee,\ef \}$ and that the monogenic nilpotent semigroup of order~$k$ is
\[
\N_k = \langle \ea \mid \ea^k = \ez \rangle = \{ \ez,\ea,\ea^2,\ldots,\ea^{k-1}\}.
\]
Consider the subsemigroup  $T = (\N_k^I\times \RZ) \setminus \{(I,\ee)\}$ of $\N_k^I\times \RZ$ and the ideal $J=\{(\ez,\ee), (\ez,\ef),(\ea^{k-1},\ef)\}$ of~$T$.
Define $\varphi \colon \{\Ox,\Oe\}^+\to T/J$ by $\Ox\varphi=(\ea,\ee)$ and $\Oe\varphi=(I,\ef)$.
Then
\[
\Ox^k \varphi = (\ez,\ee) \in J, \quad (\Ox^{k-1}\Oe) \varphi = (\ea^{k-1},\ef) \in J, \quad (\Oe\Ox) \varphi = \Ox\varphi, \quad \Oe^2\varphi = \Oe \varphi.
\]
It follows that~$\varphi$ induces a homomorphism $\OO_k \mapsto T/J$ that separates the elements in~\eqref{eq:rhs}.
\end{proof}

\begin{proposition} \label{P: Ok not ji}
The {\pvar} $\lpbr\OO_k\rpbr$ is not {\ji}\up.
\end{proposition}

\begin{proof}
Since $\OO_k\prec T/J \prec \N_k^I \times \RZ$ by the proof of Lemma~\ref{L: Ok order} (where we retain the notation of that proof), the inclusion $\lpbr\OO_k\rpbr \subseteq \lpbr\N_k^I \rpbr \vee \bRZ$ holds.
But $\lpbr\N_k^I\rpbr$ consists of commutative semigroups while $\bRZ$ consists of bands.
Therefore, $\lpbr\OO_k\rpbr \nsubseteq \lpbr\N_k^I\rpbr$ and $\lpbr\OO_k\rpbr \nsubseteq \bRZ$.
\end{proof}

It remains to prove that the {\pvar} $\lpbr\OO_k\rpbr^\ba$ is {\ji}.

\begin{lemma}\label{l:general.good.cyclic}
Suppose that~$U$ is any semigroup generated by two elements~$f$ and~$y$ such that $f^2=f$\up, $fy=y$\up, and $y^{k-1} \notin \{y^n \mid n\geq k\}$\up.
Then
\begin{enumerate}[\rm(i)]
\item $y,y^2,\ldots,y^{k-1}$ are distinct and not in $\{y^n \mid n\geq k\}$\up;
\item $f, yf, y^2f, \ldots, y^{k-2}f$ are distinct and not in $\{ y^mf \mid m \geq k-1\}$\up;
\item $y^i=y^jf$ implies that either $i=j$ or $i,j\geq k-1$\up.
\end{enumerate}
\end{lemma}

\begin{proof}
(i) This follows from the structure of monogenic semigroups.

(ii) Suppose that $y^if = y^jf$ for some $i,j \geq 0$.
Then $y^{i+1}=y^ify=y^jfy=y^{j+1}$.
Therefore, by part~(i), either $i=j$ or $i,j\geq k-1$.

(iii) Suppose that $y^i=y^jf$ for some $i \geq 1$ and $j \geq 0$.
Then $y^{i+1}=y^jfy=y^{j+1}$.
Therefore, by part~(i), either $i=j$ or $i,j\geq k-1$.
\end{proof}

Recall that the inclusion $\lpbr\OO_k\rpbr \subseteq \lpbr\N_k^I \rpbr \vee \bRZ$ was established in the proof of Proposition~\ref{P: Ok not ji}; this result is generalized in the following.

\begin{lemma}\label{l:good.cyc.rz}
Suppose that~$T$ is any finite semigroup generated by two elements~$d$ and~$z$ such that $d^2=d$\up, $dz=z$\up, and $z^{k-1} \notin \{ z^n \mid n \geq k \}$\up.
Then $\lpbr \OO_k \rpbr \subseteq \lpbr T \rpbr \vee \bRZ$\up.
\end{lemma}

\begin{proof}
Consider the semigroup $T \times \RZ$ and its subsemigroup $U=\langle y,f\rangle$ generated by $y=(z,\ee)$ and $f=(d,\ef)$.
Then it is routinely checked that
\begin{enumerate}[\ (a)]
\item $f^2=f$, $fy=y$,
\item $y^n=(z^n,\ee)$ for all $n\geq 1$,
\item $y^nf=(z^nd,\ef)$ for all $n\geq 1$.
\end{enumerate}
It follows from~(b) and the assumption $z^{k-1} \notin \{ z^n \mid n \geq k \}$ that
\begin{enumerate}[\ (a)]
\item[(d)] $y^{k-1} \notin \{y^n \mid n\geq k\}$.
\end{enumerate}
It is clear from~(a) that $U = \{ y^i, y^jf \mid i\geq 1, j\geq 0 \}$.
In fact, it follows from (a)--(d) and Lemma~\ref{l:general.good.cyclic} that
\begin{enumerate}[\ (a)]
\item[(e)] the elements $y,y^2,y^3, \ldots, y^k, f,yf,y^2f, \ldots, y^{k-1}f$ of~$U$ are distinct.
\end{enumerate}

Now it is routinely checked that the set \[J=\{y^n, y^mf \mid n\geq k,\, m\geq k-1\}\] is an ideal of~$U$.
By~(e), the set $U\setminus J$ consists of the elements \[ y,y^2,y^3, \ldots, y^{k-1}, f,yf,y^2f, \ldots, y^{k-2}f.\]
Therefore, $\OO_k\cong U/J$ by Lemma~\ref{L: Ok order}, whence $\lpbr \OO_k \rpbr \subseteq \lpbr T\rpbr \vee \bRZ$.
\end{proof}

\begin{theorem} \label{T: Ok}
\begin{enumerate}[\rm(i)]
\item For each $k \geq 2$\up, the {\pvar} $\lpbr \OO_k \rpbr^\ba$ is {\ji} and
\begin{equation}\label{id: exOk}
\excl(\OO_k^\ba) = \lbr (\be c(a^\omega b)^{k-1})^\omega \id (\be c((a^\omega b)^{k-1})^{\omega+1})^\omega \rbr,
\end{equation}
where $\be$ is an idempotent in the minimal ideal of $\wh{\{a,b,c\}^+}$\up.
\item The {\pvars} $\lpbr \OO_2 \rpbr^\ba, \lpbr \OO_3 \rpbr^\ba, \lpbr \OO_4 \rpbr^\ba, \ldots$ are all distinct\up.
\end{enumerate}
\end{theorem}

\begin{proof}
(i) Let~$\varphi$ denote the substitution into $\OO_k^\ba$ given by $a \mapsto \Oe$, $b \mapsto \Ox$, and $c \mapsto \ov 1$.
Then $(\be c(a^{\omega}b)^{k-1})^\omega \varphi = \ov{\Ox^{k-1}}$ and $(\be c((a^{\omega}b)^{k-1})^{\omega+1})^\omega \varphi = \ov{\Ox^k}$, and these are different elements of $\OO_k^\ba$.
Therefore, the semigroup $\OO_k^\ba$ violates the {\pid} in~\eqref{id: exOk}.

It remains to assume that a semigroup~$T$ violates the {\pid} in~\eqref{id: exOk}, and then show that $\OO_k^\ba\in \lpbr T \rpbr$.
Replacing~$T$ by a subsemigroup if necessary, generality is not lost by assuming the existence of a surjective homomorphism $\psi\colon \wh{\{a,b,c\}^+} \to T$ such that \[ (\be c(a^{\omega}b)^{k-1})^\omega \psi \neq (\be c((a^{\omega}b)^{k-1})^{\omega+1})^\omega \psi .\]
Put $f=a^{\omega}\psi$ and $y=(a^{\omega}b)\psi$ and note that $f^2=f$ and $fy=y$.

The semigroup~$T$ acts on the right of the set~$B$ of $\mathscr L$-classes of its minimal ideal~$J$; denote the corresponding faithful transformation semigroup by $(B,\mathsf{RLM}(T))$.
Note that $(B,\mathsf{RLM}(T)) = \ov{(B,\mathsf{RLM}(T))}$ because if $b\in B$, then any element of~$T$ in the $\mathscr L$-class of~$b$ acts on~$B$ as a constant map to~$b$ by the structure of completely simple semigroups.
Consequently, $\lpbr \mathsf{RLM}(T) \rpbr^\ba = \lpbr\mathsf{RLM}(T) \rpbr$ by Corollary~\ref{c:freeze} since the constant mappings form the minimal ideal of $\mathsf{RLM}(T)$.

Since $(\be c)\psi$ is in the minimal ideal~$J$ of~$T$, it follows that the elements $((\be c)\psi) y^{k-1}$ and $((\be c) \psi) (y^{k-1})^{\omega+1}$ are $\mathscr R$-equivalent.
However, they are not $\mathscr L$-equivalent because otherwise they would be $\mathscr H$-equivalent and hence have the same idempotent power, as~$J$ is completely simple.  Consequently, $\mathsf{RLM}(T)$ is nontrivial and so it follows from Proposition~\ref{p:aug.pv} that $\lpbr\mathsf{RLM}(T) \rpbr = \lpbr\mathsf{RLM}(T) \rpbr \vee \bRZ$.
Also, if~$z$ denotes the image of~$y$ under the quotient map $T\to \mathsf{RLM}(T)$ and~$d$ denotes the image of~$f$ under this map, then $d^2=d$, $dz=z$, and $z^{k-1}$ is not a group element (as $z^{k-1}$ and $(z^{k-1})^{\omega+1}$ act differently on the $\mathscr L$-class of $(\be c)\psi$).
Thus Lemma~\ref{l:good.cyc.rz} implies that $\OO_k \in \lpbr\mathsf{RLM}(T)  \rpbr \vee \bRZ = \lpbr\mathsf{RLM}(T)\rpbr$.
Consequently,  $O_k^\ba\in \lpbr\mathsf{RLM}(T)\rpbr^\ba = \lpbr \mathsf{RLM}(T) \rpbr \subseteq \lpbr T\rpbr$.

(ii) This holds because for each $k \geq 2$, the semigroup $O_k^\ba$ satisfies the identity $x^{k+1} \id x^k$ but violates the identity $x^k \id x^{k-1}$.
\end{proof}

\subsection{A sufficient condition for the {\jirr} of groups} \label{sub: Bergman}
Recall that a normal subgroup~$N$ of a group~$G$ \textit{splits} if there exists a subgroup~$K$ of~$G$ so that $N \cap K = \{1\}$ and $NK = G$.

\begin{theorem}[G.\,M. Bergman, private communication, 2014] \label{T: Bergman}
Suppose that~$G$ is any finite {\sdi} group with an abelian monolith~$N$ that splits\up.
Then~$G$ is {\ji}\up.
\end{theorem}

\begin{proof}
By assumption, there exists a subgroup~$K$ of~$G$ with $N \cap K = \{1\}$ and $NK = G$.
Seeking a contradiction, suppose there exist finite groups~$G_1$ and~$G_2$ and some surjective homomorphism~$f$ from a subgroup~$H$ of $G_1 \times G_2$ onto~$G$ such that $G \notin \lpbr G_1\rpbr$ and $G \notin \lpbr G_2 \rpbr$.
Clearly we can assume that $H \pi_j = G_j$ for the projection maps $\pi_j : G_1 \times G_2 \twoheadrightarrow G_j$, that is, $H$ is a subdirect product of~$G_1$ and~$G_2$.
Further, we may assume that~$G_1$, $G_2$, and~$H$ are chosen so that the order of~$H$ is minimal.

Let $H_2 = \{ h_2 \in G_2 \mid (1,h_2) \in H \} \cong \ker(\pi_1) \cap H$.
If~$H_2$ is trivial, then $\pi_1$ is injective on~$H$, so that $H \cong G_1$, whence the contradiction $G \prec G_1$ is obtained.
Hence~$H_2$ is nontrivial.
Observe that
\begin{enumerate}[\ (a)]
\item[(\dag)] if~$L$ is a subgroup of~$H_2$ such that $\{1\} \times L \unlhd H$, then $L \unlhd G_2$;
\end{enumerate}
in particular, $H_2\unlhd G_2$.
Indeed, if $\ell \in L$ and $g_2 \in G_2$, then choosing any $g_1 \in G_1$ with $(g_1,g_2) \in H$, we have
\[
(1,g_2 \ell g_2^{-1}) = (g_1,g_2)(1,\ell)(g_1,g_2)^{-1} \in \{1\} \times L
\]
by normality of $\{1\} \times L$ in~$H$, whence $g_2 \ell g_2^{-1} \in L$.

Suppose that $\ker (f)$ has nontrivial intersection with the subgroup $\{1\} \times H_2$ of~$H$, say $\ker(f) \cap (\{1\} \cap H_2) = \{1\} \times L$ for some $L \subseteq G_2$.
Then~$L$ is normal in~$H_2$ and so also normal in~$G_2$ by~(\dag).
By dividing~$G_2$ by this intersection, we could contradictorily decrease the order of~$H$.
Therefore, $\ker (f)$ intersects $\{1\} \times H_2$ trivially.

Similarly, defining $H_1 = \{ h_1 \in G_1 \mid (h_1,1) \in H \}$, we have $\{1\} \neq H_1 \unlhd G_1$ and $\ker(f)$ intersects $H_1 \times \{1\}$ trivially.
Then $H_1 \times \{1\}$, $\{1\} \times H_2$, and $\ker(f)$ are all normal in~$H$ and have pairwise trivial intersections.

Note that the centralizer of~$N$ in~$G$ is~$N$.
Indeed, since~$N$ is the unique minimal normal subgroup of~$G$, the action of~$K$ on~$N$ by conjugation is faithful (otherwise, the kernel would be a normal subgroup of $G$ not containing $N$).
If~$kn$ centralizes~$N$ with $k \in K$ and $n \in N$, then since~$N$ is abelian, we have that~$k$ centralizes~$N$ and hence $k=1$ by the previous observation.

From now on, identify~$H_1$ with $H_1 \times\{1\}$ and~$H_2$ with $\{1\} \times H_2$.
Then~$H_1$ and~$H_2$ are normal in~$H$ and commute elementwise.
We claim now that $H_1f = N = H_2f$.
Indeed, since~$f$ is injective on each of these subgroups and these subgroups are normal in~$H$, we conclude that~$N$ is contained in $H_1 f \cap H_2 f$.
Since~$H_1$ and~$H_2$ commute elementwise, both $H_1f$ and $H_2f$ are contained in the centralizer of~$N$, which is~$N$.
We conclude that $H_1f = N = H_2f$ and~$f$ restricts to an isomorphism of~$H_1$ and~$H_2$ with~$N$.

Let $H^* = Kf^{-1}$.
Then since $N \cap K =\{1\}$, it follows that $H^* \cap H_1$ is a subgroup of $\ker(f)$.
But $\ker(f) \cap H_1$ is trivial, so that $H^* \cap H_1 = \{ 1 \}$.
Similarly, $H^* \cap H_2 = \{ 1 \}$.
Note that $H^*H_1$ and $H^*H_2$ are subgroups of~$H$ because~$H_1$ and~$H_2$ are normal.
Also $(H^*H_1)f=KN=G=(H^*H_2)f$ and so by minimality of~$H$, we have $H^*H_1 = H = H^*H_2$.
In particular, $G_2 \cong H \pi_2 = (H^*H_1)\pi_2 = H^*\pi_2$ and so, since $H^* \cap H_1 = \{1\}$, we deduce that $G_2 \cong H^*$.
Similarly, $G_1 \cong H^*$.
Therefore, $G \prec G_1 \times G_1$ and so $G \in \lpbr G_1 \rpbr$, a contradiction.
\end{proof}

\subsection{Join irre\-duci\-bility of the {\pvar} $\lpbr D_4\rpbr = \lpbr Q_8\rpbr$} \label{sub: D4 Q8}

\begin{proposition} \label{P: D4 Q8}
The {\pvar} $\lpbr D_4\rpbr = \lpbr Q_8\rpbr$ is {\ji}\up.
Further\up, the {\pvar} $\excl(Q_8)$ is the class of all finite semigroups whose $2$-subgroups are abelian\up, that is\up, finite semigroups whose maximal subgroups have abelian $2$-Sylow subgroups\up.
\end{proposition}

\begin{proof}
Since the variety of $2$-groups is saturated, it follows that the finite semigroups whose $2$-subgroups are abelian form a {\pvar}.
To complete the proof, it suffices to observe that the {\pvar} generated by any finite non-abelian $2$-group contains $\lpbr Q_8\rpbr$.
A proof can be found in Almeida~\cite[Theorem~4.5]{Alm86}, based on the classification of finite $2$-groups whose proper subgroups are abelian, going back to Miller and Moreno~\cite{MM03}; we are indebted to the anonymous reviewer for pointing this out.
Kearnes~\cite{Kea18} gave a direct proof that any finite non-abelian $2$-group generates a variety containing $\lpbr Q_8\rpbr$ via a general description of identity bases for finite nilpotent groups of class~$2$.
\end{proof}

Since the group $Q_8$ is $2$-generated as a semigroup, the {\pvar} $\excl(Q_8)$ can be defined by a {\pid} over two variables \cite[Proposition~7.1.9]{RS09}.
We proceed to describe such a {\pid}.
Let $\widehat{F}=\wh{\{x,y\}^+}$ be the free profinite semigroup over $\{x,y\}$ and let~$\eta$ be an idempotent in the minimal ideal of $\widehat{F}$.
Then $\eta \widehat{F}\eta$ is a profinite group and maps onto the free profinite group on two generators under the natural projection.
The elements $x'=\eta x\eta$ and $y'=\eta y\eta$ map onto the free generators and thus freely topologically generate a free profinite subgroup of $\widehat{F}$, which is a retract.
This observation was first made by Almeida and Volkov~\cite{AV03}.
If $\bu(x,y)$ and $\bv(x,y)$ are elements of the free profinite group $\wh{F}_\bG(\{x,y\})$ over $\{x,y\}$, then $\bu(x',y'),\bv(x',y')\in \eta \widehat{F}\eta$ and it is easy to see that the {\pid} $\bu(x',y') \id \bv(x',y')$ defines the {\pvar} of all semigroups whose maximal subgroups belong to the {\pvar} defined by $\bu(x,y)\id\bv(x,y)$; see~\cite{AV03} for details.
Thus, it suffices to find a two-variable group {\pid} $\bu(x,y) \id \bv(x,y)$ defining the {\pvar} of groups with abelian $2$-Sylow subgroups (or, equivalently, $2$-subgroups).

Let $\wh{F}_{\bG_2}(\{x,y\})$ be the free pro-$2$ group over $\{x,y\}$.
We have a natural continuous surjection $\pi\colon \wh{F}_\bG(\{x,y\})\to \wh{F}_{\bG_2}(\{x,y\})$.
The {\pvar} of $2$-groups is saturated and so by Ribes and Zalesskii~\cite[Proposition~7.6.7]{RZ10}, the group $\wh{F}_{\bG_2}(\{x,y\})$ is a projective profinite group.
Therefore, there is a continuous splitting of $\pi$, that is, we can find $\bw_1(x,y),\bw_2(x,y)\in \wh{F}_\bG(\{x,y\})$ which freely topologically generate a free pro-$2$ subgroup with $\pi(\bw_1(x,y))=\pi(x)$ and $\pi(\bw_2(x,y))=\pi(y)$; in other words, $\bw_1(x,y)$ and $\bw_2(x,y)$ topologically generate a free pro-$2$ retract of $\wh{F}_\bG(\{x,y\})$.
Then the {\pvar} of groups with abelian $2$-subgroups is defined by the {\pid} $\bw_1(x,y)\bw_2(x,y) \id \bw_2(x,y)\bw_1(x,y)$.
Thus the {\pvar} $\excl(Q_8)$ is defined by $\bw_1(x',y')\bw_2(x',y') \id \bw_2(x',y')\bw_1(x',y')$, where~$x'$ and~$y'$ are as given in the previous paragraph.

\section{Join {\irr} {\pvars}} \label{sec: ji}

The present section contains 15~subsections.
Some background results are recorded in the first subsection, while the latter 14~subsections are devoted to the {\pvars} generated by the following 14~semigroups:
\begin{equation}
\begin{gathered} \label{D: ji sgps}
\Z_{p^n}, \quad \ZB, \quad \N_n, \quad \N_n^I, \quad \NB, \quad \NBI, \\ \LZ, \quad \LZ^I, \quad \LZB, \quad \Az, \quad \Az^I, \quad \At, \quad \Bt, \quad \elB.
\end{gathered}
\end{equation}
Each subsection that is concerned with a semigroup~$S$ from~\eqref{D: ji sgps} begins with a theorem that establishes the {\ji} property of~$\lpbr S \rpbr$ by exhibiting a {\pid} that defines the {\pvar} $\excl(S)$.
A basis~$\Sigma_S$ of identities for the {\pvar} $\lpbr S \rpbr$ and an identity~$\varepsilon_S$ that defines its maximal sub{\pvar} $\lpbr S \rpbr \cap \excl(S)$ are then given in a proposition.
The pair $(\Sigma_S,\varepsilon_S)$ can be used to easily test if a finite semigroup generates the {\ji} {\pvar} $\lpbr S \rpbr$.
Indeed, for any finite semigroup~$T$,
\begin{align*}
T \models \Sigma_S \ \text{ and } \ T \not\models \varepsilon_S & \quad \Longleftrightarrow \quad \lpbr T \rpbr \subseteq \lpbr S \rpbr \ \text{ and } \ \lpbr T \rpbr \nsubseteq \lpbr S \rpbr \cap \excl(S) \\
& \quad \Longleftrightarrow \quad \lpbr T \rpbr = \lpbr S \rpbr.
\end{align*}
The pairs $(\Sigma_S,\varepsilon_S)$, where~$S$ ranges over the semigroups from~\eqref{D: ji sgps}, will be used in Section~\ref{sec: proof} to locate all {\ji} {\pvars} generated by a semigroup of order up to five.

\subsection{Preliminaries} \label{sub: prelim}

The free semigroup and free monoid over a countably infinite alphabet~$\A$ are denoted by~$\A^+$ and~$\A^*$, respectively.
Elements of~$\A$ are called \textit{variables} while elements of $\A^*$ are called \textit{words}.
For any word $\bw \in \A^+$,
\begin{enumerate}[1.]
\item[$\bullet$] the number of times a variable~$x$ occurs in~$\bw$ is denoted by $\occ(x,\bw)$;
\item[$\bullet$] the \textit{content} of~$\bw$, denoted by $\cont(\bw)$, is the set of variables occurring in~$\bw$, that is, $\cont(\bw) = \{ x \in \A \mid \occ(x,\bw) \geq 1\}$;
\item[$\bullet$] the \textit{initial part} of~$\bw$, denoted by $\ini(\bw)$, is the word obtained by retaining the first occurrence of each variable in~$\bw$;
\item[$\bullet$] the \textit{final part} of~$\bw$, denoted by $\fin(\bw)$, is the word obtained by retaining the last occurrence of each variable in~$\bw$.
\end{enumerate}

\begin{lemma} \label{L: word Zn NnI LZI RZI}
Let $\bu \id \bv$ be any semigroup identity\up.
Then
\begin{enumerate}[\rm(i)]
\item $\Z_n \models \bu \id \bv$ if and only if $\occ(x,\bu) \equiv \occ(x,\bv) \pmod n$ for all $x \in \A$\up;
\item $\N_n^I \models \bu \id \bv$ if and only if for all $x \in \A$\up, either $\occ(x,\bu) = \occ(x,\bv)$ or $\occ(x,\bu), \occ(x,\bv) \geq n$\up;
\item $\LZ^I \models \bu \id \bv$ if and only if $\ini(\bu) = \ini(\bv)$\up;
\item $\RZ^I \models \bu \id \bv$ if and only if $\fin(\bu) = \fin(\bv)$\up.
\end{enumerate}
\end{lemma}

\begin{proof}
These results are well known and easily established.
For instance, parts~(i) and~(ii) follow from Almeida~\cite[Lemma~6.1.4]{Alm94} while parts~(iii) and~(iv) can be found in Petrich and Reilly~\cite[Theorem~V.1.9, parts~(viii) and~(ix)]{PR99}.
\end{proof}

The \textit{local} of a {\pvar}~$\bV$, denoted by $\mathbb{L}\bV$, is the {\pvar} of all finite semigroups~$S$ such that $eSe \in \bV$ for any idempotent $e \in S$.

\begin{lemma}[Almeida~{\cite[Exercise~10.10.1]{Alm94}}] \label{L: local}
Let~$S$ be any finite semigroup that is not a monoid\up.
If the {\pvar} $\lpbr S \rpbr$ is {\ji}\up, then the {\pvar} $\lpbr S^I\rpbr$ is also {\ji} and $\excl(S^I) = \mathbb{L} \excl(S)$\up.
\end{lemma}

\subsection{The {\pvar} $\lpbr\Z_n\rpbr$} \label{sub: Zn ji}

For any set $\pi=\{p_1,p_2,p_3,\ldots\}$ of primes, let~$\pi'$ denote the set of primes complementary to~$\pi$.
If~$p$ is a prime, then simply write~$p'$ instead of~$\{p\}'$.
For example, $2'$ denotes the set of odd primes.
Retaining the above notation, recall that in~$\wh{\{x\}^+}$, the sequence $x^{(p_1p_2\cdots p_n)^{n!}}$ converges to an element (independent of the enumeration of~$\pi$), denoted by~$x^{\pi^{\omega}}$, with the following property: if~$s$ is an element of a finite semigroup~$S$, then~$s^{\pi^{\omega}}$ is a generator of the $\pi'$-primary component of the finite cyclic group generated by~$s^{\omega+1}$.
Here we recall that for a finite abelian group~$A$, the \textit{$\pi'$-primary component} of~$A$ is the direct product of the $p$-Sylow subgroups of~$A$ with $p \notin \pi$.
In this case, $s^{(\pi')^{\omega}}$ will then be a generator of the $\pi$-primary component of $\langle s^{\omega+1}\rangle$; see~\cite[Proposition~7.1.16]{RS09}.

\begin{theorem}\label{T: Zn}
For any prime~$p$ with $n \geq 1$\up, the {\pvar} $\lpbr\Z_{p^n}\rpbr$ is {\ji} and
\begin{equation}
\excl(\Z_{p^n}) = \lbr (x^{(p')^\omega})^{p^{n-1}} \id x^{\omega} \rbr. \label{id: Zpn excl}
\end{equation}
\end{theorem}

\begin{proof}
The cyclic group $\Z_{p^n} = \langle \eg \mid \eg^{p^n} = \ei \rangle$ violates the {\pid} in~\eqref{id: Zpn excl} because $(\eg^{(p')^\omega})^{p^{n-1}} = \eg^{p^{n-1}} \neq \ei = \eg^{\omega}$.
Therefore, if~$\Z_{p^n}$ belongs to some {\pvar}~$\bV$, then~$\bV$ violates the {\pid} in~\eqref{id: Zpn excl}.

Conversely, suppose that the {\pid} in~\eqref{id: Zpn excl} is violated by~$\bV$, say it is violated by $S \in \bV$.
Generality is not lost by assuming that~$S$ is generated by an element~$s$ such that $(s^{(p')^\omega})^{p^{n-1}} \neq s^\omega$.
Replacing~$s$ by~$s^{\omega+1}$, we may assume that~$S$ is, in fact, a cyclic group generated by~$s$ such that $(s^{(p')^\omega})^{p^{n-1}} \neq 1$.
But then the $p$-primary component of~$S$ is a cyclic group of order~$p^m$ with $m\geq n$.
Therefore, $\Z_{p^n}$ divides~$S$, whence $\Z_{p^n} \in \bV$.
\end{proof}

\begin{proposition} \label{P: Zn}
Let $n \geq 1$\up.
\begin{enumerate}[\em(i)]
\item The identities satisfied by the group $\Z_n$ are axiomatized by
\[
xy \id yx, \quad x^ny \id y.
\]
\item The maximal sub{\pvars} of $\lpbr\Z_n\rpbr$ are precisely $\lpbr\Z_d\rpbr$\up, where~$d$ ranges over all maximal proper divisors of~$n$\up.
Consequently\up, for any prime~$p$ with $k \geq 1$\up, the sub{\pvar} of $\lpbr\Z_{p^k}\rpbr$ defined by
\[
x^{p^{k-1}+1} \id x
\]
is the unique maximal sub{\pvar} of $\lpbr\Z_{p^k}\rpbr$\up.
\end{enumerate}
\end{proposition}

\begin{proof}
These results are well known and easily established.
For instance, part~(i) follows from Almeida~\cite[Corollary~6.1.5]{Alm94} while part~(ii) follows from Petrich and Reilly~\cite[Lemma~VIII.6.14]{PR99}.
\end{proof}

\subsection{The {\pvar} $\lpbr\ZB\rpbr$} \label{sub: ZB ji}

\begin{theorem} \label{T: Z2bar}
The {\pvar}~$\lpbr\ZB\rpbr$ is {\ji} and
\[
\excl(\ZB) = \lbr (\be yx^{(2')^\omega})^\omega \id (\be yx^\omega)^\omega \rbr,
\]
where~$\be$ is an idempotent in the minimal ideal of $\wh{\{x,y\}^+}$.
\end{theorem}

\begin{proof}
This follows from Theorems~\ref{T: Sbar} and~\ref{T: Zn}.
\end{proof}

Alternately, Rhodes and Steinberg~\cite[Example~7.3.20]{RS09} have shown that
\[
\excl(\ZB) = \lbr ((x^\omega \be x^\omega)^\omega x^{(2')^\omega})^\omega \id (x^\omega \be x^\omega)^\omega \rbr,
\]
where~$\be$ is an idempotent in the minimal ideal of $\wh{\{x,y\}^+}$.

\begin{proposition} \label{P: Z2bar}
\quad
\begin{enumerate}[\rm(i)]
\item The identities satisfied by the semigroup $\ZB$ are axiomatized by
\[
x^3 \id x, \quad xyxy \id yx^2y.
\]
\item The sub{\pvar} of $\lpbr\ZB\rpbr$ defined by the identity
\[
xyx \id yx^2
\]
is the unique maximal sub{\pvar} of $\lpbr\ZB\rpbr$\up.
\end{enumerate}
\end{proposition}

\begin{proof}
This follows from the dual of Tishchenko~\cite[Proposition~3.16]{Tis07}, where the variety generated by~$\ZBop$ is denoted by $\mathbf{W}_2$.
\end{proof}

\subsection{The {\pvar} $\lpbr\N_n\rpbr$} \label{sub: Nn ji}

\begin{theorem} \label{T: Nn}
For each $n \geq 2$\up, the {\pvar} $\lpbr\N_n\rpbr$ is {\ji} and
\begin{equation}
\excl(\N_n) = \lbr x^{\omega+n-1} \id x^{n-1} \rbr. \label{id: Nn excl}
\end{equation}
\end{theorem}

\begin{proof}
The semigroup $\N_n = \langle \ea \mid \ea^n = \ez \rangle$ violates the {\pid} in~\eqref{id: Nn excl} because $\ea^{\omega+n-1} = \ez \neq \ea^{n-1}$.
Therefore, if~$\N_n$ belongs to some {\pvar}~$\bV$, then~$\bV$ violates the {\pid} in~\eqref{id: Nn excl}.

Conversely, suppose that the {\pid} in~\eqref{id: Nn excl} is violated by the {\pvar}~$\bV$, say it is violated by $S \in \bV$.
Then there exists some $a \in S$ such that $a^{\omega+n-1} \neq a^{n-1}$.
If there exist some $i \leq n-1$ and some $j>i$ such that $a^i = a^j$, then $a^{n-1} = a^{n-1-i} a^i = a^{n-1-i} a^j = a^{n-1} a^{j-i}$, so that
\[
a^{n-1} = a^{n-1} a^{j-i} = a^{n-1} a^{2(j-i)} = \cdots = a^{n-1} a^{\omega(j-i)} = a^{n-1+\omega},
\]
which is a contradiction.
Hence the sets $\{ a \}, \{a^2\}, \ldots, \{a^{n-1} \}, \{ a^i \mid i \geq n \}$ are pairwise disjoint.
It follows that $J = \{ a^i \mid i \geq n \}$ is an ideal of the monogenic subsemigroup~$\langle a \rangle$ of~$S$ such that $\langle a\rangle / J \cong \N_n$.
Consequently, $\N_n \in \lpbr S \rpbr \subseteq \bV$.
\end{proof}

For each nonnegative real number~$x$, let $\lfloor x\rfloor$ denote the greatest integer bounded from above by~$x$.

\begin{proposition} \label{P: Nn}
Let $n \geq 2$\up.
\begin{enumerate}[\rm(i)]
\item The identities satisfied by the semigroup $\N_n$ are axiomatized by
\begin{gather*}
xy \id yx, \quad x^n \id y_1 y_2 \cdots y_n, \\
x^{k+1} y^k z_1 z_2 \cdots z_{n-3k-1} \id x^k y^{k+1} z_1 z_2 \cdots z_{n-3k-1},
\end{gather*}
where $k = 0,1,\ldots, \lfloor(n-1)/3\rfloor$\up.

\item The sub{\pvar} of $\lpbr\N_n\rpbr$ defined by the identity \[ x^n \id x^{n-1} \] is the unique maximal sub{\pvar} of $\lpbr\N_n\rpbr$\up.
\end{enumerate}
\end{proposition}

\begin{proof}
(i) This follows from Shevrin and Volkov \cite[Proposition~21.3]{SV85}.

(ii) This follows from Theorem~\ref{T: Nn} and part~(i).
\end{proof}

\subsection{The {\pvar} $\lpbr\N_n^I\rpbr$} \label{sub: NnI ji}

\begin{theorem} \label{T: NnI}
For any $n \geq 1$\up, the {\pvar} $\lpbr\N_n^I\rpbr$ is {\ji} and \[ \excl(\N_n^I) = \mathbb{L}\excl(\N_n) = \lbr h^\omega (xh^\omega)^{\omega+n-1} \id h^\omega (xh^\omega)^{n-1} \rbr. \]
\end{theorem}

\begin{proof}
For $n=1$, the result follows from~\cite[Table~7.2]{RS09} because $\N_1^I \cong \SL$.
For $n \geq 2$, the result follows from Lemma~\ref{L: local} and Theorem~\ref{T: Nn}.
\end{proof}

\begin{proposition} \label{P: NnI}
Let $n \geq 1$\up.
\begin{enumerate}[\rm(i)]
\item The identities satisfied by the semigroup $\N_n^I$ are axiomatized by \[ x^{n+1} \id x^n, \quad xy \id yx. \]
\item The sub{\pvar} of $\lpbr\N_n^I\rpbr$ defined by the identity \[ x^ny^{n-1} \id x^{n-1}y^n \] is the unique maximal sub{\pvar} of $\lpbr\N_n^I\rpbr$\up.
\end{enumerate}
\end{proposition}

\begin{proof}
(i) This easily established result is well known; see, for example, Almeida \cite[Corollary~6.1.5]{Alm94}.

(ii) This follows from Theorem~\ref{T: NnI} and part~(i).
\end{proof}

\subsection{The {\pvar} $\lpbr\NB\rpbr$} \label{sub: NB ji}

\begin{theorem} \label{T: N2bar}
The {\pvar}~$\lpbr\NB\rpbr$ is {\ji} and \[ \excl(\NB) = \lbr(\be zx^{\omega+1})^\omega  \id (\be zx)^\omega\rbr, \] where~$\be$ is an idempotent from the minimal ideal of $\wh{\{x,z\}^+}$.
\end{theorem}

\begin{proof}
This follows from Theorems~\ref{T: Sbar} and~\ref{T: Nn}.
\end{proof}

\begin{proposition} \label{P: N2bar}
\quad
\begin{enumerate}[\rm(i)]
\item The identities satisfied by the semigroup $\NB$ are axiomatized by \[ xyz \id yz. \]
\item The sub\-{\pvar} of $\lpbr\NB\rpbr$ defined by the identity \[ xy \id y^2 \] is the unique maximal sub{\pvar} of $\lpbr\NB\rpbr$\up.
\end{enumerate}
\end{proposition}

\begin{proof}
(i) This follows from Tishchenko~\cite[Corollary~2.5(c) and Proposition~4.4]{Tis07}.

(ii) This follows from Tishchenko~\cite[Proposition~3.4]{Tis07}.
\end{proof}

\subsection{The {\pvar} $\lpbr\NBI\rpbr$} \label{sub: NBI ji}

\begin{theorem} \label{T: N2barI}
The {\pvar} $\lpbr\NBI\rpbr$ is {\ji} and \[ \excl(\NBI) = \mathbb{L}\excl(\NB) = \mathbb{L} \lbr (\be zx^{\omega+1})^\omega  \id (\be zx)^\omega \rbr, \] where~$\be$ is an idempotent from the minimal ideal of $\wh{\{x,z\}^+}$.
\end{theorem}

\begin{proof}
This follows from Lemma~\ref{L: local} and Theorem~\ref{T: N2bar}.
\end{proof}

\begin{proposition}[{Lee and Li~\cite[Corollary~6.6 and Lemma~6.7]{LL11}}] \label{P: N2barI}
\quad
\begin{enumerate}[\rm(i)]
\item The identities satisfied by the semigroup $\NBI$ are axiomatized by
\begin{gather*}
x^3 \id x^2, \quad x^2hx \id xhx, \quad xhx^2 \id hx^2, \\ xyxy \id yx^2y, \quad xyhxy \id yxhxy, \quad xyxty \id yx^2ty, \quad xyhxty \id yxhxty.
\end{gather*}
\item The sub{\pvar} of $\lpbr\NBI\rpbr$ defined by the identity
\[
xyxyh^2 \id x^2y^2h^2
\]
is the unique maximal sub{\pvar} of $\lpbr\NBI\rpbr$.
\end{enumerate}
\end{proposition}

\subsection{The {\pvar} $\lpbr\LZ\rpbr$} \label{sub: LZ ji}

\begin{theorem} \label{T: LZ}
The {\pvar} $\lpbr\LZ\rpbr$ is {\ji} and
\[
\excl(\LZ) = \lbr x^\omega (yx^\omega)^\omega \id (yx^\omega)^\omega \rbr.
\]
\end{theorem}

\begin{proof}
This result is dual to \cite[Proposition~10.10.2(b)]{Alm94}.
\end{proof}

\begin{proposition}[Rhodes and Steinberg~{\cite[Table~7.1]{RS09}}] \label{P: LZ}
\quad
\begin{enumerate}[\rm(i)]
\item The identities satisfied by the semigroup $\LZ$ are axiomatized by
\[
xy \id x.
\]
\item The {\pvar} $\lpbr\LZ\rpbr$ is an atom in the lattice~$\bPV$\up.
\end{enumerate}
\end{proposition}

\subsection{The {\pvar} $\lpbr\LZ^I\rpbr$} \label{sub: LZI ji}

\begin{theorem} \label{T: LZI}
The {\pvar} $\lpbr\LZ^I\rpbr$ is {\ji} and
\[
\excl(\LZ^I) = \mathbb{L} \excl(\LZ) = \lbr h^\omega (xh^\omega)^\omega (yh^\omega(xh^\omega)^\omega)^\omega \id h^\omega (yh^\omega(xh^\omega)^\omega)^\omega \rbr.
\]
\end{theorem}

\begin{proof}
This follows from Lemma~\ref{L: local} and Theorem~\ref{T: LZ}.
\end{proof}

\begin{proposition} \label{P: LZI}
\quad
\begin{enumerate}[\rm(i)]
\item The identities satisfied by the semigroup $\LZ^I$ are axiomatized by
\[
x^2 \id x, \quad xyx \id xy.
\]
\item The sub{\pvar} of $\lpbr\LZ^I\rpbr$ defined by the identity
\[
xyz \id xzy
\]
is the unique maximal sub{\pvar} of $\lpbr\LZ^I\rpbr$\up.
\end{enumerate}
\end{proposition}

\begin{proof}
This can be found in Almeida~\cite[Figure~5.1]{Alm94}, where the {\pvar} $\lpbr\LZ^I\rpbr$ is denoted by~$\mathbf{MK}_1$.
\end{proof}

\subsection{The {\pvar} $\lpbr\LZB\rpbr$} \label{sub: LZB ji}

\begin{theorem} \label{T: LZbar}
The {\pvar}~$\lpbr\LZB\rpbr$ is {\ji} and
\[
\excl(\LZB)= \lbr(\be zx^\omega(yx^\omega)^\omega)^\omega \id (\be z(yx^\omega)^\omega)^\omega \rbr,
\]
where~$\be$ is an idempotent in the minimal ideal of $\wh{\{x,y,z\}^+}$\up.
\end{theorem}

\begin{proof}
This follows from Theorems~\ref{T: Sbar} and~\ref{T: LZ}.
\end{proof}

Alternately, Rhodes and Steinberg~\cite[Example~7.3.16]{RS09} have shown that
\[
\excl(\LZB)= \lbr((\be z)^\omega x^\omega(yx^\omega)^\omega)^\omega \id ((\be z)^\omega (yx^\omega)^\omega)^\omega \rbr,
\]
where~$\be$ is an idempotent in the minimal ideal of $\wh{\{x,y,z\}^+}$.

\begin{proposition} \label{P: LZbar}
\quad
\begin{enumerate}[\rm(i)]
\item The identities satisfied by the semigroup $\LZB$ are axiomatized by
\[
x^2 \id x, \quad xyz \id xzxyz.
\]
\item The sub{\pvar} of $\lpbr\LZB\rpbr$ defined by the identity
\[
xyz \id xzyz
\]
is the unique maximal sub{\pvar} of $\lpbr\LZB\rpbr$\up.
\end{enumerate}
\end{proposition}

\begin{proof}
This can be found in Almeida~\cite[Figure~5.1]{Alm94}, where the {\pvar} $\lpbr\LZB\rpbr$ is denoted by $\lbrs R_3^\rho = Q_3^\rho\rbrs_\bB$.
\end{proof}

\subsection{The {\pvar} $\lpbr\Az\rpbr$} \label{sub: Az ji}

\begin{theorem}[Lee~{\cite[Proposition~2.3]{Lee17}}] \label{T: A0}
The {\pvar} $\lpbr\Az\rpbr$ is {\ji} and
\[
\excl(\Az) = \lbr (x^\omega y^\omega)^{\omega+1} \id x^\omega y^\omega \rbr.
\]
\end{theorem}

\begin{proposition}[{Lee~\cite[Section~4]{Lee04}, Lee and Volkov~\cite[Theorem~4.1]{LV07}}] \label{P: A0}
\quad
\begin{enumerate}[\rm(i)]
\item The identities satisfied by the semigroup $\Az$ are axiomatized by
\[
x^3 \id x^2, \quad x^2yx^2 \id yxy.
\]
\item The sub{\pvar} of $\lpbr\Az\rpbr$ defined by the identity
\[
x^2y^2 \id y^2x^2
\]
is the unique maximal sub{\pvar} of $\lpbr\Az\rpbr$\up.
\end{enumerate}
\end{proposition}

\subsection{The {\pvar} $\lpbr\Az^I\rpbr$} \label{sub: AzI ji}

\begin{theorem} \label{T: A0I}
The {\pvar} $\lpbr\Az^I\rpbr$ is {\ji} and
\[
\excl(\Az^I) = \mathbb{L}\excl(\Az) = \lbr h^\omega ((xh^\omega)^\omega (yh^\omega)^\omega)^{\omega+1} \id h^\omega (xh^\omega)^\omega (yh^\omega)^\omega \rbr.
\]
\end{theorem}

\begin{proof}
This follows from Lemma~\ref{L: local} and Theorem~\ref{T: A0}.
\end{proof}

\begin{proposition}[{Lee~\cite[Propositions~1.1 and~1.5(ii)]{Lee08}}] \label{P: A0I}
\quad
\begin{enumerate}[\rm(i)]
\item The identities satisfied by the semigroup $\Az^I$ are axiomatized by
\[
x^3 \id x^2, \quad x^2yx^2 \id xyx, \quad xyxy \id yxyx, \quad xyxzx \id xyzx.
\]
\item The sub{\pvar} of $\lpbr\Az^I\rpbr$ defined by the identity
\[
hx^2y^2h \id hy^2x^2h
\]
is the unique maximal sub{\pvar} of $\lpbr\Az^I\rpbr$\up.
\end{enumerate}
\end{proposition}

\subsection{The {\pvar} $\lpbr\At\rpbr$} \label{sub: At ji}

\begin{theorem}[Lee~\cite{Lee07a}] \label{T: A2}
The {\pvar} $\lpbr\At\rpbr$ is {\ji} and
\[
\excl(\At) = \lbr((x^\omega y)^\omega(yx^\omega)^\omega)^\omega \id (x^\omega yx^\omega)^\omega\rbr.
\]
\end{theorem}

\begin{proposition}[{Lee~\cite[Theorem~2.7]{Lee04}, Trahtman~\cite{Tra94}}] \label{P: A2}
\quad
\begin{enumerate}[\rm(i)]
\item The identities satisfied by the semigroup $\At$ are axiomatized by
\[
x^3 \id x^2, \quad xyxyx \id xyx, \quad xyxzx \id xzxyx.
\]
\item The sub{\pvar} of $\lpbr\At\rpbr$ defined by the identity
\[
x^2y^2x^2 \id x^2yx^2
\]
is the unique maximal sub{\pvar} of $\lpbr\At\rpbr$\up.
\end{enumerate}
\end{proposition}

\subsection{The {\pvar} $\lpbr\Bt\rpbr$} \label{sub: Bt ji}

\begin{theorem}[Rhodes and Steinberg~{\cite[Example~7.3.4]{RS09}}] \label{T: B2}
The {\pvar} $\lpbr\Bt\rpbr$ is {\ji} and
\[
\excl(\Bt) = \lbr((xy)^\omega(yx)^\omega (xy)^\omega)^\omega \id (xy)^\omega\rbr.
\]
\end{theorem}

\begin{proposition}[{Lee~\cite[Theorem~3.6]{Lee04}, Trahtman~\cite{Tra81}}] \label{P: B2}
\quad
\begin{enumerate}[\rm(i)]
\item The identities satisfied by the semigroup $\Bt$ are axiomatized by
\[
x^3 \id x^2, \quad xyxyx \id xyx, \quad x^2y^2 \id y^2x^2.
\]

\item The sub{\pvar} of $\lpbr\Bt\rpbr$ defined by the identity
\[
xy^2x \id xyx
\]
is the unique maximal sub{\pvar} of $\lpbr\Bt\rpbr$\up.
\end{enumerate}
\end{proposition}

\subsection{The {\pvar} $\lpbr\elB\rpbr$} \label{sub: elB ji}

\begin{theorem} \label{T: el3}
The {\pvar} $\lpbr\elB\rpbr$ is {\ji} and
\[
\excl(\elB) = \lbr (\be zx^\omega y)^\omega \id (\be z(x^\omega y)^{\omega+1})^\omega \rbr,
\]
where $\be$ is an idempotent in the minimal ideal of $\wh{\{x,y,z\}^+}$\up.
\end{theorem}

\begin{proof}
This is a special case of Theorem~\ref{T: Ok} since $\elB \cong \OO_2^\ba$.
\end{proof}

The remainder of this subsection is devoted to establishing a basis for the identities satisfied by $\elB$.
It turns out that it is notationally simpler to consider the dual semigroup $\elBop = \{ a,b,c,d,e\}$ given in Table~\ref{Tab: el3bar op}.

\begin{table}[ht!]
\[
\begin{array} [c]{r|ccccc}
\elBop \, & \, a & b & c & d & e \\ \hline
     a \, & \, a & a & a & a & a \\
     b \, & \, a & a & b & a & d \\
     c \, & \, a & a & c & a & e \\
     d \, & \, d & d & d & d & d \\
     e \, & \, e & e & e & e & e
\end{array}
\]
\caption{Multiplication table of $\elBop$} \label{Tab: el3bar op}
\end{table}

\begin{proposition} \label{P: el3}
\quad
\begin{enumerate}[\rm(i)]
\item The identities satisfied by the semigroup $\elBop$ are axiomatized by
\begin{equation}
xy^2 \id xy, \quad xyz \id xyzy. \label{id: el3 basis}
\end{equation}
\item The sub{\pvar} of $\lpbr\elBop\rpbr$ defined by the identity
\begin{equation}
xyzx \id xyxz \label{id: el3 max}
\end{equation}
is the unique maximal sub{\pvar} of $\lpbr\elBop\rpbr$\up.
\end{enumerate}
\end{proposition}

\begin{remark} \label{R: el3 id}
It is routinely shown that the semigroup~$\elBop$ satisfies the identities~\eqref{id: el3 basis} but violates the identity~\eqref{id: el3 max}.
\end{remark}

In this subsection, a word~$\bw$ is said to be in \textit{canonical form} if either
\begin{enumerate}[({CF}1)]
\item $\bw = x_0 x_1 \cdots x_m$ or
\item $\bw = x_0 x_1 \cdots x_k \cdot x_0 \cdot x_{k+1} x_{k+2} \cdots x_m$,
\end{enumerate}
where $x_0,x_1,\ldots,x_m$ are distinct variables with $0 \leq k \leq m$.

\begin{remark}
Note the extreme cases for the word~$\bw$ in~(CF2):
\begin{enumerate}[(i)]
\item if $0 = k= m$, then $\bw = x_0^2$;
\item if $0 = k < m$, then $\bw = x_0^2 x_1 \cdots x_m$;
\item if $0<k=m$, then $\bw = x_0 x_1 x_2 \cdots x_m x_0$.
\end{enumerate}
\end{remark}

\begin{lemma} \label{L: el3 can}
Given any word~$\bw$\up, the identities~\eqref{id: el3 basis} can be used to convert~$\bw$ into some word~$\bw'$ in canonical form with $\ini(\bw) = \ini(\bw')$\up.
\end{lemma}

\begin{proof}
Suppose that $\ini(\bw) = x_0 x_1 \cdots x_m$.
Then~$\bw$ can be written as
\[
\bw = \prod_{i=0}^m (x_i\bw_i) = x_0 \bw_0 x_1 \bw_1 \cdots x_m \bw_m,
\]
where $\bw_i \in \{x_0,x_1,\ldots,x_i\}^*$ for all~$i$.
The identities~\eqref{id: el3 basis} can be used to eliminate all occurrences of $x_1,x_2,\ldots,x_m$ from each~$\bw_i$, resulting in the word
\[
\bw' = \prod_{i=0}^m (x_ix_0^{e_i}) = x_0 x_0^{e_0} x_1 x_0^{e_1} \cdots x_m x_0^{e_m},
\]
where $e_0,e_1,\ldots,e_m \geq 0$.
If $e_0=e_1=\cdots=e_m = 0$, then the word~$\bw'$ is in canonical form~(CF1) such that $\ini(\bw) = \ini(\bw')$.
If $k \geq 0$ is the least index such that $e_k \geq 1$, then $e_0 = e_1 = \cdots = e_{k-1}=0$, so that
\[
\bw' = \bigg(\prod_{i=0}^{k-1} x_i\bigg) x_kx_0^{e_k} \bigg(\prod_{i=k+1}^m (x_ix_0^{e_i})\bigg) \stackrel{\eqref{id: el3 basis}}{\id} \underbrace{\bigg(\prod_{i=0}^{k-1} x_i\bigg) x_kx_0 \bigg(\prod_{i=k+1}^m x_i\bigg)}_{\bw''}.
\]
The word~$\bw''$ is in canonical form~(CF2) with $\ini(\bw) = \ini(\bw'')$.
\end{proof}

\begin{proof}[Proof of Proposition~\ref{P: el3}(ii)]
As observed in Remark~\ref{R: el3 id}, the semigroup $\elBop$ violates the identity~\eqref{id: el3 max}.
Hence $\lpbr \elBop \rpbr \cap \lbrs \eqref{id: el3 max} \rbrs$ is a proper sub{\pvar} of $\lpbr\elBop\rpbr$.
It remains to show that each proper sub{\pvar}~$\bV$ of $\lpbr\elBop\rpbr$ satisfies the identity~\eqref{id: el3 max}.
Since $\bV \neq \lpbr \elBop \rpbr$, there exists an identity $\bu \id \bv$ of~$\bV$ that is violated by $\elBop$.
Further, since the identities~\eqref{id: el3 basis} are satisfied by~$\elBop$ and so also by~$\mathbf{V}$, it follows from Lemma~\ref{L: el3 can} that the words~$\bu$ and~$\bv$ can be chosen to be in canonical form.
There are two cases.

\paragraph{\sc Case~1}

$\ini(\bu) \neq \ini(\bv)$.
Then by Theorem~\ref{T: LZI}, the {\pvar}~$\bV$ satisfies the {\pid} that defines $\excl(\LZ^I)$.
Since
\begin{align*}
h^\omega (xh^\omega)^\omega (yh^\omega(xh^\omega)^\omega)^\omega \stackrel{\eqref{id: el3 basis}}{\id} h^2xy \quad \text{and} \quad h^\omega (yh^\omega(xh^\omega)^\omega)^\omega \stackrel{\eqref{id: el3 basis}}{\id} h^2yx,
\end{align*}
the {\pvar}~$\bV$ satisfies the identity $\alpha: h^2xy \id h^2yx$.
Since
\[
xyxz \stackrel{\eqref{id: el3 basis}}{\id} xy^2xz \stackrel{\alpha}{\id} xy^2zx \stackrel{\eqref{id: el3 basis}}{\id} xyzx,
\]
the {\pvar}~$\bV$ satisfies the identity~\eqref{id: el3 max}.

\paragraph{\sc Case~2}

$\ini(\bu) = \ini(\bv)$ and $\bu \neq \bv$.
If the words~$\bu$ and~$\bv$ are both of the form~(CF1), then they are contradictorily equal.
Hence either~$\bu$ or~$\bv$ is of the form~(CF2).
By symmetry, there are two subcases.
\begin{enumerate}[{2.}1.]
\item $\bu$ and~$\bv$ are both of the form~(CF2). Then
\begin{align*}
\bu & = x_0 x_1 \cdots x_j \cdot x_0 \cdot x_{j+1} x_{j+2} \cdots x_m \\
\text{and} \quad \bv & = x_0 x_1 \cdots x_k \cdot x_0 \cdot x_{k+1} x_{k+2} \cdots x_m,
\end{align*}
where $0 \leq j,k \leq m$.
Since $j \neq k$, it suffices to assume by symmetry that $0 \leq j < k \leq m$.
Let~$\varphi$ denote the substitution given by $x_0 \mapsto xy$, $x_i \mapsto y$ for all $i \in \{ 1,2,\ldots,j\}$, and $x_i \mapsto z$ otherwise.
Then
\begin{align*}
\bu \varphi & = x_0 \varphi \cdot (x_1 \cdots x_j) \varphi \cdot x_0 \varphi \cdot (x_{j+1} x_{j+2} \cdots x_m) \varphi \\
            & = xy \cdot y^j \cdot xy \cdot z^{m-j} \stackrel{\eqref{id: el3 basis}}{\id} xyxz \quad \text{and} \\
\bv \varphi & = x_0 \varphi \cdot (x_1 \cdots x_j) \varphi \cdot (x_{j+1} x_{j+1} \cdots x_k) \varphi \cdot x_0 \varphi \cdot (x_{k+1} x_{j+2} \cdots x_m) \varphi \\
            & = xy \cdot y^j \cdot z^{k-j} \cdot xy \cdot z^{m-k} \stackrel{\eqref{id: el3 basis}}{\id} xyzx.
\end{align*}
Therefore, the identity~\eqref{id: el3 max} is deducible from~\eqref{id: el3 basis} and $\bu \id \bv$.
The {\pvar}~$\bV$ thus satisfies the identity~\eqref{id: el3 max}.
\item $\bu$ is of the form~(CF1) while~$\bv$ is of the form~(CF2).
Then
\[
\bu=x_0 x_1 \cdots x_m \quad \text{and} \quad \bv = x_0 x_1 \cdots x_j \cdot x_0 \cdot x_{j+1} x_{j+2} \cdots x_m.
\]
Since
\begin{align*}
\bu x_{m+1} x_0 & = \overbrace{x_0 x_1 \cdots x_m x_{m+1} x_0}^{\bu'} \\
\text{and} \quad \bv x_{m+1} x_0 & \stackrel{\eqref{id: el3 basis}}{\id} \underbrace{x_0 x_1 \cdots x_j \cdot x_0 \cdot x_{j+1} x_{j+2} \cdots x_m x_{m+1}}_{\bv'},
\end{align*}
the {\pvar}~$\bV$ satisfies the identity $\bu' \id \bv'$.
Now~$\bu'$ and~$\bv'$ are distinct words in canonical form~(CF2) such that $\ini(\bu') = \ini(\bv')$.
Thus the arguments in Subcase~2.1 can be repeated to show that~$\bV$ satisfies the identity~\eqref{id: el3 max}.
\qedhere
\end{enumerate}
\end{proof}

\begin{proof}[Proof of Proposition~\ref{P: el3}(i)]
As noted in Remark~\ref{R: el3 id}, the identities~\eqref{id: el3 basis} are satisfied by the semigroup~$\elBop$.
Conversely, suppose that $\bu \id \bv$ is any identity satisfied by~$\elBop$.
By Lemma~\ref{L: el3 can}, the identities~\eqref{id: el3 basis} can be used to convert~$\bu$ and~$\bv$ into words~$\bu'$ and $\bv'$ in canonical form.
Since the subsemigroup $\{a,c,e\}$ of~$\elBop$ and the semigroup~$\LZ^I$ are isomorphic, it follows from Lemma~\ref{L: word Zn NnI LZI RZI}(iii) that $\ini(\bu') = \ini(\bv')$.
Suppose that $\bu' \neq \bv'$.
Then by repeating the arguments in Case~2 of the proof of Proposition~\ref{P: el3}(ii), the identity~\eqref{id: el3 max} is deducible from~\eqref{id: el3 basis} and $\bu' \id \bv'$.
Since the semigroup~$\elBop$ satisfies the identities~\eqref{id: el3 basis} and $\bu' \id \bv'$, it also satisfies~\eqref{id: el3 max}; but this is impossible by Remark~\ref{R: el3 id}.
Therefore, $\bu'=\bv'$.
Since
\[
\bu \stackrel{\eqref{id: el3 basis}}{\id} \bu'=\bv' \stackrel{\eqref{id: el3 basis}}{\id} \bv,
\]
the identity $\bu \id \bv$ is deducible from~\eqref{id: el3 basis}.
\end{proof}

\section{Non-{\ji} {\pvars}} \label{sec: non-ji}

This section contains nine subsections, each of which establishes one or more sufficient conditions for a finite semigroup to generate a non-{\ji} {\pvar}.
Each of these sufficient conditions, given as a corollary of some general result, presents some finite set~$\Sigma$ of identities and some identities $\varepsilon_1,\varepsilon_2,\ldots,\varepsilon_k$ with the property that for any finite semigroup~$S$,
\[
S \models \Sigma \ \text{ and } \ S \not\models \varepsilon_i \ \text{ for all~$i$} \quad \Longrightarrow \quad \text{$\lpbr S\rpbr$ is not {\ji}}.
\]
In most cases, $\Sigma$ will be a basis of identities for some join $\bV = \bigvee_{i=1}^k \bV_i$ of compact {\pvars} $\bV_1, \bV_2,\ldots,\bV_k$ that satisfy the {\pids} $\varepsilon_1,\varepsilon_2,\ldots,\varepsilon_k$, respectively.

Sufficient conditions developed in this section will be used in Section~\ref{sec: proof} to locate all non-{\ji} {\pvars} generated by a semigroup of order up to five.

\subsection{The {\pvar} $\lpbr\Z_3,\Z_4,\ZB,\ZBop,\N_3^I\rpbr$}

In this subsection, it is convenient to write
\[
\mathscr{K} = \{\Z_3,\Z_4,\ZB,\ZBop,\N_3^I\}.
\]

\begin{proposition}[Lee and Li~\cite{LL15}] \label{P: M4}
The identities satisfied by the semigroup $\Z_3 \times \Z_4 \times \ZB \times \ZBop \times \N_3^I$ are axiomatized by
\begin{equation} \label{id: M4 basis}
\begin{gathered}
x^{15} \id x^3, \quad x^{14}hx \id x^2hx, \quad x^{13}hx^2 \id xhx^2, \quad x^{13} hxtx \id xhxtx, \\
x^3hx \id xhx^3, \quad xhx^2tx \id x^3htx, \\
xhx^2y^2ty \id xhy^2x^2ty, \\
xhykxytxdy \id xhykyxtxdy, \quad xhykxytydx \id xhykyxtydx.
\end{gathered}
\end{equation}
\end{proposition}

\begin{corollary} \label{C: M4v3}
Suppose that~$S$ is any finite semigroup that satisfies the identities~\eqref{id: M4 basis} but violates all of the identities
\begin{equation}
x^3 \id x, \quad xy \id yx. \label{id: M4v3 not ji}
\end{equation}
Then $\lpbr S\rpbr$ is a sub{\pvar} of $\lpbr \mathscr{K}\rpbr$ that is not {\ji}\up.
\end{corollary}

\begin{proof}
By Proposition~\ref{P: M4}, the inclusion
\[
\lpbr S\rpbr \subseteq \lpbr\mathscr{K}\rpbr = \lpbr\ZB, \ZBop\rpbr \vee \lpbr\Z_3, \Z_4, \N_3^I\rpbr
\]
holds.
But the two identities in~\eqref{id: M4v3 not ji} are satisfied by $\ZB \times \ZBop$ and $\Z_3 \times \Z_4 \times \N_3^I$, respectively.
Therefore, $\lpbr S\rpbr \nsubseteq \lpbr\ZB, \ZBop\rpbr$ and $\lpbr S\rpbr \nsubseteq \lpbr\Z_3, \Z_4, \N_3^I\rpbr$.
\end{proof}

\begin{corollary} \label{C: M4v2}
Suppose that~$S$ is any finite semigroup that satisfies the identities~\eqref{id: M4 basis} but violates all of the identities
\begin{equation}
xy \id yx, \quad xyx^2 \id xy, \quad x^2yx \id yx. \label{id: M4v2 not ji}
\end{equation}
Then $\lpbr S\rpbr$ is a sub{\pvar} of $\lpbr\mathscr{K}\rpbr$ that is not {\ji}\up.
\end{corollary}

\begin{proof}
By Proposition~\ref{P: M4}, the inclusion
\[
\lpbr S\rpbr \subseteq \lpbr\mathscr{K}\rpbr = \lpbr\Z_3, \Z_4, \N_3^I\rpbr \vee \lpbr\ZBop\rpbr \vee \lpbr\ZB\rpbr
\]
holds.
Since the three identities in~\eqref{id: M4v2 not ji} are satisfied by $\Z_3 \times \Z_4 \times \N_3^I$, $\ZBop$, and $\ZB$, respectively, the exclusions $\lpbr S\rpbr \nsubseteq \lpbr \Z_3, \Z_4, \N_3^I \rpbr$, $\lpbr S\rpbr \nsubseteq \lpbr\ZBop\rpbr$, and $\lpbr S\rpbr \nsubseteq \lpbr\ZB\rpbr$ follow.
\end{proof}

\begin{corollary} \label{C: M4v1}
Suppose that~$S$ is any finite semigroup that satisfies the identities~\eqref{id: M4 basis} but violates all of the identities
\begin{equation}
x^4 \id x^3, \quad x^4y \id y, \quad x^3y \id y, \quad xyx^2 \id xy, \quad x^2yx \id yx. \label{id: M4v1 not ji}
\end{equation}
Then $\lpbr S\rpbr$ is a sub{\pvar} of $\lpbr\mathscr{K}\rpbr$ that is not {\ji}\up.
\end{corollary}

\begin{proof}
The inclusion $\lpbr S\rpbr \subseteq \bigvee \{ \lpbr T \rpbr \mid T \in \mathscr{K}\}$ holds by Proposition~\ref{P: M4}.
But the five identities in~\eqref{id: M4v1 not ji} are satisfied by $\N_3^I$, $\Z_4$, $\Z_3$, $\ZBop$, and $\ZB$, respectively.
Therefore, $\lpbr S\rpbr \nsubseteq \lpbr T\rpbr$ for all $T \in \mathscr{K}$.
\end{proof}

\subsection{The {\pvar} $\lpbr\Z_m, \N_n^I, \LZ^I,\RZ^I,\Az^I\rpbr$}

In this subsection, it is convenient to write
\[
\mathscr{T}_{m,n} = \{\Z_m, \N_n^I, \LZ^I, \RZ^I, \Az^I\}
\]
and $T_{m,n} = \Z_m \times \N_n^I \times \LZ^I \times \RZ^I \times \Az^I$.

\begin{proposition} \label{P: ZmjNnIjLZIjRZIjA0I}
Let $m \geq 1$ and $n \geq 2$\up.
Then the identities satisfied by the semi\-group $T_{m,n}$ are axiomatized by
\begin{equation} \label{id: ZmjNnIjLZIjRZIjA0I basis}
x^{m+n} \id x^n, \quad x^{m+n-1}yx \id x^{n-1}yx, \quad x^2yx \id xyx^2, \quad xyxzx \id x^2yzx.
\end{equation}
\end{proposition}

\begin{remark}
\begin{enumerate}[(i)]
\item Since $\N_2$ is isomorphic to the subsemigroup $\{ \ez,\ef\ee\}$ of $\Az$, the monoid $\N_2^I$ belongs to $\lpbr \Az^I \rpbr$.
Therefore, $\lpbr T_{m,1} \rpbr = \lpbr T_{m,2} \rpbr$.
This is the reason for the assumption $n \geq 2$ in Proposition~\ref{P: ZmjNnIjLZIjRZIjA0I}.

\item The basic case $(m,n)=(1,2)$ for Proposition~\ref{P: ZmjNnIjLZIjRZIjA0I} was first established in Lee~\cite[Proposition~2.3(i)]{Lee11}.
\end{enumerate}
\end{remark}

Suppose that a word~$\bw$ can be written in the form
\[
\bw = \bw_0\prod_{i=1}^r (x^{e_i} \bw_i) = \bw_0 x^{e_1} \bw_1 x^{e_2} \bw_2 \cdots x^{e_r} \bw_r,
\]
where $x \in \A$, $\bw_0, \bw_r \in \A^*$, and $\bw_1, \bw_2,\ldots, \bw_{r-1} \in \A^+$ are such that $x \notin \cont(\bw_i)$ for all~$i$, and $e_1,e_2,\ldots,e_r \in \{1,2,3,\ldots\}$.
Then the factors $x^{e_1}, x^{e_2},\ldots,x^{e_r}$ are call \textit{$x$-stacks}, or simply \textit{stacks}, of~$\bw$.
The \textit{weight} of the $x$-stack~$x^{e_i}$ is~$e_i$.

It is easily shown that the identities~\eqref{id: ZmjNnIjLZIjRZIjA0I basis} can be used to convert any word into a word~$\bw$ such that for each $x \in \A$,
\begin{enumerate}[(I)]
\item the number of $x$-stacks in~$\bw$ is at most two;
\item if~$\bw$ has one $x$-stack, then its weight is at most $m+n-1$;
\item if~$\bw$ has two $x$-stacks, then the weight of the first $x$-stack is at most $m+n-2$ while the weight of the second $x$-stack is one.
\end{enumerate}
In this subsection, a word~$\bw$ that satisfies (I)--(III) is said to be in \textit{canonical form}.
Note that if~$\bw$ is a word in canonical form, then $\occ(x,\bw) \leq m+n-1$ for any $x \in \A$.

\begin{lemma} \label{L: word ZmjNnIjLZIjRZIjA0I}
Let~$\bu$ and~$\bv$ be any words in canonical form such that the identity $\bu \id \bv$ is satisfied by the semigroup $T_{m,n}$\up.
Then for any $x\in \A$\up,
\begin{enumerate}[\rm(i)]
\item $\occ(x,\bu) \equiv \occ(x,\bv) \pmod m$\up;
\item either $\occ(x,\bu) = \occ(x,\bv) \leq n$ or $n < \occ(x,\bu), \occ(x,\bv) \leq m+n-1$.
\end{enumerate}
\end{lemma}

\begin{proof}
This follows from Lemma~\ref{L: word Zn NnI LZI RZI} parts~(i) and~(ii).
\end{proof}

For any word~$\bw$ and distinct variables $x_1,x_2,\ldots,x_r$, let $\bw_{\{x_1,x_2,\ldots,x_r\}}$ denote the word obtained from~$\bw$ by retaining only the variables $x_1,x_2,\ldots,x_r$.
Any monoid that satisfies an identity $\bu \id \bv$ also satisfies $\bu_{\{x_1,x_2,\ldots,x_r\}} \id \bv_{\{x_1,x_2,\ldots,x_r\}}$ for any distinct variables $x_1,x_2,\ldots,x_r$.

\begin{lemma} \label{L: stack ZmjNnIjLZIjRZIjA0I}
Let~$\bu$ and~$\bv$ be any words in canonical form such that the identity $\bu \id \bv$ is satisfied by the semigroup $T_{m,n}$\up.
Then
\begin{enumerate}[\rm(i)]
\item for any distinct $x,y \in \A$\up, the identity $\bu_{\{x,y\}} \id \bv_{\{x,y\}}$ cannot be any of
\begin{equation}
x^{e_1}y^{f_1} \id x^{e_2} y^{f_2} x^{e_3}, \quad x^{e_1}y^{f_1} \id y^{f_2} x^{e_2}y^{f_3}, \quad x^{e_1}y^{f_1} \id x^{e_2} y^{f_2} x^{e_3}y^{f_3}, \label{id: nonA0I}
\end{equation}
where $e_1,f_1,e_2,f_2,e_3,f_3 \geq 1$\up;
\item $\bu$ has two $x$-stacks if and only if~$\bv$ has two $x$-stacks\up;
\item $x^e$ is the first $x$-stack of~$\bu$ if and only if~$x^e$ is the first $x$-stack of~$\bv$\up.
\end{enumerate}
\end{lemma}

\begin{proof}
(i) The three identities in~\eqref{id: nonA0I} are violated by the semigroups~$\RZ^I$, $\LZ^I$, and $\Az^I$, respectively.

(ii) Suppose that~$\bu$ has two $x$-stacks.
Then by~(III),
\[
\bu = \bu_1 x^{e-1} \bu_2 x \bu_3
\]
for some $\bu_1, \bu_3 \in \A^*$ and $\bu_2 \in \A^+$ with $x \notin \cont(\bu_1\bu_2\bu_3)$ and $2 \leq e \leq m+n-1$.
Seeking a contradiction, suppose that~$\bv$ has only one $x$-stack.
Then by~(II) and part~(i),
\[
\bv = \bv_1 x^f \bv_2
\]
for some $\bv_1,\bv_2 \in \A^*$ with $x \notin \cont(\bv_1 \bv_2)$ and $1 \leq f \leq m+n-1$.
Since the word~$\bu_2$ is nonempty, it contains some $y$-stack.
Since $\cont(\bu)=\cont(\bv)$ by Lemma~\ref{L: word Zn NnI LZI RZI}(ii), it follows that $y \in \cont(\bv_1 \bv_2)$.
By symmetry, it suffices to assume that $y \in \cont(\bv_1)$, so that $\ini(\bv) = \cdots y \cdots x \cdots$.
Since $\ini(\bu)=\ini(\bv)$ by Lemma~\ref{L: word Zn NnI LZI RZI}(iii), it follows that $y \in \cont(\bu_1)$.
Hence the word~$\bu$ contains two $y$-stacks, the first of which occurs in~$\bu_1$ while the second occurs in~$\bu_2$.
Thus $\fin(\bu) = \cdots y \cdots x \cdots$.
Since $\fin(\bu)=\fin(\bv)$ by Lemma~\ref{L: word Zn NnI LZI RZI}(iv), it follows that $y \notin \cont(\bv_2)$.
Therefore, the identity $\bu_{\{x,y\}} \id \bv_{\{x,y\}}$ is $y^rx^{e-1}yx \id y^sx^f$ for some $r,s \geq 1$, but this contradicts part~(i).

(iii) Let~$x^e$ be a first $x$-stack of~$\bu$.
By part~(ii), there are two cases.

\paragraph{\sc Case~1}

$\bu$ and~$\bv$ each has only one $x$-stack.
Then by~(II),
\[
\bu = \bu_1 x^e \bu_2 \quad \text{and} \quad \bv = \bv_1 x^f \bv_2
\]
for some $\bu_1, \bu_2, \bv_1, \bv_2 \in \A^*$ with $x \notin \cont(\bu_1 \bu_2 \bv_1 \bv_2)$ and $1 \leq e,f \leq m+n-1$.
Since $e = \occ(x,\bu)$ and $f = \occ(x,\bv)$, it follows from part~(ii) that either $e=f \leq n$ or $n < e,f \leq m+n-1$.
If $n < e,f \leq m+n-1$, then $e=f$ by part~(i).

\paragraph{\sc Case~2}

$\bu$ and~$\bv$ each has two $x$-stacks.
Then by~(III),
\[
\bu = \bu_1 x^{e-1} \bu_2 x \bu_3 \quad \text{and} \quad \bv = \bv_1 x^{f-1} \bv_2 x \bv_3
\]
for some $\bu_1, \bu_3, \bv_1, \bv_3 \in \A^*$ and $\bu_2, \bv_2 \in \A^+$ with $x \notin \cont(\bu_1 \bu_2 \bu_3 \bv_1 \bv_2 \bv_3)$ and $2 \leq e,f \leq m+n-1$.
Since $e = \occ(x,\bu)$ and $f = \occ(x,\bv)$, it follows from the same argument in Case~1 that $e=f$.
\end{proof}

\begin{lemma} \label{L: adjacent ZmjNnIjLZIjRZIjA0I}
Let~$\bu$ and~$\bv$ be any words in canonical form such that the identity $\bu \id \bv$ is satisfied by the semigroup $T_{m,n}$\up.
Then the following are equivalent\up:
\begin{enumerate}[\rm(a)]
\item $\bu \in \A^* x^ey^f \A^*$ where~$x^e$ and~$y^f$ are stacks of~$\bu$\up;
\item $\bv \in \A^* x^ey^f \A^*$ where~$x^e$ and~$y^f$ are stacks of~$\bv$\up.
\end{enumerate}
Further\up, $x^e$ is the first $x$-stack of~$\bu$ if and only if~$x^e$ is the first $x$-stack of~$\bv$\up, and $y^f$ is the first $y$-stack of~$\bu$ if and only if~$y^f$ is the first $y$-stack of~$\bv$\up.
\end{lemma}

\begin{proof}
First, note that $\ini(\bu) = \ini(\bv)$ and $\fin(\bu)=\fin(\bv)$ by Lemma~\ref{L: word Zn NnI LZI RZI}.
Suppose that~(a) holds. Then
\[
\bu = \bu_1 x^e y^f \bu_2
\]
for some $\bu_1, \bu_2 \in \A^*$ such that~$\bu_1$ does not end with~$x$ while~$\bu_2$ does not begin with~$y$.
There are four cases depending on which of~$x^e$ and~$y^f$ are first stacks in~$\bu$.

\paragraph{\sc Case~1}

$x^e$ is the first $x$-stack in~$\bu$ and~$y^f$ is the first $y$-stack in~$\bu$.
Then clearly $x,y \notin \cont(\bu_1)$, so that $\ini(\bu) = \cdots xy \cdots$.
By Lemma~\ref{L: stack ZmjNnIjLZIjRZIjA0I}(iii), $x^e$ is the first $x$-stack of~$\bv$ and~$y^f$ is the first $y$-stack of~$\bv$.
Since $\ini(\bv) = \ini(\bu) = \cdots xy \cdots$,
\[
\bv = \bv_1 x^e \bv_2 y^f \bv_3
\]
for some $\bv_1, \bv_2, \bv_3 \in \A^*$ such that $x \notin \cont(\bv_1)$ and $y \notin \cont(\bv_1 \bv_2)$, and that any stack of~$\bv$ that occurs in~$\bv_2$ cannot be a first stack.
Suppose that $\bv_2 \neq \emptyset$.
Then the first variable~$z$ of~$\bv_2$ constitutes the second $z$-stack of~$\bv$.
Hence
\[
\bv = \underbrace{\,\cdots z^r \cdots\,}_{\bv_1} \, x^e \underbrace{\,z\cdots\,}_{\bv_2} \, y^f \bv_3,
\]
where $z^r$ is the first $z$-stack of~$\bv$, and $\ini(\bv) = \cdots z \cdots xy \cdots$.
By Lemma \ref{L: stack ZmjNnIjLZIjRZIjA0I}(ii), the word~$\bu$ contains two $z$-stacks; by part~(iii) of the same lemma, the first $z$-stack of~$\bu$ is~$z^r$.
Since $\ini(\bu) = \ini(\bv) = \cdots z \cdots xy \cdots$, the $z$-stack~$z^r$ of~$\bu$ occurs in~$\bu_1$:
\[
\bu = \underbrace{\,\cdots z^r \cdots\,}_{\bu_1} \, x^e y^f \bu_2.
\]
The second $z$-stack of~$\bu$ occurs in either~$\bu_1$ or~$\bu_2$.
There are two subcases.
\begin{enumerate}[{1.}1.]
\item The second $z$-stack of~$\bu$ occurs in~$\bu_1$.
Then $\fin(\bv) = \fin(\bu) = \cdots z \cdots x \cdots$, so that~$\bv$ must contain a second $x$-stack occurring in either~$\bv_2$ or~$\bv_3$.
The identity $\bu_{\{x,z\}} \id \bv_{\{x,z\}}$ is thus $z^{r+1} x^{e+1} \id z^rx^ezx$, which is impossible by Lemma~\ref{L: stack ZmjNnIjLZIjRZIjA0I}(i).
\item The second $z$-stack of~$\bu$ occurs in~$\bu_2$.
Then $\fin(\bu) = \fin(\bv) = \cdots z \cdots y \cdots$, so that~$\bu$ must contain a second $y$-stack occurring after the second $z$-stack:
\[
\bu = \underbrace{\,\cdots z^r \cdots\,}_{\bu_1} \, x^e y^f \underbrace{\,\cdots z \cdots y \cdots\,}_{\bu_2}.
\]
The identity $\bu_{\{y,z\}} \id \bv_{\{y,z\}}$ is thus $z^ry^fzy \id z^{r+1} y^{f+1}$, which is impossible by Lemma~\ref{L: stack
ZmjNnIjLZIjRZIjA0I}(i).
\end{enumerate}
Since both subcases are impossible, $\bv_2 = \emptyset$.
Hence~(b) holds.

\paragraph{\sc Case~2}

$x^e$ is the first $x$-stack in~$\bu$ and~$y^f$ is the second $y$-stack in~$\bu$.
Then $f=1$ by~(III) and
\[
\bu = \underbrace{\,\cdots y^r \cdots\,}_{\bu_1} \, x^e y \bu_2,
\]
where~$y^r$ is the first $y$-stack of~$\bu$.
Since $\ini(\bv) = \ini(\bu) = \cdots y \cdots x \cdots$, it follows from Lemma~\ref{L: stack ZmjNnIjLZIjRZIjA0I} parts~(i) and~(iii) that
\[
\bv = \bv_1 y^r \bv_2 x^e \bv_3 y \bv_4
\]
for some $\bv_1,\bv_2,\bv_3,\bv_4 \in \A^*$, where~$y^r$ is the first $y$-stack of~$\bv$ and~$x^e$ is the first $x$-stack of~$\bv$.
Suppose that $\bv_3 \neq \emptyset$. Then $\bv_3$ contains some $z$-stack~$z^s$:
\[
\bv = \bv_1 y^r \bv_2 x^e \underbrace{\,\cdots z^s \cdots\,}_{\bv_3} \, y \bv_4.
\]
There are two subcases depending on whether $z^s$ is the first or second $z$-stack in~$\bv$.
\begin{enumerate}[{2.}1.]
\item $z^s$ is the first $z$-stack in~$\bv$.
Then $\ini(\bu) = \ini(\bv) = \cdots y \cdots x \cdots z \cdots$, so that every~$z$ of~$\bu$ occurs in~$\bu_2$.
Hence $\bu_{\{y,z\}} \in y^{r+1} \{z\}^+$ and
\[
\bv_{\{y,z\}} =
\begin{cases}
y^r z^syz & \text{if~$\bv$ has a second $z$-stack occurring in $\bv_4$}, \\
y^r z^{s+1}y & \text{if~$\bv$ has a second $z$-stack occurring in $\bv_3$}, \\
y^r z^sy & \text{if~$\bv$ has no second $z$-stack}.
\end{cases}
\]
But this is impossible by Lemma~\ref{L: stack ZmjNnIjLZIjRZIjA0I}(i).
\item $z^s$ is the second $z$-stack in~$\bv$.
Then $\fin(\bu) = \fin(\bv) = \cdots z \cdots y \cdots$, so that every~$z$ of~$\bu$ occurs in~$\bu_1$.
Hence $\bu_{\{x,z\}} \in \{z\}^+\{x\}^+$ and
\[
\bv_{\{x,z\}} \in
\begin{cases}
\{z\}^+\{x\}^+z^s\{x\}^* & \text{if the first $z$-stack of~$\bv$ occurs in~$\bv_1$ or~$\bv_2$}, \\
\{x\}^+\{z\}^+\{x\}^* & \text{if the first $z$-stack of~$\bv$ occurs in~$\bv_3$}.
\end{cases}
\]
But this is impossible by Lemma~\ref{L: stack ZmjNnIjLZIjRZIjA0I}(i).
\end{enumerate}
Since both subcases are impossible, $\bv_3 = \emptyset$.
Hence~(b) holds.

\paragraph{\sc Case~3}

$x^e$ is the second $x$-stack in~$\bu$ and~$y^f$ is the first $y$-stack in~$\bu$.
Then $e=1$ by~(III) and
\[
\bu = \underbrace{\,\cdots x^r \cdots\,}_{\bu_1} \, x y^f \bu_2,
\]
where~$x^r$ is the first $x$-stack of~$\bu$ with and $y \notin \cont(\bu_1)$ and $x \notin \cont(\bu_2)$.
Since $\ini(\bv) = \ini(\bu) = \cdots x \cdots y \cdots$, it follows from Lemma~\ref{L: stack ZmjNnIjLZIjRZIjA0I} that
\[
\bv = \bv_1 x^r \bv_2 x \bv_3 y^f \bv_4
\]
for some $\bv_1, \bv_2, \bv_3, \bv_4 \in \A^*$ with $x \notin \cont(\bv_1 \bv_2 \bv_3 \bv_4)$ and $y\notin \cont(\bv_1 \bv_2 \bv_3)$.
Suppose that $\bv_3 \neq \emptyset$. Then $\bv_3$ contains some $z$-stack~$z^s$:
\[
\bv = \bv_1 x^r \bv_2 x \underbrace{\,\cdots z^s \cdots\,}_{\bv_3} \, y^f \bv_4.
\]
There are two subcases depending on whether or not~$z^s$ is the first $z$-stack of~$\bv$.
\begin{enumerate}[{3.}1.]
\item $z^s$ is the first $z$-stack of~$\bv$.
Since $\ini(\bu) = \ini(\bv) = \cdots z \cdots y \cdots$, the first $z$-stack~$z^s$ of~$\bu$ occurs in~$\bu_1$.
But since $\fin(\bu) = \fin(\bv) = \cdots x \cdots z \cdots$, the word~$\bu$ must contain a second $z$-stack in~$\bu_2$, whence~$\bv$ must also contain a second $z$-stack by Lemma~\ref{L: stack ZmjNnIjLZIjRZIjA0I}(ii).
Hence $\bv_{\{x,z\}} = x^{r+1}z^{s+1}$ and
\[
\bu_{\{x,z\}} =
\begin{cases}
z^sx^{r+1}z & \text{if~$z^s$ in~$\bu$ occurs before the first $x$-stack}, \\
x^rz^sxz & \text{if~$z^s$ in~$\bu$ occurs between the two $x$-stacks}.
\end{cases}
\]
But this is impossible by Lemma~\ref{L: stack ZmjNnIjLZIjRZIjA0I}(i).
\item $z^s$ is the second $z$-stack of~$\bv$. Then the identity $\bu_{\{y,z\}} \id \bv_{\{y,z\}}$ obtained by an argument symmetrical to the one in Subcase~3.1 produces a similar contradiction.
\end{enumerate}
Since both subcases are impossible, $\bv_3 = \emptyset$. Hence~(b) holds.

\paragraph{\sc Case~4}

$x^e$ is the second $x$-stack in~$\bu$ and~$y^f$ is the second $y$-stack in~$\bu$.
Then~(b) holds by an argument symmetrical to Case~1.

Therefore, (b) holds in all four cases. By symmetry, (b) implies~(a).
\end{proof}

\begin{proof}[Proof of Proposition~\ref{P: ZmjNnIjLZIjRZIjA0I}]
It is routinely verified that the semigroup $T_{m,n}$ satisfies the identities~\eqref{id: ZmjNnIjLZIjRZIjA0I basis}.
Conversely, suppose that $\bu \id \bv$ is any identity satisfied by the semigroup $T_{m,n}$.
As observed earlier, the identities~\eqref{id: ZmjNnIjLZIjRZIjA0I basis} can be used to convert~$\bu$ and~$\bv$ into words~$\bu'$ and~$\bv'$ in canonical form.
By Lemma~\ref{L: word ZmjNnIjLZIjRZIjA0I} parts~(ii) and~(iii), the words~$\bu'$ and~$\bv'$ share the same set of stacks.
By Lemma~\ref{L: adjacent ZmjNnIjLZIjRZIjA0I}, two stacks are adjacent in~$\bu'$ if and only if they are adjacent in~$\bv'$.
Therefore, $\bu'=\bv'$.
Since
\[
\bu \stackrel{\eqref{id: ZmjNnIjLZIjRZIjA0I basis}}{\id} \bu'=\bv' \stackrel{\eqref{id: ZmjNnIjLZIjRZIjA0I basis}}{\id} \bv,
\]
the identity $\bu \id \bv$ is deducible from~\eqref{id: ZmjNnIjLZIjRZIjA0I basis}.
\end{proof}

\begin{corollary} \label{C: Z6jN5IjLZIjRZIjA0Iv1}
Suppose that~$S$ is any finite semigroup that satisfies the identities
\begin{equation}
x^{11} \id x^5, \quad x^{10}yx \id x^4yx, \quad x^2yx \id xyx^2, \quad xyxzx \id x^2yzx \label{id: Z6jN5IjLZIjRZIjA0I basis}
\end{equation}
but violates all of the identities
\begin{equation}
x^2 \id x, \quad xyxy \id yxyx. \label{id: Z6jN5IjLZIjRZIjA0Iv1 not ji}
\end{equation}
Then $\lpbr S\rpbr$ is a sub{\pvar} of $\lpbr\mathscr{T}_{6,5}\rpbr$ that is not {\ji}.
\end{corollary}

\begin{proof}
By Proposition~\ref{P: ZmjNnIjLZIjRZIjA0I} with $m=6$ and $n=5$, the identities satisfied by~$T_{6,5}$ are axiomatized by~\eqref{id: Z6jN5IjLZIjRZIjA0I basis}.
Hence the inclusion
\[
\lpbr S\rpbr \subseteq \lpbr\mathscr{T}_{6,5}\rpbr = \lpbr\LZ^I, \RZ^I\rpbr \vee \lpbr\Z_6, \N_5^I, \Az^I\rpbr
\]
holds.
But the two identities in~\eqref{id: Z6jN5IjLZIjRZIjA0Iv1 not ji} are satisfied by $\LZ^I \times \RZ^I$ and $\Z_6 \times \N_5^I \times \Az^I$, respectively.
Therefore, $\lpbr S\rpbr \nsubseteq \lpbr\LZ^I, \RZ^I\rpbr$ and $\lpbr S\rpbr \nsubseteq \lpbr\Z_6, \N_5^I, \Az^I\rpbr$.
\end{proof}

\begin{corollary} \label{C: Z6jN5IjLZIjRZIjA0Iv2}
Suppose that~$S$ is any finite semigroup that satisfies the identities~\eqref{id: Z6jN5IjLZIjRZIjA0I basis} but violates all of the identities
\begin{equation}
x^6 \id x^5, \quad x^6y \id y. \label{id: Z6jN5IjLZIjRZIjA0Iv2 not ji}
\end{equation}
Then $\lpbr S\rpbr$ is a sub{\pvar} of $\lpbr\mathscr{T}_{6,5}\rpbr$ that is not {\ji}.
\end{corollary}

\begin{proof}
The argument in the proof of Corollary~\ref{C: Z6jN5IjLZIjRZIjA0Iv1} implies the inclusion
\[
\lpbr S\rpbr \subseteq \lpbr\mathscr{T}_{6,5}\rpbr = \lpbr\N_5^I, \LZ^I, \RZ^I, \Az^I\rpbr \vee \lpbr\Z_6\rpbr.
\]
The two identities in~\eqref{id: Z6jN5IjLZIjRZIjA0Iv2 not ji} are satisfied by $\N_5^I \times \LZ^I \times \RZ^I \times \Az^I$ and $\Z_6$, respectively.
Therefore, $\lpbr S\rpbr \nsubseteq \lpbr\N_5^I, \LZ^I, \RZ^I, \Az^I\rpbr$ and $\lpbr S\rpbr \nsubseteq \lpbr\Z_6\rpbr$.
\end{proof}

\subsection{Pseudovarieties of noncommutative nilpotent semigroups}

\begin{proposition} \label{P: noncomm nil}
Any {\ji} {\pvar} of nilpotent semigroups is commutative.
\end{proposition}

\begin{proof}
Let~$\bV$ be any {\ji} {\pvar} of nilpotent semigroups.
Then the inclusion $\bV \subseteq \bCom \vee \bG$ holds~\cite[Figure~9.1]{Alm94}.
Since~$\bV$ is {\ji} and $\bV \nsubseteq \bG$, the inclusion $\bV \subseteq \bCom$ follows.
\end{proof}

\begin{corollary} \label{C: noncomm nil}
Suppose that~$S$ is any finite semigroup that satisfies the identity
\[
x^6 \id y_1y_2y_3y_4y_5y_6
\]
but violates the identity
\[
xy \id yx.
\]
Then $\lpbr S\rpbr$ is a {\pvar} of nilpotent semigroups that is not {\ji}.
\end{corollary}

\begin{proof}
By assumption, the semigroup~$S$ is nilpotent and noncommutative.
The result then holds by Proposition~\ref{P: noncomm nil}.
\end{proof}

\subsection{The {\pvar} $\lpbr\N_{n+r},\N_n^I\rpbr$}

\begin{proposition} \label{P: Nn+rjNnI}
Let $n,r \geq 1$.
Then the identities satisfied by the semigroup $\N_{n+r} \times \N_n^I$ are axiomatized by
\begin{subequations} \label{id: Nn+rjNnI basis}
\begin{align}
xy & \id yx, \label{id: Nn+rjNnI basis xy} \\
x^{n+1} y_1 y_2 \cdots y_r & \id x^n y_1 y_2\cdots x_r, \label{id: Nn+rjNnI basis xn+1y1y2...yr} \\
x_1^{e_1} x_2^{e_2} \cdots x_m^{e_m} & \id x_1^{f_1} x_2^{f_2} \cdots x_m^{f_m} \label{id: Nn+rjNnI basis list}
\end{align}
\end{subequations}
for all $m \geq 1$ and $e_1,e_2,\ldots,e_m,f_1,f_2,\ldots,f_m \geq 1$ such that
\begin{enumerate}[\ \rm(a)]
\item $e = f < n+r$\up, where $e = \sum_{i=1}^m e_i$ and $f = \sum_{i=1}^m f_i$;
\item for each $k \in \{ 1,2, \ldots, m\}$\up, either
\begin{enumerate}[1.]
\item[$\bullet$] $e_k = f_k$ or
\item[$\bullet$] $e_k,f_k \geq n$ and $e+e_k, f+f_k \geq n+r$\up.
\end{enumerate}
\end{enumerate}
\end{proposition}

\begin{proof}
It is straightforwardly verified that the semigroup $\N_{n+r} \times \N_n^I$ satisfies the identities~\eqref{id: Nn+rjNnI basis}.
Conversely, let $\bu \id \bv$ be any identity satisfied by the semigroup $\N_{n+r} \times \N_n^I$. In view of Lemma~\ref{L: word Zn NnI LZI RZI}(ii), the identity~\eqref{id: Nn+rjNnI basis xy} can be used to convert~$\bu$ and~$\bv$ into
\[
\bu' = x_1^{e_1} x_2^{e_2} \cdots x_m^{e_m} \quad \text{and} \quad \bv' = x_1^{f_1} x_2^{f_2} \cdots x_m^{f_m},
\]
respectively, where $e_i = \occ(x_i,\bu)$ and $f_i = \occ(x_i,\bv)$ are such that either $e_i = f_i$ or $e_i,f_i \geq n$.
Let $e = \sum_{i=1}^m e_i$ and $f = \sum_{i=1}^m f_i$.
Generality is not lost by assuming that $e \leq f$.
There are four cases to consider.

\paragraph{\sc Case~1}

$n+r \leq e \leq f$.
Choose any $i \in \{1,2,\ldots,m\}$.
Suppose that $e_i \neq f_i$.
Then as observed earlier, $e_i,f_i \geq n$.
Hence
\begin{align*}
\bu' & \stackrel{\eqref{id: Nn+rjNnI basis xy}}{\id} x_i^n \, \overbrace{x_1^{e_1} x_2^{e_2} \cdots x_{i-1}^{e_{i-1}} x_i^{e_i-n} x_{i+1}^{e_{i+1}} \cdots x_m^{e_m}}^\text{$e-n \geq r$ variables} \\
     & \stackrel{\eqref{id: Nn+rjNnI basis xn+1y1y2...yr}}{\id} x_i^{n+f_i} x_1^{e_1} x_2^{e_2} \cdots x_{i-1}^{e_{i-1}} x_i^{e_i-n} x_{i+1}^{e_{i+1}} \cdots x_m^{e_m} \\
     & \stackrel{\eqref{id: Nn+rjNnI basis xy}}{\id} x_1^{e_1} x_2^{e_2} \cdots x_{i-1}^{e_{i-1}} x_i^{e_i+f_i} x_{i+1}^{e_{i+1}} \cdots x_m^{e_m}.
\end{align*}
Similarly, $\bv' \stackrel{\eqref{id: Nn+rjNnI basis}}{\id} x_1^{f_1} x_2^{f_2} \cdots x_{i-1}^{f_{i-1}} x_i^{f_i+e_i} x_{i+1}^{f_{i+1}} \cdots x_m^{f_m}$.
Therefore, the identities~\eqref{id: Nn+rjNnI basis} can be used to convert~$\bu'$ into~$\bv'$.
It follows that $\bu \id \bv$ is deducible from~\eqref{id: Nn+rjNnI basis}.

\paragraph{\sc Case~2}

$e < n+r \leq f$.
Let $\varphi : \A \rightarrow \N_{n+r}$ be the substitution that maps all variables to~$\ea$.
Then $\bu' \varphi = \ea^e \neq 0$ and $\bv' \varphi = \ea^f = 0$ imply the contradiction $\bu' \varphi \neq \bv' \varphi$.
The present case is thus impossible.

\paragraph{\sc Case~3}

$e < f < n+r$.
Then the contradiction $\bu' \varphi = \ea^e \neq \ea^f = \bv' \varphi$ is obtained.
Hence the present case is impossible.

\paragraph{\sc Case~4}

$e = f < n+r$.
Suppose that $e_k \neq f_k$ for some~$k$ so that $e+e_k \neq f+f_k$.
Then as observed earlier, $e_k,f_k \geq n$.
Let $\psi : \A \rightarrow \N_{n+r}$ be the substitution that maps~$x_k$ to $\ea^2$ and all other variables to~$\ea$.
Then $\bu' \psi = \bv' \psi$ in $\N_n$, where
\[
\bu' \psi = \bigg(\prod_{i=1}^{k-1} \ea^{e_i} \bigg) (\ea^2)^{e_k} \bigg(\prod_{i=k+1}^m \ea^{e_i} \bigg) = \ea^{e+e_k}
\]
and $\bv' \psi = \ea^{f+f_k}$ similarly.
Thus $\ea^{e+e_k} = \ea^{f+f_k}$.
But $e+e_k \neq f+f_k$ implies that $e+e_k, f+f_k \geq n$.
Hence the identity $\bu' \id \bv'$ also satisfies~(ii) and is deducible from~\eqref{id: Nn+rjNnI basis list}.
The identity $\bu \id \bv$ is thus deducible from~\eqref{id: Nn+rjNnI basis}.
\end{proof}

\begin{corollary} \label{C: N4jN2I}
Suppose that~$S$ is any finite semigroup that satisfies the identities
\begin{equation}
xy \id yx, \quad x^3y_1y_2 \id x^2y_1y_2 \label{id: N4jN2I basis}
\end{equation}
but violates all of the identities
\begin{equation}
x^3 \id x^2, \quad x^2y \id xy^2. \label{id: N4jN2I not ji}
\end{equation}
Then $\lpbr S\rpbr$ is a sub{\pvar} of $\lpbr\N_4,\N_2^I\rpbr$ that is not {\ji}.
\end{corollary}

\begin{proof}
By Proposition~\ref{P: Nn+rjNnI} with $n=r=2$, the identities satisfied by $\N_4 \times \N_2^I$ are axiomatized by~\eqref{id: N4jN2I basis}.
Hence the inclusion
\[
\lpbr S\rpbr \subseteq \lpbr\N_4,\N_2^I\rpbr = \lpbr\N_4\rpbr \vee \lpbr\N_2^I\rpbr
\]
holds.
But the two identities in~\eqref{id: N4jN2I not ji} are satisfied by $\N_2^I$ and $\N_4$, respectively.
Therefore, $\lpbr S\rpbr \nsubseteq \lpbr\N_4\rpbr$ and $\lpbr S\rpbr \nsubseteq \lpbr\N_2^I\rpbr$.
\end{proof}

\begin{corollary} \label{C: N5jN1I}
Suppose that~$S$ is any finite semigroup that satisfies the identities
\begin{equation}
xy \id yx, \quad x^2yz \id xy^2z, \quad x^2y_1y_2y_3y_4 \id xy_1y_2y_3y_4 \label{id: N5jN1I basis}
\end{equation}
but violates all of the identities
\begin{equation}
x^2 \id x, \quad x^5 \id y^5. \label{id: N5jN1I not ji}
\end{equation}
Then $\lpbr S\rpbr$ is a sub{\pvar} of $\lpbr\N_5,\N_1^I\rpbr$ that is not {\ji}.
\end{corollary}

\begin{proof}
By Proposition~\ref{P: Nn+rjNnI} with $n=1$ and $r=4$, the identities satisfied by $\N_5 \times \N_1^I$ are axiomatized by~\eqref{id: N5jN1I basis}.
Hence the inclusion
\[
\lpbr S\rpbr \subseteq \lpbr\N_5,\N_1^I\rpbr = \lpbr\N_5\rpbr \vee \lpbr\N_1^I\rpbr
\]
holds.
But the two identities in~\eqref{id: N5jN1I not ji} are satisfied by $\N_1^I$ and $\N_5$, respectively.
Therefore, $\lpbr S\rpbr \nsubseteq \lpbr\N_5\rpbr$ and $\lpbr S\rpbr \nsubseteq \lpbr\N_1^I\rpbr$.
\end{proof}

\subsection{The {\pvar} $\lpbr\N_n^I,\NB\rpbr$}

It turns out to be notationally simpler to find a basis of identities for $\N_n^I \times \NBop$ instead of $\N_n^I \times \NB$.

\begin{proposition} \label{P: NnIjN2barop}
Let $n \geq 2$.
Then the identities satisfied by the semigroup $\N_n^I \times \NBop$ are axiomatized by
\begin{subequations} \label{id: NnIjN2barop basis}
\begin{gather}
x^{n+1} \id x^n, \quad xyx^n \id xyx^{n-1}, \label{id: NnIjN2barop basis xn+1} \\
xyzt \id xytz. \label{id: NnIjN2barop basis xyzt}
\end{gather}
\end{subequations}
\end{proposition}

In this subsection, a word of length at least two is said to be in \textit{canonical form} if it is either
\begin{enumerate}[({CF}1)]
\item $x^2 \cdot x^ez_1^{f_1} z_2^{f_2} \cdots z_k^{f_k}$ or
\item $xy \cdot x^{e_1}y^{e_2}z_1^{f_1} z_2^{f_2} \cdots z_k^{f_k}$,
\end{enumerate}
where
\begin{enumerate}[(I)]
\item $x,y,z_1,z_2,\ldots,z_k$ are distinct variables with $k \geq 0$;
\item $z_1,z_2,\ldots,z_k$ are in alphabetical order;
\item $e \in \{0,1,\ldots,n-2\}$, $e_i \in \{ 0,1,\ldots,n-1\}$, and $f_i \in \{ 1,2,\ldots,n\}$.
\end{enumerate}

\begin{lemma} \label{L: NnIjN2barop can}
The identities~\eqref{id: NnIjN2barop basis} can be used to convert any word of length at least two into a word in canonical form.
\end{lemma}

\begin{proof}
Let $\bw$ be any word of length at least two.
Then $\bw=x^2\bu$ or $\bw = xy\bu$ for some distinct $x,y \in \A$ and $\bu \in \A^*$.
The identity~\eqref{id: NnIjN2barop basis xyzt} can first be used to rearrange the variables of the suffix~$\bu$ until~$\bw$ becomes a word of the form~(CF1) or~(CF2) with~(I) and~(II) satisfied.
The identities~\eqref{id: NnIjN2barop basis xn+1} can then be used to reduce the exponents $e,e_i,f_i$ so that~(III) is satisfied.
\end{proof}

\begin{lemma} \label{L: N2bar word}
The semigroup $\NB$ satisfies an identity $\bu \id \bv$ if and only if the words~$\bu$ and~$\bv$ share the same suffix of length two.
\end{lemma}

\begin{proof}
This is routinely established and its dual result for~$\NBop$ was observed by Lee and Li~\cite[Remark~6.2(i)]{LL11}.
\end{proof}

\begin{proof}[Proof of Proposition~\ref{P: NnIjN2barop}]
It is easily verified, either directly or by Lemmas~\ref{L: word Zn NnI LZI RZI}(ii) and~\ref{L: N2bar word}, that the identities~\eqref{id: NnIjN2barop basis} are satisfied by the semigroup $\N_n^I \times \NBop$.
Hence it remains to show that any identity $\bu \id \bv$ satisfied by the semigroup $\N_n^I \times \NBop$ is deducible from the identities~\eqref{id: NnIjN2barop basis}.
It is easily shown that if either~$\bu$ or~$\bv$ is a single variable, then the identity $\bu \id \bv$ is trivial by Lemma~\ref{L: word Zn NnI LZI RZI}(ii) and so is vacuously deducible from the identities~\eqref{id: NnIjN2barop basis}.
Therefore, assume that~$\bu$ and~$\bv$ are words of length at least two.
By Lemma~\ref{L: NnIjN2barop can}, the identities~\eqref{id: NnIjN2barop basis} can be used to convert~$\bu$ and~$\bv$ into words~$\bu'$ and~$\bv'$ in canonical form.
By Lemma~\ref{L: N2bar word}, the words~$\bu'$ and~$\bv'$ share the same prefix of length two.
Therefore, $\bu'$ and~$\bv'$ are both of the form~(CF1) or both of the form~(CF2).
In any case, it is routinely verified by Lemma~\ref{L: word Zn NnI LZI RZI}(ii) that $\bu'=\bv'$.
Since
\[
\bu \stackrel{\eqref{id: NnIjN2barop basis}}{\id} \bu'=\bv' \stackrel{\eqref{id: NnIjN2barop basis}}{\id} \bv,
\]
the identity $\bu \id \bv$ is deducible from~\eqref{id: NnIjN2barop basis}.
\end{proof}

\begin{corollary} \label{C: N5IjN2bar}
Suppose that $S$ is any finite semigroup that satisfies the identities
\begin{equation}
x^6 \id x^5, \quad x^5yx \id x^4yx, \quad xyzt \id yxzt \label{id: N5IjN2bar basis}
\end{equation}
but violates all of the identities
\begin{equation}
xy \id yx, \quad xyz \id yz. \label{id: N5IjN2bar not ji}
\end{equation}
Then $\lpbr S\rpbr$ is a sub{\pvar} of $\lpbr\N_5^I, \NB\rpbr$ that is not {\ji}.
\end{corollary}

\begin{proof}
By Proposition~\ref{P: NnIjN2barop} with $n=5$, the identities satisfied by $\N_5^I \times \NB$ are axiomatized by~\eqref{id: N5IjN2bar basis}.
Hence the inclusion
\[
\lpbr S\rpbr \subseteq \lpbr\N_5^I, \NB\rpbr = \lpbr\N_5^I\rpbr \vee \lpbr\NB\rpbr
\]
holds.
But the two identities in~\eqref{id: N5IjN2bar not ji} are satisfied by $\N_5^I$ and $\NB$, respectively.
Therefore, $\lpbr S\rpbr \nsubseteq \lpbr\N_5^I\rpbr$ and $\lpbr S\rpbr \nsubseteq \lpbr\NB\rpbr$.
\end{proof}

\subsection{The {\pvar} $\lpbr\N_2^I,\LZBop\rpbr$}

\begin{proposition} \label{P: N2IjLZbarop}
The identities satisfied by the semigroup $\N_2^I \times \LZBop$ are axiomatized by
\begin{equation}
x^3 \id x^2, \quad x^2yx^2 \id xyx, \quad xhytxy \id x^2hyty, \quad xhytyx \id xhy^2tx. \label{id: N2IjLZbarop basis}
\end{equation}
\end{proposition}

\begin{proof}
Let $T_6 = \{a,b,c,d,e,f\}$ be the semigroup given in Table~\ref{Tab: T6}.
The identities satisfied by~$T_6$ are axiomatized by~\eqref{id: N2IjLZbarop basis} \cite[Proposition~26.1]{LZ15}.
It is easily deduced from the proof of this result that any identity violated by~$T_6$ is also violated by one of the following subsemigroups of~$T_6$:
\begin{gather*}
\{a,d,e\} \cong \LZ^I, \quad \{a,b,e\} \cong \N_2^I, \quad \{e,f\} \cong \RZ, \\ \text{and} \quad \langle c,e,f\rangle = \{a,c,d,e,f\} \cong \LZBop.
\end{gather*}
Since $\LZ^I, \RZ \in \lpbr\LZBop\rpbr$, any identity violated by~$T_6$ is violated by $\N_2^I$ or~$\LZBop$.
Therefore, the semigroup $\N_2^I \times \LZBop$ does not generate any proper sub{\pvar} of $\lpbr T_6 \rpbr$, whence $\lpbr\N_2^I \times \LZBop\rpbr=\lpbr T_6 \rpbr$.
\end{proof}

\begin{table}[ht!]
\[\def\arraystretch{1.3}
\begin{array} [c]{r|cccccc}
T_6 \, & \, a & b & c & d & e & f \\ \hline
  a \, & \, a & a & a & a & a & a \\
  b \, & \, a & a & a & a & b & b \\
  c \, & \, c & c & c & c & c & c \\
  d \, & \, d & d & d & d & d & d \\
  e \, & \, a & b & a & d & e & f \\
  f \, & \, a & b & d & d & e & f
\end{array}
\]
\caption{Multiplication table of $T_6$} \label{Tab: T6}
\end{table}

\begin{corollary} \label{C: N2IjLZbarop}
Suppose that~$S$ is any finite semigroup that satisfies the identities~\eqref{id: N2IjLZbarop basis} but violates all of the identities
\begin{equation}
x^2 \id x, \quad xy \id yx. \label{id: N2IjLZbarop not ji}
\end{equation}
Then $\lpbr S\rpbr$ is a sub{\pvar} of $\lpbr\N_2^I, \LZBop\rpbr$ that is not {\ji}.
\end{corollary}

\begin{proof}
By Proposition~\ref{P: N2IjLZbarop}, the inclusion
\[
\lpbr S\rpbr \subseteq \lpbr\N_2^I, \LZBop\rpbr = \lpbr\N_2^I\rpbr \vee \lpbr\LZBop\rpbr
\]
holds.
But the two identities in~\eqref{id: N2IjLZbarop not ji} are satisfied by $\LZBop$ and $\N_2^I$, respectively.
Therefore, $\lpbr S\rpbr \nsubseteq \lpbr\N_2^I\rpbr$ and $\lpbr S\rpbr \nsubseteq \lpbr\LZBop\rpbr$.
\end{proof}

\subsection{The {\pvar} $\lpbr\LZ^I, \el, \elop\rpbr$}

\begin{proposition}[{Zhang and Luo~\cite[Proposition~3.2(3) and Figure~5]{ZL09}}] \label{P: LZIjl3jl3op}
The identities satisfied by the semigroup $\LZ^I \times \el \times \elop$ are axiomatized by
\begin{equation}
x^3 \id x^2, \quad xyx \id x^2y^2, \quad xy^2z \id xyz. \label{id: LZIjl3jl3op basis}
\end{equation}
\end{proposition}

\begin{corollary} \label{C: LZIjl3jl3op}
Suppose that $S$ is any finite semigroup that satisfies the identities~\eqref{id: LZIjl3jl3op basis} but violates all of the identities
\begin{equation}
x^2y \id xy, \quad xy^2 \id xy. \label{id: LZIjI3jI3op not ji}
\end{equation}
Then $\lpbr S\rpbr$ is a sub{\pvar} of $\lpbr\LZ^I, \el, \elop\rpbr$ that is not {\ji}\up.
\end{corollary}

\begin{proof}
By Proposition~\ref{P: LZIjl3jl3op}, the inclusion
\[
\lpbr S\rpbr \subseteq \lpbr\LZ^I, \el, \elop\rpbr = \lpbr\LZ^I, \el\rpbr \vee \lpbr\elop\rpbr
\]
holds.
But the two identities in~\eqref{id: LZIjI3jI3op not ji} are satisfied by $\LZ^I \times \el$ and $\elop$, respectively.
Therefore, $\lpbr S\rpbr \nsubseteq \lpbr\LZ^I, \el\rpbr$ and $\lpbr S\rpbr \nsubseteq \lpbr\elop\rpbr$.
\end{proof}

\subsection{The {\pvar} $\lpbr\Az, \el^I, (\elop)^I \rpbr$}

\begin{proposition}[{Lee~\cite[Proposition~2.8]{Lee08}}] \label{P: A0jB0I}
The identities satisfied by the semigroup $\Az \times \el^I \times (\elop)^I$ are axiomatized by
\begin{equation} \label{id: A0jB0I basis}
\begin{gathered}
x^3 \id x^2, \quad x^2yx^2 \id xyx, \quad xyxy \id yxyx, \\ xyxzx \id xyzx, \quad
xy^2z^2x \id xz^2y^2x.
\end{gathered}
\end{equation}
\end{proposition}

\begin{proof}
The identities satisfied by the semigroup $\Az \times \Bz^I$ are axiomatized by the identities~\eqref{id: A0jB0I basis} \cite[Proposition~2.8]{Lee08}.
Since $\lpbr \Bz^I \rpbr = \lpbr \el^I, (\elop)^I \rpbr$ \cite[Figure~4]{Lee08}, the result follows.
\end{proof}

\begin{corollary} \label{C: A0jB0I}
Suppose that~$S$ is any finite semigroup that satisfies the identities~\eqref{id: A0jB0I basis} but violates all of the identities
\begin{equation}
xyx \id yxy, \quad xyx \id x^2y, \quad xyx \id yx^2. \label{id: A0jB0I not ji}
\end{equation}
Then $\lpbr S\rpbr$ is a sub{\pvar} of $\lpbr\Az, \el^I, (\elop)^I \rpbr$ that is not {\ji}\up.
\end{corollary}

\begin{proof}
By Proposition~\ref{P: A0jB0I}, the inclusion
\[
\lpbr S\rpbr \subseteq \lpbr\Az, \el^I, (\elop)^I \rpbr = \lpbr\Az\rpbr \vee \lpbr\el^I\rpbr \vee \lpbr(\elop)^I\rpbr
\]
holds.
But the three identities in~\eqref{id: A0jB0I not ji} are satisfied by $\Az$, $(\elop)^I$, and $\el^I$, respectively.
Therefore, $\lpbr S\rpbr \nsubseteq \lpbr\Az\rpbr$, $\lpbr S\rpbr \nsubseteq \lpbr\el^I\rpbr$, and $\lpbr S\rpbr \nsubseteq \lpbr(\elop)^I\rpbr$.
\end{proof}

\subsection{The {\pvar} $\lpbr(\N_2^\ba)^I,\LZB\rpbr$} \label{sub: W}

The semigroup $\W = \{a,b,c,d,e\}$ given in Table~\ref{Tab: W} is required in this subsection.

\begin{table}[ht!]
\[\def\arraystretch{1.2}
\begin{array} [c]{c|ccccc}
\W & \, a & b & c & d & e \\ \hline
 a & \, a & a & a & d & e \\
 b & \, a & a & b & d & e \\
 c & \, a & a & c & d & e \\
 d & \, a & a & d & d & e \\
 e & \, a & d & a & d & e
\end{array}
\]
\caption{Multiplication table of $\W$} \label{Tab: W}
\end{table}

\begin{proposition} \label{P: W}
\quad
\begin{enumerate}[\rm(i)]
\item The identities satisfied by the semigroup $\W$ are axiomatized by
\begin{equation}
x^3 \id x^2, \quad xyx \id y^2x. \label{id: W basis}
\end{equation}
\item The sub{\pvar} of $\lpbr\W\rpbr$ defined by the identity
\begin{equation}
x^2y^2z^2 \id x^2yz^2 \label{id: W max}
\end{equation}
is the unique maximal proper sub{\pvar} of $\lpbr\W\rpbr$\up.
\end{enumerate}
\end{proposition}

\begin{remark}
Proposition~\ref{P: W}(i) was first established in Tishchenko and Volkov \cite[Theorem~2]{TV95}.
But since its proof is not long and an understanding of the identities satisfied by~$\W$ is fundamental to the proof of Proposition~\ref{P: W}(ii), it is given below for the sake of completeness.
\end{remark}

In this subsection, a word of the form \[ x_1^{e_1} x_2^{e_2} \cdots x_m^{e_m}, \] where $x_1,x_2,\ldots,x_m$ are distinct variables and $e_1,e_2,\ldots,e_m \in \{1,2\}$, is said to be in \textit{canonical form}.

\begin{remark} \label{R: W can}
It is easily shown that the identities~\eqref{id: W basis} can be used to convert any word into one in canonical form.
\end{remark}

\begin{proof}[Proof of Proposition~\ref{P: W}]
(i) It is routinely shown that the semigroup~$\W$ satisfies the identities~\eqref{id: W basis}.
Conversely, suppose that $\bu \id \bv$ is any identity satisfied by~$\W$.
By Remark~\ref{R: W can}, the identities~\eqref{id: W basis} can be used to convert~$\bu$ and~$\bv$ into some words~$\bu'$ and $\bv'$ in canonical form.
Since the subsemigroup $\{a,c,d\}$ of~$\W$ and the semigroup~$\RZ^I$ are isomorphic, $\fin(\bu') = \fin(\bv')$ by Lemma~\ref{L: word Zn NnI LZI RZI}(iv).
Hence \[ \bu' = x_1^{e_1} x_2^{e_2} \cdots x_m^{e_m} \quad \text{and} \quad \bv' = x_1^{f_1} x_2^{f_2} \cdots x_m^{f_m} \] for some distinct $x_1,x_2,\ldots,x_m \in \A$ and $e_1,e_2,\ldots,e_m, f_1,f_2,\ldots,f_m \in \{1,2\}$.
If $e_k \neq f_k$, then by making the substitution~$\varphi$ given by $x_k \mapsto b$, $x_i \mapsto e$ for any $i<k$, and $x_i \mapsto c$ for any $i > k$, the contradiction $\bu' \varphi \neq \bv' \varphi$ is obtained.
Therefore, $e_i = f_i$ for all~$i$, so that $\bu' = \bv'$.
Consequently, the identity $\bu \id \bv$ is deducible from the identities~\eqref{id: W basis}.

(ii) The semigroup~$\W$ violates the identity~\eqref{id: W max} because $e^2b^2c^2 \neq e^2bc^2$.
Therefore, $\lpbr\W\rpbr \cap \lbrs\eqref{id: W max}\rbrs$ is a proper sub{\pvar} of~$\lpbr\W \rpbr$.
It remains to verify that every proper sub{\pvar}~$\mathbf{V}$ of $\lpbr\W\rpbr$ satisfies the identity~\eqref{id: W max}.
Since $\bV \neq \lpbr\W\rpbr$, there exists an identity $\bu \id \bv$ of~$\bV$ that is violated by $\W$.
Further, since the identities~\eqref{id: W basis} are satisfied by~$\bV$, it follows from Remark~\ref{R: W can} that the words~$\bu$ and~$\bv$ can be chosen to be in canonical form.
There are two cases.

\paragraph{\sc Case~1}

$\fin(\bu) \neq \fin(\bv)$.
Then by Lemma~\ref{L: word Zn NnI LZI RZI}(iv) and the dual of Theorem~\ref{T: LZI}, the {\pvar}~$\bV$ satisfies the {\pid} \[ h^\omega (xh^\omega)^\omega (yh^\omega(xh^\omega)^\omega)^\omega \id h^\omega (xh^\omega(yh^\omega)^\omega)^\omega. \]
Since
\begin{align*}
h^\omega (xh^\omega)^\omega (yh^\omega(xh^\omega)^\omega)^\omega \stackrel{\eqref{id: W basis}}{\id} y^2x^2h^2 \quad
\text{and} \quad h^\omega (xh^\omega(yh^\omega)^\omega)^\omega \stackrel{\eqref{id: W basis}}{\id} x^2y^2h^2,
\end{align*}
the {\pvar}~$\bV$ satisfies the identity
\begin{equation}
x^2y^2h^2 \id y^2x^2h^2. \label{id: W no RZI}
\end{equation}
Since
\[
x^2yz^2 \stackrel{\eqref{id: W basis}}{\id} x^2(x^2y)^2z^2 \stackrel{\eqref{id: W no RZI}}{\id} (x^2y)^2x^2z^2 \stackrel{\eqref{id: W basis}}{\id} y^2x^2z^2 \stackrel{\eqref{id: W no RZI}}{\id} x^2y^2z^2,
\]
the {\pvar}~$\bV$ satisfies the identity~\eqref{id: W max}.

\paragraph{\sc Case~2}

$\fin(\bu) = \fin(\bv)$ and $\bu \neq \bv$.
Then \[ \bu = x_1^{e_1} x_2^{e_2} \cdots x_m^{e_m} \quad \text{and} \quad \bv = x_1^{f_1} x_2^{f_2} \cdots x_m^{f_m} \] for some distinct $x_1,x_2,\ldots,x_m \in \A$ and $e_1,e_2,\ldots,e_m, f_1,f_2,\ldots,f_m \in \{1,2\}$ such that $e_k \neq f_k$ for some~$k$, say $(e_k,f_k) = (2,1)$.
Let~$\varphi$ denote the substitution given by $x_k \mapsto y$, $x_i \mapsto x^2$ for any $i<k$, and $x_i \mapsto z^2$ for any $i>k$.
Then \[ x^2(\bu\varphi)z^2 \stackrel{\eqref{id: W basis}}{\id} x^2y^2z^2 \quad \text{and} \quad x^2(\bv\varphi)z^2 \stackrel{\eqref{id: W basis}}{\id} x^2yz^2, \] so that the {\pvar}~$\bV$ satisfies the identity~\eqref{id: W max}.
\end{proof}

\begin{proposition} \label{P: W not ji}
The {\pvar} $\lpbr\W\rpbr$ is {\sji} but not {\ji}\up.
\end{proposition}

\begin{proof}
The {\pvar} $\lpbr\W\rpbr$ is {\sji} by Proposition~\ref{P: W}(ii).
To show that $\lpbr\W\rpbr$ is not {\ji}, the semigroups $\N_2^\ba = \{ \ov\ez, \ea, \ov\ea, \ov I \}$ and $\LZB = \{ \ee, \ef, \ov\ee, \ov\ef, \ov I \}$ from Subsections~\ref{sub: nilpotent} and~\ref{sub: bands} are required.
Let $T_{11}$ denote the subsemigroup of $(\N_2^\ba)^I \times \LZB$ generated by $\eE=(\ei,\ee)$, $\eX = (\ea,\ef)$, and $\eY = (\ov I, \ov I)$.
It is routinely checked that  \[ \eE^2 = \eE, \quad \eX\eE=\eX, \quad \eX^3=\eX^2, \quad \text{and} \quad \eE\eY=\eX\eY=\eY^2=\eY, \] so that the semigroup $T_{11}$ consists of the following 11~elements:
\begin{align*}
a & =\eE, &b&=\eE\eX, &c&=\eE\eX^2, &d&=\eX, &e&=\eX^2, &f&=\eY, \\ g & =\eY\eX, &h&=\eY\eX^2, &i&=\eY\eE, &j&=\eY\eE\eX, &k&=\eY\eE\eX^2; &&
\end{align*}
see Table~\ref{Tab: T11}.
Identifying the elements $\{b,c,e,h,i,j,k\}$ in $T_{11}$ results in a semigroup that is isomorphic to~$\W$.
Therefore, \[\lpbr \W \rpbr \subseteq \lpbr T_{11} \rpbr \subseteq \lpbr(\N_2^\ba)^I\rpbr \vee \lpbr\LZB\rpbr.\]
But the semigroup~$\W$ violates the identities
\begin{equation}
xyx^2 \id yx^2, \quad x^2 \id x, \label{id: W not ji}
\end{equation}
and these identities are satisfied by $(\N_2^\ba)^I$ and $\LZB$, respectively.
Consequently, $\lpbr \W \rpbr \nsubseteq \lpbr(\N_2^\ba)^I\rpbr$ and $\lpbr \W \rpbr \nsubseteq \lpbr\LZB\rpbr$.
\end{proof}

\begin{table}[ht!]
\[\def\arraystretch{1.2}
\begin{array} [c]{c|ccccccccccc}
T_{11} & \, a & b & c & d & e & f & g & h & i & j & k \\ \hline
     a & \, a & b & c & b & c & f & g & h & i & j & k \\
     b & \, b & c & c & c & c & f & g & h & i & j & k \\
     c & \, c & c & c & c & c & f & g & h & i & j & k \\
     d & \, d & e & e & e & e & f & g & h & i & j & k \\
     e & \, e & e & e & e & e & f & g & h & i & j & k \\
     f & \, i & j & k & g & h & f & g & h & i & j & k \\
     g & \, g & h & h & h & h & f & g & h & i & j & k \\
     h & \, h & h & h & h & h & f & g & h & i & j & k \\
     i & \, i & j & k & j & k & f & g & h & i & j & k \\
     j & \, j & k & k & k & k & f & g & h & i & j & k \\
     k & \, k & k & k & k & k & f & g & h & i & j & k
\end{array}
\]
\caption{Multiplication table of $T_{11}$} \label{Tab: T11}
\end{table}

\begin{corollary} \label{C: W}
Suppose that~$S$ is any finite semigroup that satisfies the identities~\eqref{id: W basis} but violates all the identities in~\eqref{id: W not ji}\up.
Then $\lpbr S\rpbr$ is a sub\-{\pvar} of $\lpbr(\N_2^\ba)^I,\LZB\rpbr$ that is not {\ji}\up.
\end{corollary}

\begin{proof}
The inclusions \[\lpbr S \rpbr \subseteq \lpbr \W \rpbr \subseteq \lpbr(\N_2^\ba)^I \rpbr \vee \lpbr\LZB\rpbr\] hold by Proposition~\ref{P: W}(i) and the proof of Proposition~\ref{P: W not ji}.
But the identities in~\eqref{id: W not ji} are satisfied by $(\N_2^\ba)^I$ and $\LZB$, respectively.
Therefore, $\lpbr S \rpbr \nsubseteq \lpbr(\N_2^\ba)^I\rpbr$ and $\lpbr S \rpbr \nsubseteq \lpbr\LZB\rpbr$.
\end{proof}

\section{Pseudovarieties generated by a semigroup of order up to five} \label{sec: proof}

\begin{theorem} \label{T: ji}
Let~$S$ be any nontrivial semigroup of order up to five\up.
Suppose that the {\pvar} $\lpbr S \rpbr$ is {\ji}\up.
Then $\lpbr S \rpbr$ coincides with one of the following $30$~{\pvars}\up:
\begin{align*}
& \lpbr\Z_2\rpbr,   && \lpbr\Z_3\rpbr,   && \lpbr\Z_4\rpbr,   && \lpbr\Z_5\rpbr,    && \lpbr\ZB\rpbr,       && \lpbr\ZBop\rpbr, \\
& \lpbr\N_2\rpbr,   && \lpbr\N_3\rpbr,   && \lpbr\N_4\rpbr,   && \lpbr\N_5\rpbr,    && \lpbr\NB\rpbr,       && \lpbr\NBop\rpbr, \\
& \lpbr\N_1^I\rpbr, && \lpbr\N_2^I\rpbr, && \lpbr\N_3^I\rpbr, && \lpbr\N_4^I\rpbr,  && \lpbr\NBI\rpbr,      && \lpbr\NBIop\rpbr, \\
& \lpbr\LZ\rpbr,    && \lpbr\LZ^I\rpbr,  && \lpbr\LZB\rpbr,   && \lpbr\LZ^\op\rpbr, && \lpbr(\LZ^I)^\op\rpbr, && \lpbr\LZBop\rpbr, \\
& \lpbr\Az\rpbr,    & & \lpbr\Az^I\rpbr,  & & \lpbr\At\rpbr,   & & \lpbr\Bt\rpbr,   & & \lpbr\elB\rpbr,   & & \lpbr\elBop\rpbr.
\end{align*}
\end{theorem}

\begin{proof}
The 30~{\pvars} are {\ji} by results in Sec.~\ref{sec: ji}; see Table~\ref{Tab: results}.
Up to isomorphism and anti-isomorphism, there exist $1308$~nontrivial semigroups of order up to five.
With the aid of a computer, it is routinely determined, using the sufficient conditions given in Subsections~\ref{sub: ji} and~\ref{sub: non-ji} below, which of these semigroups generate {\ji} {\pvars}.
Specifically, by Conditions~A1--A23 and their dual conditions, $241$~of the $1308$~semigroups generate the~30 {\ji} {\pvars}, while by Conditions~B1--B13 and their dual conditions, the remaining $1067$~semigroups generate {\pvars} that are not {\ji}; see Table~\ref{Tab: number}.
The proof of Theorem~\ref{T: ji} is thus complete.
\end{proof}

\newcommand{\spaceone}{-0.04in}
\begin{table}[ht!]
\[\def\arraystretch{1.8}
\begin{tabular}[m]{clllc} \hline
& Pseudovarieties                                                                & & Join {\irr} by          & \\ \hline
& $\lpbr\Z_2\rpbr$, $\lpbr\Z_3\rpbr$, $\lpbr\Z_4\rpbr$, $\lpbr\Z_5\rpbr$         & & Theorem~\ref{T: Zn}     & \\[\spaceone]
& $\lpbr\ZB\rpbr$, $\lpbr\ZBop\rpbr$                                             & & Theorem~\ref{T: Z2bar}  & \\[\spaceone]
& $\lpbr\N_2\rpbr$, $\lpbr\N_3\rpbr$, $\lpbr\N_4\rpbr$, $\lpbr\N_5\rpbr$         & & Theorem~\ref{T: Nn}     & \\[\spaceone]
& $\lpbr\N_1^I\rpbr$, $\lpbr\N_2^I\rpbr$, $\lpbr\N_3^I\rpbr$, $\lpbr\N_4^I\rpbr$ & & Theorem~\ref{T: NnI}    & \\[\spaceone]
& $\lpbr\NB\rpbr$, $\lpbr\NBop\rpbr$                                             & & Theorem~\ref{T: N2bar}  & \\[\spaceone]
& $\lpbr\NBI\rpbr$, $\lpbr\NBIop\rpbr$                                           & & Theorem~\ref{T: N2barI} & \\[\spaceone]
& $\lpbr\LZ\rpbr$, $\lpbr\LZ^\op\rpbr$                                           & & Theorem~\ref{T: LZ}     & \\[\spaceone]
& $\lpbr\LZ^I\rpbr$, $\lpbr(\LZ^I)^\op\rpbr$                                     & & Theorem~\ref{T: LZI}    & \\[\spaceone]
& $\lpbr\LZB\rpbr$, $\lpbr\LZBop\rpbr$                                           & & Theorem~\ref{T: LZbar}  & \\[\spaceone]
& $\lpbr\Az\rpbr$                                                                & & Theorem~\ref{T: A0}     & \\[\spaceone]
& $\lpbr\Az^I\rpbr$                                                              & & Theorem~\ref{T: A0I}    & \\[\spaceone]
& $\lpbr\At\rpbr$                                                                & & Theorem~\ref{T: A2}     & \\[\spaceone]
& $\lpbr\Bt\rpbr$                                                                & & Theorem~\ref{T: B2}     & \\[\spaceone]
& $\lpbr\elB\rpbr$, $\lpbr\elBop\rpbr$                                           & & Theorem~\ref{T: el3}    & \\[0.04in] \hline
\end{tabular}
\]
\caption{Results for {\jirr} of {\pvars} in Theorem~\ref{T: ji}} \label{Tab: results}
\end{table}

\begin{table}[ht!]
\[\def\arraystretch{1.5}
\begin{tabular}[m]{|l|cccc|c|} \hline
                        & $n=2$              & $n=3$               & $n=4$                & $n=5$                 & $2 \leq n \leq 5$     \\ \hline
Number of {\ji}         & \multirow{2}{*}{4} & \multirow{2}{*}{8}  & \multirow{2}{*}{33}  & \multirow{2}{*}{196}  & \multirow{2}{*}{241} \\[-0.08in]
semigroups of order~$n$ &                    &                     &                      &                       &                       \\ \hline
Number of non-{\ji}     & \multirow{2}{*}{0} & \multirow{2}{*}{10} & \multirow{2}{*}{93}  & \multirow{2}{*}{964}  & \multirow{2}{*}{1067} \\[-0.08in]
semigroups of order~$n$ &                    &                     &                      &                       &                       \\ \hline
Number of         & \multirow{2}{*}{4} & \multirow{2}{*}{18} & \multirow{2}{*}{126} & \multirow{2}{*}{1160} & \multirow{2}{*}{1308} \\[-0.08in]
semigroups of order~$n$ &                    &                     &                      &                       &                       \\ \hline
\end{tabular}
\]
\caption{Number of {\ji} semigroups of order up to five} \label{Tab: number}
\end{table}

\subsection{Conditions sufficient for {\jirr}} \label{sub: ji}

The following conditions and their dual conditions are sufficient for a finite semigroup~$S$ to generate a {\ji} {\pvar} in Theorem~\ref{T: ji}.

\begin{conA}[Proposition~\ref{P: Zn}] The equality ${ \lpbr S\rpbr = \lpbr\Z_2\rpbr }$ holds if
\begin{enumerate}[1.]
\item[$\bullet$] $S \models \{xy \id yx, \, x^2y \id y\}$\up,
\item[$\bullet$] $S \not\models x \id y$\up.
\end{enumerate}
\end{conA}

\begin{conA}[Proposition~\ref{P: Zn}] The equality ${ \lpbr S\rpbr = \lpbr\Z_3\rpbr }$ holds if
\begin{enumerate}[1.]
\item[$\bullet$] $S \models \{xy \id yx, \, x^3y \id y\}$\up,
\item[$\bullet$] $S \not\models x \id y$\up.
\end{enumerate}
\end{conA}

\begin{conA}[Proposition~\ref{P: Zn}] The equality ${ \lpbr S\rpbr = \lpbr\Z_4\rpbr }$ holds if
\begin{enumerate}[1.]
\item[$\bullet$] $S \models \{xy \id yx, \, x^4y \id y\}$\up,
\item[$\bullet$] $S \not\models x^3 \id x$\up.
\end{enumerate}
\end{conA}

\begin{conA}[Proposition~\ref{P: Zn}] The equality ${ \lpbr S\rpbr = \lpbr\Z_5\rpbr }$ holds if
\begin{enumerate}[1.]
\item[$\bullet$] $S \models \{xy \id yx, \, x^5y \id y\}$\up,
\item[$\bullet$] $S \not\models x \id y$\up.
\end{enumerate}
\end{conA}

\begin{conA}[Proposition~\ref{P: Z2bar}] The equality ${ \lpbr S\rpbr = \lpbr\ZB\rpbr }$ holds if
\begin{enumerate}[1.]
\item[$\bullet$] $S \models \{x^3 \id x, \, xyxy \id yx^2y\}$\up,
\item[$\bullet$] $S \not\models xyx \id yx^2$\up.
\end{enumerate}
\end{conA}

\begin{conA}[Proposition~\ref{P: Nn}] The equality ${ \lpbr S\rpbr = \lpbr\N_2\rpbr }$ holds if
\begin{enumerate}[1.]
\item[$\bullet$] $S \models x^2 \id y_1y_2$\up,
\item[$\bullet$] $S \not\models x^2 \id x$\up.
\end{enumerate}
\end{conA}

\begin{conA}[Proposition~\ref{P: Nn}] The equality ${ \lpbr S\rpbr = \lpbr\N_3\rpbr }$ holds if
\begin{enumerate}[1.]
\item[$\bullet$] $S \models \{xy \id yx, \, x^3 \id y_1y_2y_3\}$\up,
\item[$\bullet$] $S \not\models x^3 \id x^2$\up.
\end{enumerate}
\end{conA}

\begin{conA}[Proposition~\ref{P: Nn}] The equality ${ \lpbr S\rpbr = \lpbr\N_4\rpbr }$ holds if
\begin{enumerate}[1.]
\item[$\bullet$] $S \models \{xy \id yx, \, x^2y \id xy^2, \, x^4 \id y_1y_2y_3y_4\}$\up,
\item[$\bullet$] $S \not\models x^4 \id x^3$\up.
\end{enumerate}
\end{conA}

\begin{conA}[Proposition~\ref{P: Nn}] The equality ${ \lpbr S\rpbr = \lpbr\N_5\rpbr }$ holds if
\begin{enumerate}[1.]
\item[$\bullet$] $S \models \{xy \id yx, \, x^2yz \id xy^2z, \, x^5 \id y_1y_2y_3y_4y_5\}$\up,
\item[$\bullet$] $S \not\models x^5 \id x^4$\up.
\end{enumerate}
\end{conA}

\begin{conA}[Proposition~\ref{P: NnI}] The equality ${ \lpbr S\rpbr = \lpbr\N_1^I\rpbr }$ holds if
\begin{enumerate}[1.]
\item[$\bullet$] $S \models \{x^2 \id x, \, xy \id yx\}$\up,
\item[$\bullet$] $S \not\models x \id y$\up.
\end{enumerate}
\end{conA}

\begin{conA}[Proposition~\ref{P: NnI}] The equality ${ \lpbr S\rpbr = \lpbr\N_2^I\rpbr }$ holds if
\begin{enumerate}[1.]
\item[$\bullet$] $S \models \{x^3 \id x^2, \, xy \id yx\}$\up,
\item[$\bullet$] $S \not\models x^2y \id xy^2$\up.
\end{enumerate}
\end{conA}

\begin{conA}[Proposition~\ref{P: NnI}] The equality ${ \lpbr S\rpbr = \lpbr\N_3^I\rpbr }$ holds if
\begin{enumerate}[1.]
\item[$\bullet$] $S \models \{x^4 \id x^3, \, xy \id yx\}$\up,
\item[$\bullet$] $S \not\models x^3y^2 \id x^2y^3$\up.
\end{enumerate}
\end{conA}

\begin{conA}[Proposition~\ref{P: NnI}] The equality ${ \lpbr S\rpbr = \lpbr\N_4^I\rpbr }$ holds if
\begin{enumerate}[1.]
\item[$\bullet$] $S \models \{x^5 \id x^4, \, xy \id yx\}$\up,
\item[$\bullet$] $S \not\models x^4y^3 \id x^3y^4$\up.
\end{enumerate}
\end{conA}

\begin{conA}[Proposition~\ref{P: N2bar}] The equality ${ \lpbr S\rpbr = \lpbr\NB\rpbr }$ holds if
\begin{enumerate}[1.]
\item[$\bullet$] $S \models xyz \id yz$\up,
\item[$\bullet$] $S \not\models xy \id y^2$\up.
\end{enumerate}
\end{conA}

\begin{conA}[Proposition~\ref{P: N2barI}] The equality ${ \lpbr S\rpbr = \lpbr\NBI\rpbr }$ holds if
\begin{enumerate}[1.]
\item[$\bullet$] $S \models \left\{ \begin{array}[c]{l} x^3 \id x^2, \, x^2hx \id xhx, \, xhx^2 \id hx^2, \, xyxy \id yx^2y, \\ xyhxy \id yxhxy, \, xyxty \id yx^2ty, \, xyhxty \id yxhxty \end{array} \right\}$\up,
\item[$\bullet$] $S \not\models xyxyh^2 \id x^2y^2h^2$\up.
\end{enumerate}
\end{conA}

\begin{conA}[Proposition~\ref{P: LZ}] The equality ${ \lpbr S\rpbr = \lpbr\LZ\rpbr }$ holds if
\begin{enumerate}[1.]
\item[$\bullet$] $S \models xy \id x$\up,
\item[$\bullet$] $S \not\models x \id y$\up.
\end{enumerate}
\end{conA}

\begin{conA}[Proposition~\ref{P: LZI}] The equality ${ \lpbr S\rpbr = \lpbr\LZ^I\rpbr }$ holds if
\begin{enumerate}[1.]
\item[$\bullet$] $S \models \{x^2 \id x, \, xyx \id xy\}$\up,
\item[$\bullet$] $S \not\models xyz \id xzy$\up.
\end{enumerate}
\end{conA}

\begin{conA}[Proposition~\ref{P: LZbar}] The equality ${ \lpbr S\rpbr = \lpbr\LZB\rpbr }$ holds if
\begin{enumerate}[1.]
\item[$\bullet$] $S \models \{x^2 \id x, \, xyz \id xzxyz\}$\up,
\item[$\bullet$] $S \not\models xyz \id xzyz$\up.
\end{enumerate}
\end{conA}

\begin{conA}[Proposition~\ref{P: A0}] The equality ${ \lpbr S\rpbr = \lpbr\Az\rpbr }$ holds if
\begin{enumerate}[1.]
\item[$\bullet$] $S \models \{x^3 \id x^2, \, x^2yx^2 \id yxy \}$\up,
\item[$\bullet$] $S \not\models x^2y^2 \id y^2x^2$\up.
\end{enumerate}
\end{conA}

\begin{conA}[Proposition~\ref{P: A0I}] The equality ${ \lpbr S\rpbr = \lpbr\Az^I\rpbr }$ holds if
\begin{enumerate}[1.]
\item[$\bullet$] $S \models \{x^3 \id x^2, \, x^2yx^2 \id xyx, \, xyxy \id yxyx, \, xyxzx \id xyzx\}$\up,
\item[$\bullet$] $S \not\models hx^2y^2h \id hy^2x^2h$\up.
\end{enumerate}
\end{conA}

\begin{conA}[Proposition~\ref{P: A2}] The equality ${ \lpbr S\rpbr = \lpbr\At\rpbr }$ holds if
\begin{enumerate}[1.]
\item[$\bullet$] $S \models \{x^3 \id x^2, \, xyxyx \id xyx, \, xyxzx \id xzxyx\}$\up,
\item[$\bullet$] $S \not\models x^2y^2x^2 \id x^2yx^2$\up.
\end{enumerate}
\end{conA}

\begin{conA}[Proposition~\ref{P: B2}] The equality ${ \lpbr S\rpbr = \lpbr\Bt\rpbr }$ holds if
\begin{enumerate}[1.]
\item[$\bullet$] $S \models \{x^3 \id x^2, \, xyxyx \id xyx, \, x^2y^2 \id y^2x^2\}$\up,
\item[$\bullet$] $S \not\models xy^2x \id xyx$\up.
\end{enumerate}
\end{conA}

\begin{conA}[Proposition~\ref{P: el3}] The equality ${ \lpbr S\rpbr = \lpbr\elB\rpbr }$ holds if
\begin{enumerate}[1.]
\item[$\bullet$] $S \models \{x^2y \id xy, \, xyz \id yxyz\}$\up,
\item[$\bullet$] $S \not\models xyzx \id yxzx$\up.
\end{enumerate}
\end{conA}

\subsection{Conditions sufficient for non-{\jirr}} \label{sub: non-ji}

The following conditions and their dual conditions are sufficient for a finite semigroup~$S$ to generate a {\pvar} that is not {\ji}.

\begin{conB}[Corollary~\ref{C: M4v3}]
A {\pvar} $\lpbr S\rpbr$ is a non-{\ji} sub{\pvar} of ${ \lpbr\Z_3,\Z_4,\ZB,\ZBop,\N_3^I\rpbr }$ if
\begin{enumerate}[1.]
\item[$\bullet$] $S \models \left\{ \begin{array}[c]{l} x^{15} \id x^3, \, x^{14}hx \id x^2hx, \, x^{13}hx^2 \id xhx^2, \, x^{13} hxtx \id xhxtx, \\ x^3hx \id xhx^3, \, xhx^2tx \id x^3htx, \, xhx^2y^2ty \id xhy^2x^2ty, \\ xhykxytxdy \id xhykyxtxdy, \, xhykxytydx \id xhykyxtydx \end{array} \right\}$\up,
\item[$\bullet$] $S \not\models x^3 \id x$\up, \quad $S \not\models xy \id yx$\up.
\end{enumerate}
\end{conB}

\begin{conB}[Corollary~\ref{C: M4v2}]
A {\pvar} $\lpbr S\rpbr$ is a non-{\ji} sub{\pvar} of ${ \lpbr\Z_3,\Z_4,\ZB,\ZBop,\N_3^I\rpbr }$ if
\begin{enumerate}[1.]
\item[$\bullet$] $S \models \left\{ \begin{array}[c]{l} x^{15} \id x^3, \, x^{14}hx \id x^2hx, \, x^{13}hx^2 \id xhx^2, \, x^{13} hxtx \id xhxtx, \\ x^3hx \id xhx^3, \, xhx^2tx \id x^3htx, \, xhx^2y^2ty \id xhy^2x^2ty, \\ xhykxytxdy \id xhykyxtxdy, \, xhykxytydx \id xhykyxtydx \end{array} \right\}$\up,
\item[$\bullet$] $S \not\models xy \id yx$\up, \quad $S \not\models xyx^2 \id xy$\up, \quad $S \not\models x^2yx \id yx$\up.
\end{enumerate}
\end{conB}

\begin{conB}[Corollary~\ref{C: M4v1}]
A {\pvar} $\lpbr S\rpbr$ is a non-{\ji} sub{\pvar} of ${ \lpbr\Z_3,\Z_4,\ZB,\ZBop,\N_3^I \rpbr }$ if
\begin{enumerate}[1.]
\item[$\bullet$] $S \models \left\{ \begin{array}[c]{l} x^{15} \id x^3, \, x^{14}hx \id x^2hx, \, x^{13}hx^2 \id xhx^2, \, x^{13} hxtx \id xhxtx, \\ x^3hx \id xhx^3, \, xhx^2tx \id x^3htx, \, xhx^2y^2ty \id xhy^2x^2ty, \\ xhykxytxdy \id xhykyxtxdy, \, xhykxytydx \id xhykyxtydx \end{array} \right\}$\up,
\item[$\bullet$] $S \not\models x^4 \id x^3$\up, \quad $S \not\models x^4y \id y$\up, \quad $S \not\models x^3y \id y$\up, \quad $S \not\models xyx^2 \id xy$\up, \quad $S \not\models x^2yx \id yx$\up.
\end{enumerate}
\end{conB}

\begin{conB}[Corollary~\ref{C: Z6jN5IjLZIjRZIjA0Iv1}]
A {\pvar} $\lpbr S\rpbr$ is a non-{\ji} sub{\pvar} of ${ \lpbr \Z_6,\N_5^I,\LZ^I,\RZ^I,\Az^I \rpbr }$ if
\begin{enumerate}[1.]
\item[$\bullet$] $S \models \{x^{11} \id x^5, \, x^{10}yx \id x^4yx, \, x^2yx \id xyx^2, \, xyxzx \id x^2yzx\}$\up,
\item[$\bullet$] $S \not\models x^2 \id x$\up, \quad $S \not\models xyxy \id yxyx$\up.
\end{enumerate}
\end{conB}

\begin{conB}[Corollary~\ref{C: Z6jN5IjLZIjRZIjA0Iv2}]
A {\pvar} $\lpbr S\rpbr$ is a non-{\ji} sub{\pvar} of ${ \lpbr \Z_6,\N_5^I,\LZ^I,\RZ^I,\Az^I \rpbr }$ if
\begin{enumerate}[1.]
\item[$\bullet$] $S \models \{x^{11} \id x^5, \, x^{10}yx \id x^4yx, \, x^2yx \id xyx^2, \, xyxzx \id x^2yzx\}$\up,
\item[$\bullet$] $S \not\models x^6 \id x^5$\up, \quad $S \not\models x^6y \id y$\up.
\end{enumerate}
\end{conB}

\begin{conB}[Corollary~\ref{C: noncomm nil}]
A {\pvar} $\lpbr S\rpbr$ is a non-{\ji} {\pvar} of nilpotent semigroups if
\begin{enumerate}[1.]
\item[$\bullet$] $S \models x^6 \id y_1y_2y_3y_4y_5y_6$\up,
\item[$\bullet$] $S \not\models xy \id yx$\up.
\end{enumerate}
\end{conB}

\begin{conB}[Corollary~\ref{C: N4jN2I}]
A {\pvar} $\lpbr S\rpbr$ is a non-{\ji} sub{\pvar} of ${ \lpbr \N_4,\N_2^I \rpbr }$ if
\begin{enumerate}[1.]
\item[$\bullet$] $S \models \{xy \id yx, \, x^3y_1y_2 \id x^2y_1y_2\}$\up,
\item[$\bullet$] $S \not\models x^3 \id x^2$\up, \quad $S \not\models x^2y \id xy^2$\up.
\end{enumerate}
\end{conB}

\begin{conB}[Corollary~\ref{C: N5jN1I}]
A {\pvar} $\lpbr S\rpbr$ is a non-{\ji} sub{\pvar} of ${ \lpbr\N_5,\N_1^I \rpbr }$ if
\begin{enumerate}[1.]
\item[$\bullet$] $S \models \{xy \id yx, \, x^2yz \id xy^2z, \, x^2y_1y_2y_3y_4 \id xy_1y_2y_3y_4 \}$\up,
\item[$\bullet$] $S \not\models x^2 \id x$\up, \quad $S \not\models x^5 \id y^5$\up.
\end{enumerate}
\end{conB}

\begin{conB}[Corollary~\ref{C: N5IjN2bar}]
A {\pvar} $\lpbr S\rpbr$ is a non-{\ji} sub{\pvar} of ${ \lpbr\N_5^I,\NB \rpbr }$ if
\begin{enumerate}[1.]
\item[$\bullet$] $S \models \{x^6 \id x^5, \, x^5yx \id x^4yx, \, xyzt \id yxzt \}$\up,
\item[$\bullet$] $S \not\models xy \id yx$\up, \quad $S \not\models xyz \id yz$\up.
\end{enumerate}
\end{conB}

\begin{conB}[Corollary~\ref{C: N2IjLZbarop}]
A {\pvar} $\lpbr S\rpbr$ is a non-{\ji} sub{\pvar} of ${ \lpbr \N_2^I,\LZBop \rpbr }$ if
\begin{enumerate}[1.]
\item[$\bullet$] $S \models \{x^3 \id x^2, \, x^2yx^2 \id xyx, \, xhytxy \id x^2hyty, \, xhytyx \id xhy^2tx \}$\up,
\item[$\bullet$] $S \not\models x^2 \id x$\up, \quad $S \not\models xy \id yx$\up.
\end{enumerate}
\end{conB}

\begin{conB}[Corollary~\ref{C: LZIjl3jl3op}]
A {\pvar} $\lpbr S\rpbr$ is a non-{\ji} sub{\pvar} of ${ \lpbr\LZ^I,\el,\elop \rpbr }$ if
\begin{enumerate}[1.]
\item[$\bullet$] $S \models \{x^3 \id x^2, \, xyx \id x^2y^2, \, xy^2z \id xyz\}$\up,
\item[$\bullet$] $S \not\models x^2y \id xy$\up, \quad $S \not\models xy^2 \id xy$\up.
\end{enumerate}
\end{conB}

\begin{conB}[Corollary~\ref{C: A0jB0I}]
A {\pvar} $\lpbr S\rpbr$ is a non-{\ji} sub{\pvar} of ${ \lpbr\Az, \el^I,(\elop)^I \rpbr }$ if
\begin{enumerate}[1.]
\item[$\bullet$] $S \models \left\{ \begin{array}[c]{l} x^3 \id x^2, \, x^2yx^2 \id xyx, \, xyxy \id yxyx, \\ xyxzx \id xyzx, \, xy^2z^2x \id xz^2y^2x \end{array} \right\}$\up,
\item[$\bullet$] $S \not\models xyx \id yxy$\up, \quad $S \not\models xyx \id x^2y$\up, \quad $S \not\models xyx \id yx^2$\up.
\end{enumerate}
\end{conB}

\begin{conB}[Corollary~\ref{C: W}]
A {\pvar} $\lpbr S\rpbr$ is a non-{\ji} sub{\pvar} of $\lpbr(\N_2^\ba)^I,\LZB\rpbr$ if
\begin{enumerate}[1.]
\item[$\bullet$] $S \models \{ x^3 \id x^2, \, xyx \id y^2x \}$\up,
\item[$\bullet$] $S \not\models x^2 \id x$\up, \quad $S \not\models xyx^2 \id yx^2$\up.
\end{enumerate}
\end{conB}

\section*{Acknowledgments}

We are very grateful to the following colleagues:
the anonymous reviewer, for a careful reading of the entire paper and a number of useful suggestions;
Norman Reilly, for a helpful discussion on {\sji} {\pvars} of bands;
George Bergman, for allowing us to include Theorem~\ref{T: Bergman} and its proof in the paper;
and Wendy Wong, for checking the sufficient conditions in Sec.~\ref{sec: proof}, with a computer, against all semigroups of order up to five.
We also thank Keith Kearnes and the reviewer for information on Proposition~\ref{P: D4 Q8}.

John Rhodes was supported by Simons Foundation Collaboration Grants for Mathematicians \#313548.
Benjamin Steinberg was supported by Simons Foundation \#245268, United States--Israel Binational Science Foundation \#2012080, and NSA MSP \#H98230-16-1-0047.

\end{document}